\documentclass[11pt,reqno]{amsart}
\usepackage{fullpage,times,graphicx,amssymb,amsmath,psfrag,xcolor,natbib}
\usepackage{booktabs} 
\DeclareSymbolFont{extraup}{U}{zavm}{m}{n}
\DeclareMathSymbol{\varheartsuit}{\mathalpha}{extraup}{86}
\DeclareMathSymbol{\Diamondblack}{\mathalpha}{extraup}{87}
\usepackage{blkarray}
\usepackage{subcaption}
\usepackage{multirow}
\newcommand{\beginsupplement}{%
	\setcounter{table}{0}
	\renewcommand{\thetable}{S\arabic{table}}%
	\setcounter{figure}{0}
	\renewcommand{\thefigure}{S\arabic{figure}}%
	\setcounter{algorithm}{0}
	\renewcommand\thealgorithm{S\arabic{algorithm}}
}
\usepackage[colorlinks,citecolor=bbluegray,linkcolor=ddarkbrown,urlcolor=blue,breaklinks]{hyperref}
\usepackage{algorithmic,algorithm}
\renewcommand{\algorithmicrequire}{\textbf{Input:}} 
\renewcommand{\algorithmicensure}{\textbf{Output:}}

\newcommand{\pushpagenumber}{%
	\setlength{\footskip}{35pt}
}
\usepackage[off]{auto-pst-pdf} 

\newtheorem{theorem}{Theorem}[section]
\newtheorem{proposition}[theorem]{Proposition}
\newtheorem{definition}[theorem]{Definition}
\newtheorem{lemma}[theorem]{Lemma}

\renewenvironment{proof}{\textbf{Proof.}}{\QED\bigskip}

\usepackage{algorithmic,algorithm}
\renewcommand{\algorithmicrequire}{\textbf{Input:}} 
\renewcommand{\algorithmicensure}{\textbf{Output:}}

\definecolor{ddarkbrown}{rgb}{0.5,0.2,0.05} \definecolor{bbluegray}{rgb}{0.05,0,0.5}

\newcommand{\BEAS}{\begin{eqnarray*}}
\newcommand{\EEAS}{\end{eqnarray*}}
\newcommand{\BEA}{\begin{eqnarray}}
\newcommand{\EEA}{\end{eqnarray}}
\newcommand{\BEQ}{\begin{equation}}
\newcommand{\EEQ}{\end{equation}}
\newcommand{\BIT}{\begin{itemize}}
\newcommand{\EIT}{\end{itemize}}
\newcommand{\BNUM}{\begin{enumerate}}
\newcommand{\ENUM}{\end{enumerate}}

\newcommand{\BA}{\begin{array}}
\newcommand{\EA}{\end{array}}

\newcommand{\eg}{{\it e.g.}}
\newcommand{\ie}{{\it i.e.}}
\newcommand{\ones}{\mathbf 1}

\newcommand{\reals}{{\mathbb R}}

\newcommand{\posints}{\mbox{\bf N}}

\newcommand{\symm}{{\mbox{\bf S}}}  

\newcommand{\diag}{\mathop{\bf diag}}

\newcommand{\idm}{\mathbf{I}}

\newcommand{\QED}{~~\rule[-1pt]{6pt}{6pt}}

\newcommand{\argmin}{\mathop{\rm argmin}}

\newcommand{\argmax}{\mathop{\rm argmax}}

\providecommand{\cB}{\mathcal{B}}
\providecommand{\cP}{\mathcal{P}}
\providecommand{\cR}{\mathcal{S}_{R}}
\providecommand{\cZ}{\mathcal{Z}}
\providecommand{\cPH}{\mathcal{PH}}

\providecommand{\cM}{\mathcal{M}}
\providecommand{\PAP}{\Pi A \Pi^T}
\providecommand{\cH}{\mathcal{H}}

\providecommand{\st}{\text{such that}}
\providecommand{\find}{\text{find}}

\newcommand{\argsort}{\mathop{\rm argsort}}

\let \oldsection \section
\renewcommand{\section}{\vspace{3ex plus 1ex}\oldsection}

\begin{document}
\title{Robust Seriation and Applications to Cancer Genomics}

\author{Antoine Recanati}
\address{CNRS \& D.I., UMR 8548,\vskip 0ex
\'Ecole Normale Sup\'erieure, Paris, France.}
\email{antoine.recanati@inria.fr}
\author{Nicolas Servant}
\address{Institut Curie, PSL Research University, INSERM, U900, F-75005 Paris, France\\
	MINES ParisTech, PSL Research University, CBIO - Centre for Computational Biology, F-75006 Paris, France}
\email{Nicolas.Servant@curie.fr}
\author{Jean-Philippe Vert}
\address{MINES ParisTech, PSL Research University, CBIO - Centre for Computational Biology, F-75006 Paris, France\\
	Institut Curie, PSL Research University, INSERM, U900, F-75005 Paris, France\\
	Ecole Normale Supérieure, Department of Mathematics and Applications, CNRS, PSL Research University, F-75005 Paris, France}
\email{Jean-Philippe.Vert@Mines-ParisTech.fr}
\author{Alexandre d'Aspremont}
\address{CNRS \& D.I., UMR 8548,\vskip 0ex
\'Ecole Normale Sup\'erieure, Paris, France.}
\email{aspremon@ens.fr}

\date{\today}
\subjclass[2018]{}

\begin{abstract}
	The seriation problem seeks to reorder a set of elements given pairwise similarity information, so that elements with higher similarity are closer in the resulting sequence.
	When a global ordering consistent with the similarity information exists, an exact spectral solution recovers it in the noiseless case and seriation is equivalent to the combinatorial 2-SUM problem over permutations, for which several relaxations have been derived. However, in applications such as DNA assembly, similarity values are often heavily corrupted, and the solution of 2-SUM may no longer yield an approximate serial structure on the elements. We introduce the robust seriation problem and show that it is equivalent to a modified 2-SUM problem for a class of similarity matrices modeling those observed in DNA assembly. We explore several relaxations of this modified 2-SUM problem and compare them empirically on both synthetic matrices and real DNA data. We then introduce the problem of seriation with duplications, which is a generalization of Seriation motivated by applications to cancer genome reconstruction. We propose an algorithm involving robust seriation to solve it, and present preliminary results on synthetic data sets.
	
\end{abstract}

\maketitle
	
\section{Introduction}
In the seriation problem, we are given a similarity matrix between a set of $n$ elements, which we assume to have a serial structure, \ie, which can be ordered along a chain where the similarity between elements decreases with their distance within this chain. 
The problem has its roots in archeology where it was used to find the chronological order of a set of graves based on the artifacts they share \citep{Robi51}. It also has applications in, \eg, envelope reduction \citep{Barn95}, bioinformatics \citep{atkins1996physical,cheema2010thread,jones2012anges} and in DNA sequencing \citep{Meid98,Garr11,recanati2016spectral}. The main structural hypothesis on similarity matrices related to robust seriation is the concept of strong $R$-matrix, which we introduce below.
\begin{definition}\label{def:strong-R-mat}
	We say that $A\in\symm_n$ is a strong-R-matrix (or strong Robinson matrix) iff it is symmetric and satisfies
	$A_{ij}\leq A_{kl}$ for all $(i,j,k,l)$ such that $|i-j|\geq |k-l|$.
\end{definition}
Here, $\symm_n$ denotes the set of real symmetric matrices of dimension~$n$.
Definition~\ref{def:strong-R-mat} is more restrictive than the usual R-matrix property used in \citet{Atkins,Fogel}, which only requires the entries of the matrix to decrease when moving away from the diagonal \textit{on a given row or column}. For strong-R matrices, we impose that the entries on a given diagonal are no greater than any entry located on the previous diagonals (see Figure~\ref{fig:RvsStrongRmat}).
\begin{figure}[hbt]
	\begin{center}
		\begin{subfigure}[htb]{0.3\textwidth}
			\includegraphics[width=\textwidth]{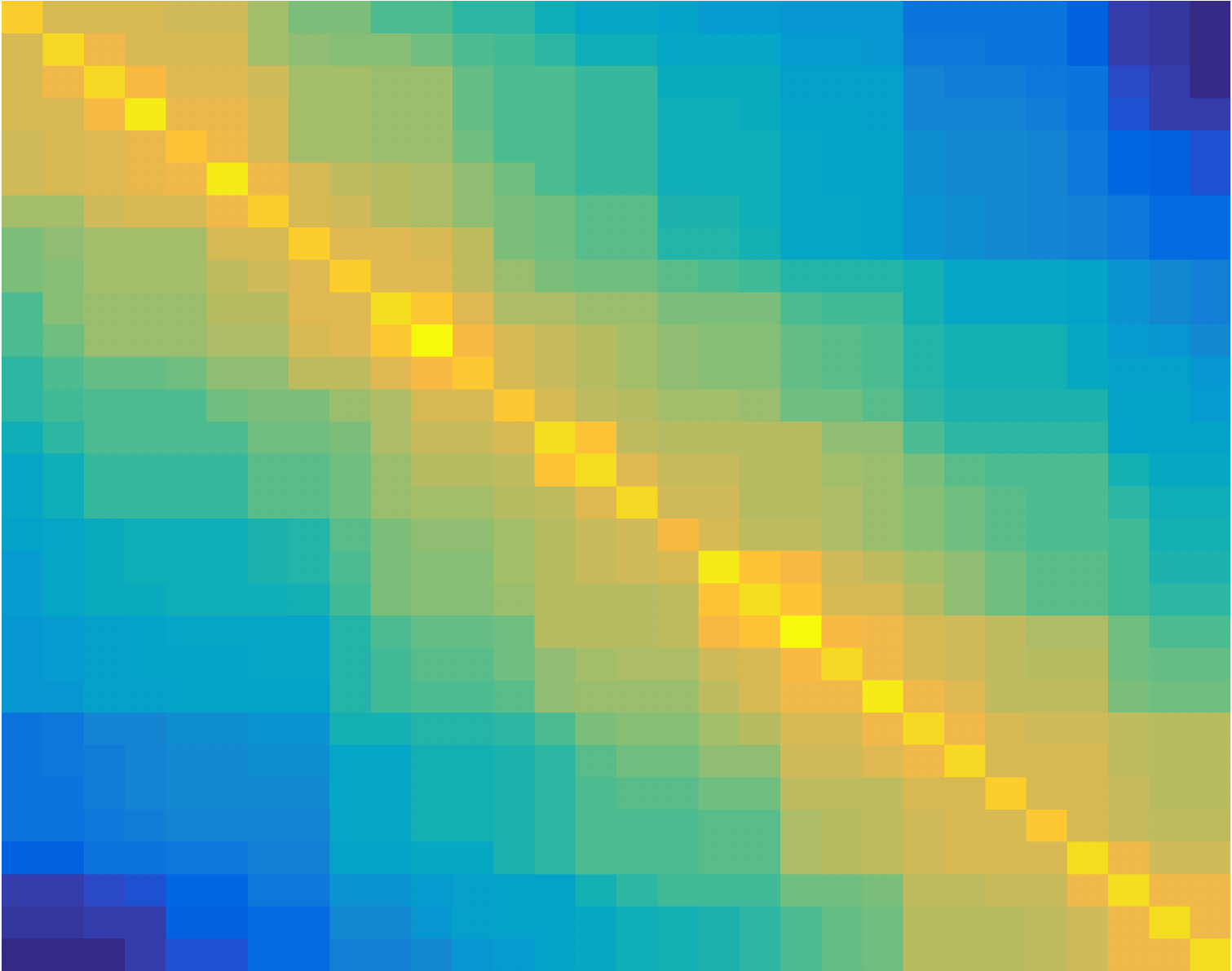}
			\caption{R-matrix}\label{subfig:softRmat}
		\end{subfigure}
		\begin{subfigure}[htb]{0.3\textwidth}
			\includegraphics[width=\textwidth]{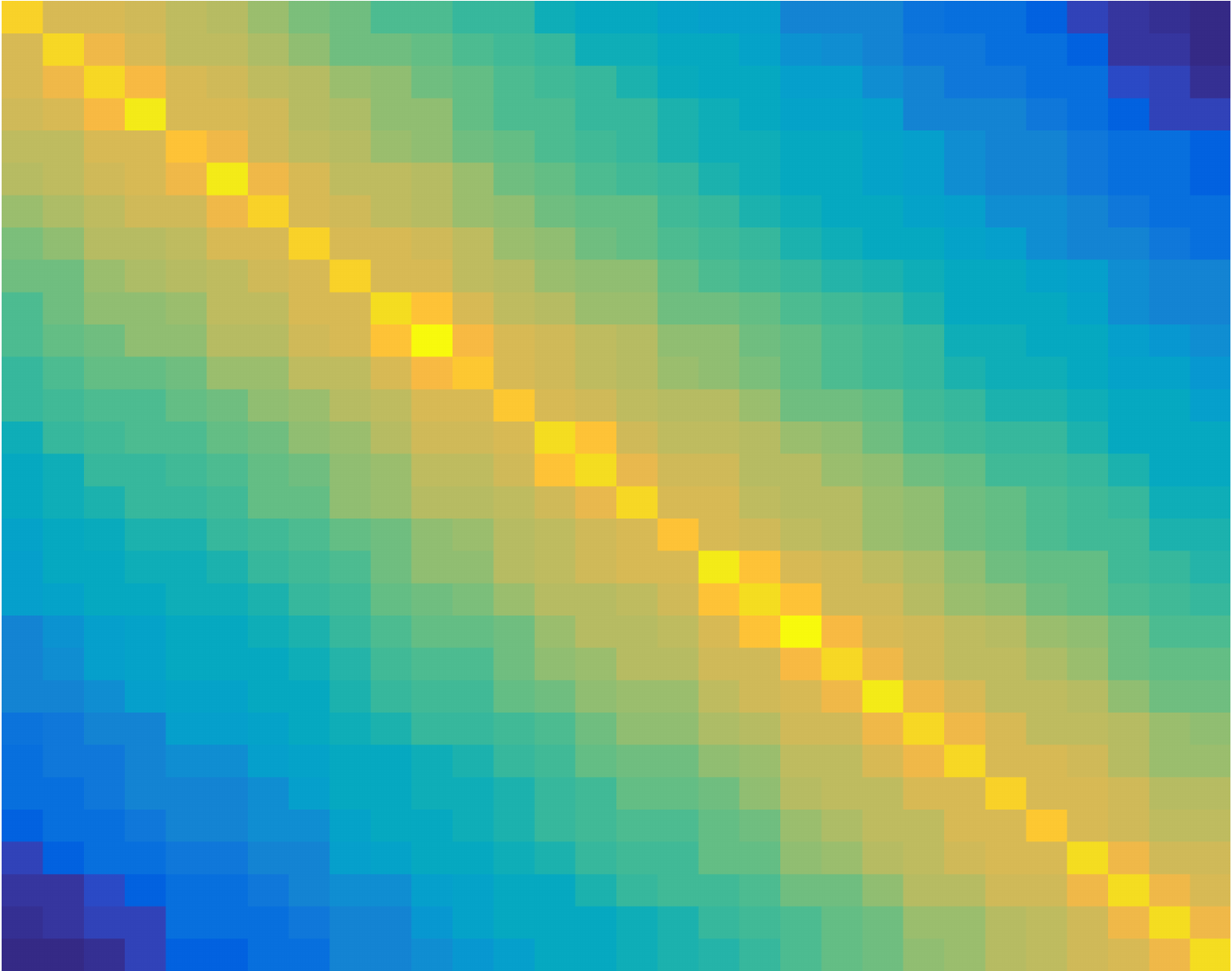}
			\caption{strong R-matrix}\label{subfig:strongRmat}
		\end{subfigure}
		\begin{subfigure}[htb]{0.3\textwidth}
			\includegraphics[width=\textwidth]{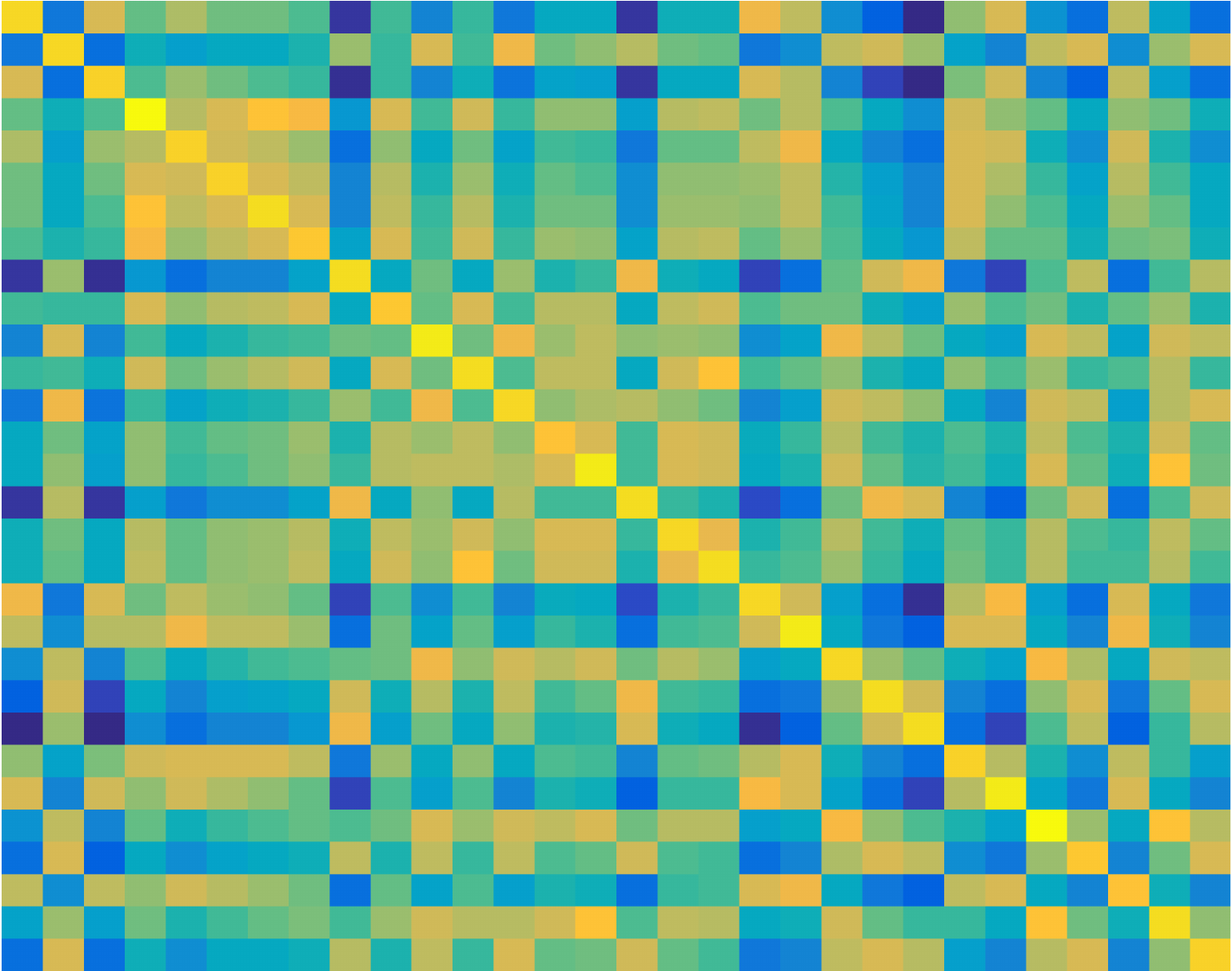}
			\caption{permuted strong-R}\label{subfig:2dRperm}
		\end{subfigure}
		\caption{A R-matrix (\textsc{A}) and its projection on the set of strong-R matrices (\textsc{B}). 
			A pre-strong-R matrix (\textsc{C}) is a strong-R matrix up to a permutation of the rows and columns.
			Seriation seeks to recover the R-matrix (\textsc{B}) from a randomly permuted observation (\textsc{C}).
		}
		\label{fig:RvsStrongRmat}
	\end{center}
	\vskip -0.2in
\end{figure}

In what follows, we write $\cR^n$ the set of strong-R-matrices of size $n$, and $\cP_n$ the set of permutations of $n$ elements. A permutation can be represented by a vector $\pi$ (lower case) or a matrix $\Pi \in \{0,1\}^{n \times n}$ (upper case) defined by $\Pi_{ij} = 1$ iff $\pi(i) = j$, and $\pi = \Pi \pi_{Id}$ where $\pi_{Id} = (1, \dots, n)^T$. We refer to both representations by $\cP_n$ and may omit the subscript $n$ whenever the dimension is clear from the context.
We say that $A \in \symm_n$ is pre-$\cR$ if there exists a permutation $\Pi \in \cP$ such that the matrix $\Pi A \Pi^T$ (whose entry $(i,j)$ is $A_{\pi(i),\pi(j)}$) is a strong-R-matrix, and the seriation problem seeks to recover this permutation $\Pi$, \ie, solve
\begin{align}\tag{Seriation} \label{eqn:seriation}
\BA{ll}
\find &  \Pi \in \cP\\
\st & \Pi A \Pi^T  \in \cR
\EA
\end{align}
in the variable $\Pi \in \cP$. This is illustrated in Figure~\ref{fig:RvsStrongRmat}.
 Given $A \in \symm_n$, 2-SUM is an optimization problem over permutations, written
\begin{align}
\tag{2-SUM} \label{eqn:2sum}
\BA{ll}
\mbox{minimize}  & \sum_{i,j=1}^n 
A_{ij} |\pi_i - \pi_j|^2 \\
\st & \pi \in \cP_n
\EA
\end{align}
Remark that the search space $\cP_n$ is discrete and of cardinality $n!$, thus preventing the use of exhaustive search or greedy branch and bound methods for \ref{eqn:seriation} or \ref{eqn:2sum} when $n$ gets large \citep{hahsler2008getting}.
Yet, for pre-$\cR$ matrices, \ref{eqn:seriation} is equivalent to \ref{eqn:2sum} \citep{Fogel}, which can be solved exactly using a spectral relaxation (Supplementary Algorithm~\ref{alg:spectral}) in polynomial time \citep{Atkins}.

Problem \ref{eqn:2sum} is also a particular case of the Quadratic Assignment Problem \citep{koopmans1957QAP}, written
\begin{align}\tag{QAP(A,B)} \label{eqn:QAP}
\min_{\pi \in \cP_n}\:\: & \sum_{i,j=1}^n 
A_{i,j}B_{\pi(i),\pi(j)}
\end{align}
with $B_{ij} = |i - j|^2$.
\citet{laurent2015QAPsolvable} showed that for pre-$\cR$ matrices, \ref{eqn:seriation} is equivalent to \ref{eqn:QAP} when $-B \in \cR^n$, i.e. when $B$ has increasing values when moving away from the diagonal, and has constant values across a given diagonal (i.e. $B$ is a Toeplitz matrix). This includes p-SUM problems, for $p >0$, corresponding to $B_{ij} = |i-j|^p$. The case $p=1$ is also known as the minimum linear arrangement problem (MLA) \citep{George:1997aa}.
For pre-$\cR$ matrices, these problems are all equivalent and can be solved by the spectral algorithm of \citet{Atkins}, detailed in the Supplementary Material (Algorithm~\ref{alg:spectral}).
However, when $A$ is not pre-$\cR$, the \ref{eqn:seriation} problem has multiple local solutions, and the spectral algorithm does not necessarily find a global optimum for \ref{eqn:2sum}, p-SUM or \ref{eqn:QAP} with $B$ a Toeplitz, negated R matrix. In fact, these problems are NP-hard in general \citep{sahni1976p}.

More recently, several relaxations have been proposed to tackle \ref{eqn:2sum} and \ref{eqn:QAP}, although there are no approximation bounds in the general case \citep{lyzinski2016graph}. \citet{vogelstein2011fast} used the Frank-Wolfe algorithm to minimize the objective of \ref{eqn:QAP} over the convex hull of the permutation matrices, namely the Birkhoff polytope $\cB$. \citet{Fogel} presented a convex relaxation of \ref{eqn:2sum} in~$\cB$, and used a quadratic programming approach where the variable's membership to $\cB$ is enforced through linear constraints (instead of the implicit projection of the Frank-Wolfe algorithm).
\citet{lim2014beyond} proposed a similar relaxation in the convex hull of the set of permutation vectors, the Permutahedron $\cPH_n$, represented with $\Theta(n\log n)$ variables and constraints, instead of $\Theta(n^2)$ for permutation matrices, thanks to an extended formulation by \citet{Goem09}.
All these relaxations for \ref{eqn:2sum} suffer from a symmetry problem, because flipping permutations leaves the objective unchanged, and the minimum of \ref{eqn:2sum} is achieved for a vector proportional to $\ones = (1, \dots, 1)^T$, which lies in the center of the convex hull of permutation vectors.
To overcome this issue, constraints can be added to the problem, corresponding to either {\it a priori} kwowledge, or to pure ``tie-breaking'', \eg, $\pi_1 + 1 \leq \pi_n$, ensuring that the center is excluded from the constraint set, thus breaking symmetry without loss of generality. \citet{lim2014beyond} stated that a Frank-Wolfe algorithm could also be used for \ref{eqn:2sum} in $\cPH$ if no other constraint but the tie-breaking was enforced, thanks to a specific linear minimization oracle, thus implicitly enforcing membership to $\cPH$ without imposing the constraints from \citet{Goem09}.
\citet{lim2016box} generalized the use of the representation of \citet{Goem09} for $\cPH$ to tackle \ref{eqn:QAP}, with a coordinate descent algorithm and a continuation scheme to move away from the center of the convex hull of permutations.
\citet{evangelopoulos2017graduated} proposed a Frank-Wolfe algorithm in $\cPH$ with a continuation scheme (instead of a tie-breaking constraint) to tackle \ref{eqn:2sum} and avoid the center. They also discussed problems of the form \eqref{eqn:QAP} where $B_{ij} = \text{Pseudo-Huber}(|i - j|)$ \citep{evangelopoulos2017unpublished}, which helps in solving robust seriation as we will see below.

Our contribution here is twofold. In Section~\ref{sec:RobustSeriation}, we introduce the robust seriation problem, motivated by applications to DNA sequencing.
We show that for DNA data obeying a simple model which takes repeats into account, robust seriation is equivalent to Robust 2-SUM, which is a QAP problem similar to \ref{eqn:2sum}, where the squared distance to the diagonal that appears in the loss function is truncated. This truncated quadratic can be relaxed as a Huber loss. We present experiments to compare existing and new algorithmic approaches to solve this problem on two datasets: synthetic data following our simple model, and real data from an {\it E. coli} genome sequenced with third generation sequencing tools.

In Section~\ref{sec:SeriationDuplications}, we introduce the problem of seriation with duplications, which is a generalization of seriation with additional affine constraints on the solution, and apply it to the analysis of Hi-C data from cancer genome. The Hi-C protocol combines proximity ligation and sequencing techniques to investigate the spatial organization of genomes \citep{Lieberman2009Comprehensive}. While several methods have been proposed to perform genome assembly from normal (haploid or diploid) Hi-C data \citep{selvaraj2013, Marie-Nelly2014, dudchenko2017novo}, to our knowledge, none of these methods were applied to cancer genomes reconstruction with duplicated subsequences. We therefore propose an alternating projection approach where one of the steps reduces to solving robust seriation. We detail preliminary experimental results on synthetic data.

\section{Robust Seriation}
\label{sec:RobustSeriation}
Classical \ref{eqn:seriation} is written as a feasibility problem: find the permutation that reorders the input matrix into an Robinson matrix. When $A$ is pre-$\cR$, solving \ref{eqn:2sum} yields this permutation. However, when $A$ is not pre-$\cR$, the matrix $A$ reordered using the permutation that minimizes \ref{eqn:2sum} may be far from being R. Robust seriation seeks to find the closest pre-$\cR$ matrix to $A$ and reorder it, solving instead
\begin{align}\tag{Robust Seriation} \label{eqn:robustseriation1}
\BA{ll}
\mbox{minimize} &  \|S - \Pi A \Pi^T \| \\
\st & \Pi \in \cP, \quad S \in \cR.
\EA
\end{align}
where the variable $\Pi \in \cP$ is a permutation matrix, the variable $S \in \cR$ is a strong-R-matrix, and the norm is typically either the $l_1$ norm on components or the Froebenius norm.

\subsection{Application of Seriation to Genome Sequencing}
In {\it de novo} genome sequencing, a whole DNA strand is reconstructed from randomly sampled sub-fragments (called {\it reads}) whose positions within the genome are unknown. The genome is oversampled so that all parts are covered by multiple reads with high probability. 
Overlap-Layout-Consensus (OLC) is a major assembly paradigm based on three main steps.
First, compute the overlaps between all pairs of read. This provides a similarity matrix $A$, whose entry $(i,j)$ measures how much reads $i$ and $j$ overlap (and is zero if they do not).
Then, determine the layout from the overlap information, that is to say find an ordering and positioning of the reads that is consistent with the overlap constraints. This step, akin to solving a one dimensional jigsaw puzzle, is a key step in the assembly process.
Finally, given the tiling of the reads obtained in the layout stage, the consensus step aims at determining the most likely DNA sequence that can be explained by this tiling. It essentially consists in performing multi-sequence alignments.

In the true ordering (corresponding to the sorted reads' positions along the genome), a given read overlaps much with the next one, slightly less with the one after it, and so on, until a point where it has no overlap with the reads that are further away. This makes the read similarity matrix Robinson and roughly band-diagonal (with non-zero values confined to a diagonal band). 
Finding the layout of the reads therefore fits the \ref{eqn:seriation} framework.
In practice however, there are some repeated sequences (called {\it repeats}) along the genome that induce false positives in the overlap detection tool \citep{Pop04}, resulting in non-zero similarity values outside (and possibly far away) from the diagonal band. The similarity matrix ordered with the ground truth is then the sum of a Robinson band matrix and a sparse ``noise'' matrix, as in Figure~\ref{subfig:sim_mat_out}.

Repeats longer than the overlap length are perhaps the most fundamental issue in genome assembly as they lead to ambiguous reconstructions. For example, consider the sequence ARBRCRD, where A,B,C,D,R are subsequences and R is repeated three times. The overlap constraints arising from this sequence are identical to those of ARCRBRD, therefore the overlap constraints are not sufficient to uniquely determine the layout.
Recently, long-reads sequencers such as PacBio’s SMRT and Oxford Nanopore Technology (ONT) spurred a renaissance in assembly by enabling sequencing reads over 10kbp (kilo basepairs) long, resolving many small repeats \citep{KorenOneChr}.
However, their error rate is high $(\sim15\%)$.
Thus, many assemblers include a correction module in a preprocessing step, which can help in separating repeats when the repeated copies slightly differ \citep{Pop04}.
They also use statistical models on the data generation in order to filter out the overlaps that are likely to be repeat-induced, and retrospectively inspect the overlap graph for potential errors in a greedy fashion, until the graph is ``cleaned'' and contains as few ambiguities for reconstruction as the model allows for \citep{koren2017canu,Li:Miniasm}.
When there are ambiguities, the ambiguous reads are simply removed and the obtained assembly is fragmented.

While most of these state of the art methods deal with repeats through complex pipelines involving heuristics and additional information on the data, \citet{recanati2016spectral} developed an assembler based on seriation which only uses the overlap-based similarity matrix $A$ to find the layout, computing it with the spectral algorithm (\ref{alg:spectral}). Yet, the presence of repeats often corrupts that ordering, as we illustrate in Figure~\ref{fig:twoSumWrongOnEcoli}. To overcome this issue, the method also ends up removing overlaps from the graph, yielding fragmented assemblies.

Here, we seek to apply \ref{eqn:robustseriation1} to genome sequencing, dealing with the repeats in a principled manner. We write $\cM_n(\delta, s)$ the set of matrices in $\{0,1\}^{n\times n}$ that are the sum of a band matrix of bandwidth~$\delta$ and a sparse out-of-band matrix with $s$ non-zero elements,
\begin{definition}\label{def:bd+noise-mat}
	$A \in\{0,1\}^{n\times n}$ belongs to $\cM_n(\delta, s)$ iff it is symmetric and satisfies
	$A_{ij} = 1$ for all $(i,j)$ such that $| i - j| \leq \delta$, and $\text{nnz}(A) = \left( n + (2 n - 1) \delta - \delta^2 \right) + s$.
\end{definition}
Here $\text{nnz}(A)$ is the number of non-zero elements of $A$, and the first term in the sum is the total number of elements in the bands. This means in particular $s \leq n^2 - \left( n + (2 n - 1) \delta - \delta^2 \right)$ (the total number of non-zeros cannot exceed $n^2$). In this setting, we wish to find an ordering in which most pairs of similar elements are nearby.
The \ref{eqn:2sum} objective can perform poorly here, since it strongly penalizes orderings with non-zero values far away from the diagonal, even when there is a small number of them, as we can see in Figure~\ref{fig:twoSumWrongOnEcoli}. Reducing this penalty on outliers is the goal of the robust seriation methods detailed below.
\begin{figure}[h]
	\begin{center}
	\begin{subfigure}[htb]{0.33\textwidth}
		\includegraphics[width=\textwidth]{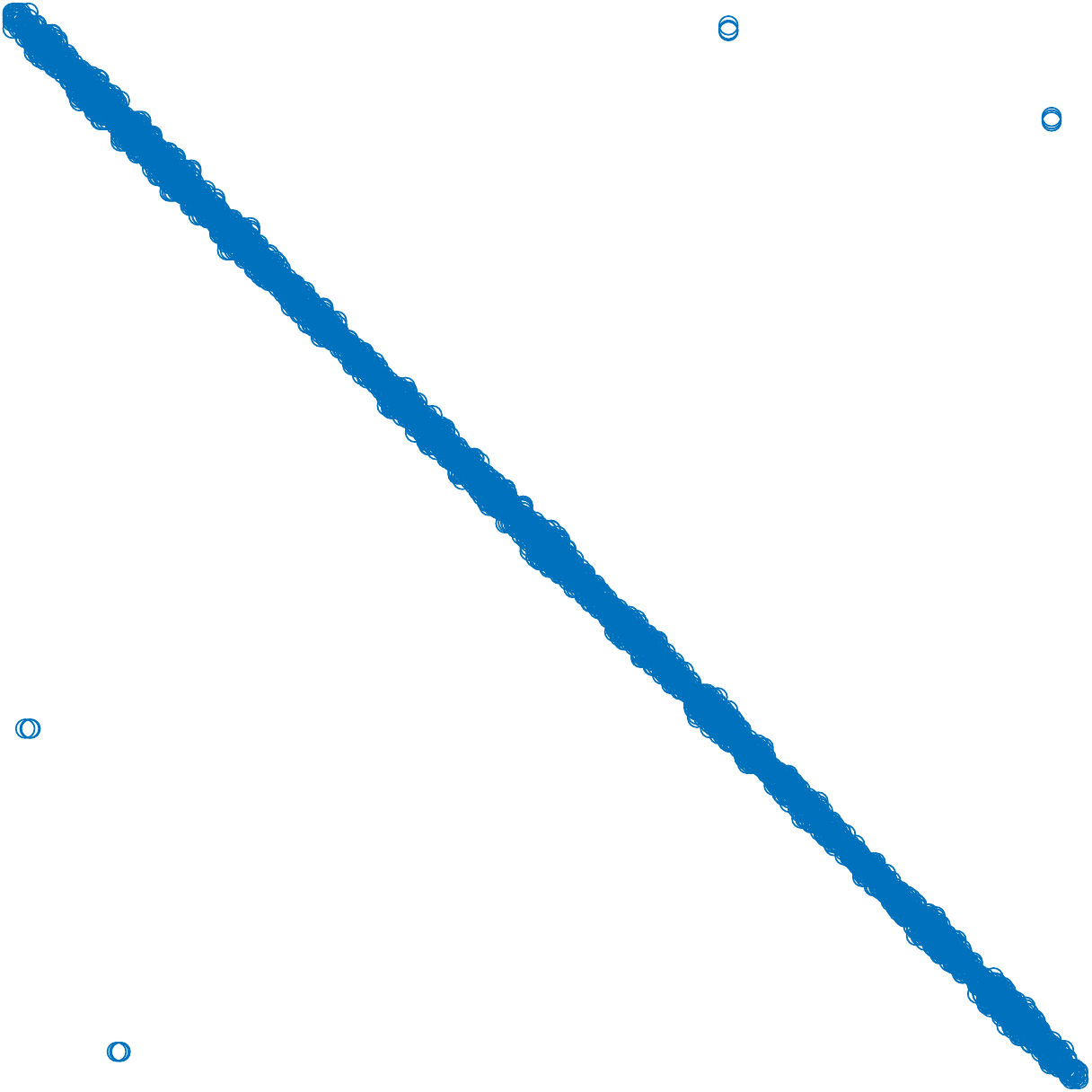}
		\caption{Ground truth}\label{subfig:sim_mat_out}
	\end{subfigure}
	\qquad 	\qquad 	\qquad
	\begin{subfigure}[htb]{0.33\textwidth}
		\includegraphics[width=\textwidth]{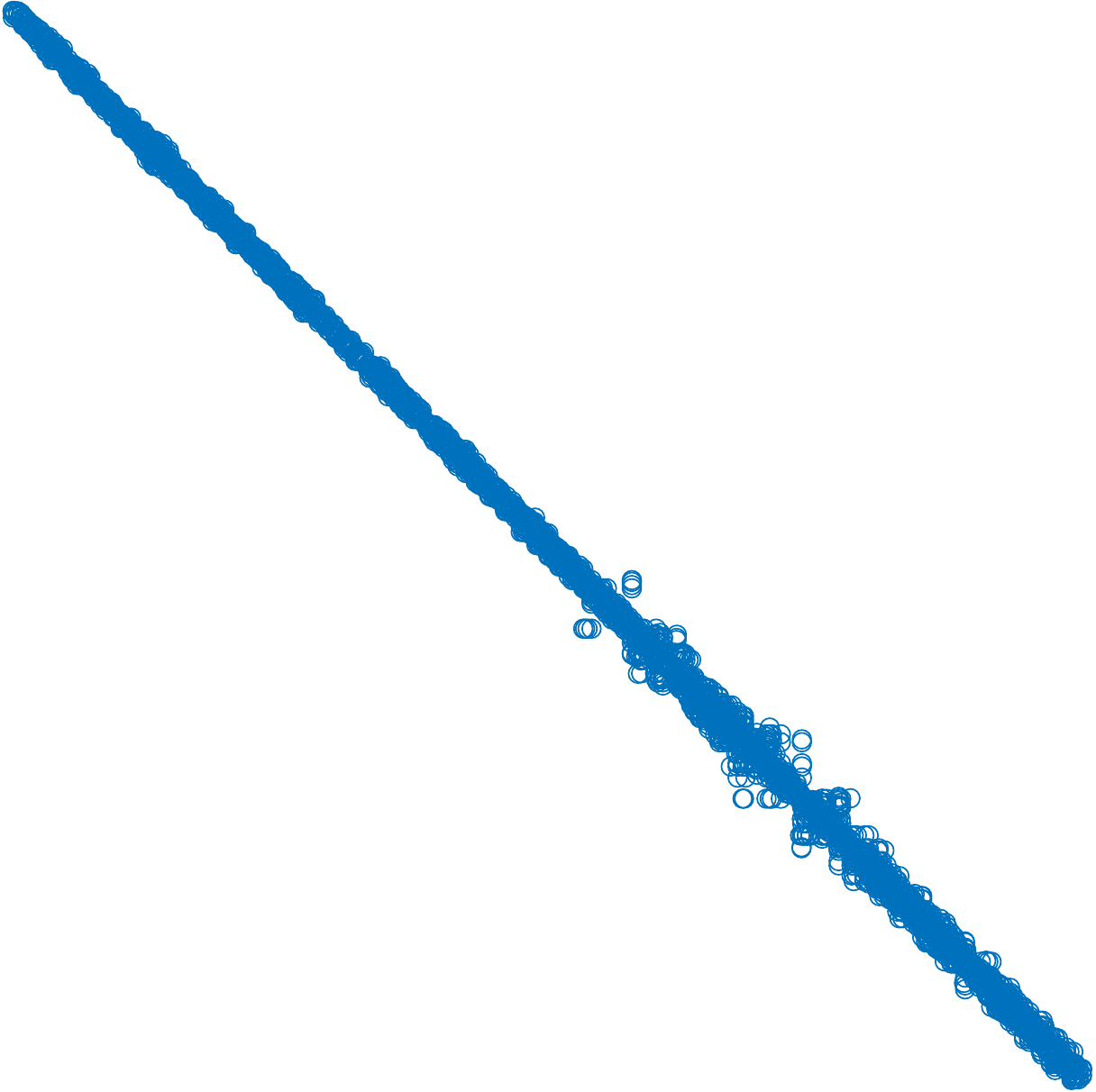}
		\caption{2SUM}\label{subfig:sim_mat_2sum}
	\end{subfigure}
		\caption{Similarity matrix from a subset of Oxford Nanopore reads of {\it E. coli} in the ordering given by the ground truth position of the reads along the genome (\ref{subfig:sim_mat_out}, left), and
			the same matrix reordered by minimizing the 2SUM objective (\ref{subfig:sim_mat_2sum}, right), which pushes the out-of-diagonals terms close to the main diagonal and yields a corrupted ordering.}
		\label{fig:twoSumWrongOnEcoli}
	\end{center}
	\vskip -0.2in
\end{figure}

\subsection{Robust 2-SUM}\label{ssec:robust2sum}
Given $A \in \symm_n$, \ref{eqn:robustseriation1} seeks to find a pre-$\cR$ matrix that is as close to $A$ as possible. 
Instead of searching directly for a perturbation of~$A$ that is pre-$\cR$, we search for a perturbation of~$A$ that yields a low \ref{eqn:2sum} score, solving
\begin{align}\tag{R2S($\lambda$)} \label{eqn:robust2SUM1}
\BA{ll}
\mbox{minimize} &  \sum_{i,j=1}^n 
S_{ij} |\pi_i - \pi_j|^2 +  \lambda \|A - S \|_1 \\
\st & \pi \in \cP, \quad S \in \symm_+.
\EA
\end{align}
where $\symm_+$ is the set of symmetric matrices with non-negative entries, and we use the $l_1$ norm on the difference between $A$ and $S$ to enforce sparsity in errors. Here, $\lambda$ is a parameter that controls the deviation of~$S$ from~$A$.
The sum is separable and the minimization in $S$ is closed form. Indeed, for a given $(i,j)$, the function $S_{ij} \rightarrow S_{ij} \Delta_{ij}^2 + \lambda |S_{ij} - A_{ij}|$ is piecewise linear, with slope $\Delta_{ij}^2 - \lambda$ for $S_{ij} \leq A_{ij}$, and $\Delta_{ij}^2 + \lambda$ for $S_{ij} \geq A_{ij}$, and is therefore minimal at $S_{ij} = A_{ij}$ if $\Delta_{ij}^2 \leq \lambda$ and $S_{ij} = 0$ otherwise (recall that $S_{ij}$ is constrained to be non-negative).
Hence, \ref{eqn:robust2SUM1} is equivalent to
\begin{align}\tag{R2SUM($\lambda$)} \label{eqn:robust2SUM2}
\BA{ll}
\mbox{minimize} &  \sum_{i,j=1}^n 
A_{ij} \min(\lambda, |\pi_i - \pi_j|^2)\\
\st & \pi \in \cP.
\EA
\end{align}
in the variable $\pi \in \cP$. We now show that for stylized genome assembly similarity matrices, if the number of reads spanning repeated regions is controlled, then solving \ref{eqn:robust2SUM2} also solves \ref{eqn:robustseriation1}.

\begin{proposition}\label{prop:RSvsR2SUM}
	For $s \leq s_\mathrm{lim} \triangleq (n - \delta - 1)$ and  $A \in \symm_n$,
	if $A$ can be permuted to belong to $ \cM_n(\delta,s)$,
	\ie, if there is $\Pi \in \cP_n: \: \Pi A \Pi^T \in \cM_n(\delta,s)$,
	then $\Pi$ solves both \ref{eqn:robustseriation1} and  \ref{eqn:robust2SUM2} with parameter $\lambda = \delta^2$, and the $\ell_1$ norm in \ref{eqn:robustseriation1}.
\end{proposition}
\begin{proof}
	Let $\delta, s$ be two positive integers such that $\delta \leq n$, $s \leq (n-\delta -1)$.
	Without loss of generality, assume that $A \in  \cM(\delta,s)$,  \ie, $\Pi = \idm$, the identity permutation (otherwise, we simply factor out the true permutation). First, let us observe that for $\lambda = \delta^2$, $\idm$ is optimal for \ref{eqn:robust2SUM2}.
	Indeed, since $A \in \{0,1\}^{n \times n}$, the objective in \ref{eqn:robust2SUM2} is the sum of $\min(\delta^2, |\pi_i - \pi_j|^2)$ over all indexes $(i,j)$ such that $A_{ij}=1$.
	This sum can be split into two terms, \\
	\[f_\text{in} = \sum_{(i,j) : A_{ij}=1\: , \: |\pi_i - \pi_j| \leq \delta} |\pi_i - \pi_j|^2,\] and \\
	\[f_\text{out} = \sum_{(i,j) : A_{ij}=1\: , \: |\pi_i - \pi_j| > \delta} \delta^2.\]
	For $\Pi = \idm$, the number of terms in $f_\text{in}$ is maximized since $A_{ij} = 1$ for all $(i,j)$ such that $|i - j| \leq \delta$ ($A \in \cM(\delta,s)$).
	The sum of the number of terms in $f_\text{in}$ and $f_\text{out}$ is equal to nnz($A$) and is invariant by permutation (therefore, the number of terms in $f_\text{out}$ is also minimized for $\Pi = \idm$)
	Since any term in $f_\text{in}$ is smaller than any term in $f_\text{out}$, $\Pi = \idm$ is optimal for \ref{eqn:robust2SUM2} with $\lambda = \delta^2$.
	
	Now, let us see that $\Pi = \idm$ is also optimal for \ref{eqn:robustseriation1}. Given $\Pi$, optimizing over $S$ in \ref{eqn:robustseriation1} yields $ S_{\Pi} = \text{Proj}_{\cR}(\PAP)$, the projection of $\PAP$ onto the set of strong-R-matrices.
	Let us assume that we use the $\ell_1$ norm in \eqref{eqn:robustseriation1}.
	Then, $S_{\Pi}$, the projection in $\ell_1$ norm of the binary matrix $\PAP$, is also binary (see Lemma~\ref{lemm:proj-strongR-binary-is-binary} in Supp. Mat.).
	A sparse, $\{0,1\}$ strong-R-matrix is necessarily of the form 
	\[\left\{\BA{ll}
	S_{ij} = 1 & \mbox{if $ |i - j| \leq k$,}\\
	S_{ij} = 0 & \mbox{if $ |i - j| > k+1$,}\\
	S_{ij} \in \{0,1\} & \mbox{for $|i - j|= k+1$,}
	\EA\right.\]
	with the integer $k+1$ denoting the bandwidth of $S$.
	Given $S_{\Pi}$ and the corresponding $k$, the distance between $\PAP$ and $S_{\Pi}$ appearing in \ref{eqn:robustseriation1} is separable (whether we use the $l_1$ or Frobenius norm, since $A \in \{0,1\}^{n \times n}$) and can be grouped into three terms, according to whether $(i,j)$ is such that $|i-j| > k+1$, $|i-j| \leq k$ or $|i-j| = k+1$.
	The first term, $n_{\text{out}}(k) \geq 0$, equals the number of non-zero elements of $\PAP$ such that $|i-j| > k+1$.
	The second, $n_{\text{in}}(k) \geq 0$, equals the number of zero elements of $\PAP$ such that $|i-j| \leq k$.
	The third equals zero, because setting the (k+1)-th diagonal of $S$ identical to the (k+1)-th diagonal of $\PAP$ does not violate the R property of $S_{\Pi}$, and $S_{\Pi}$ is by definition the strong-R-matrix that minimizes the distance to $\PAP$.
	For any $\Pi$, if $k > \delta$, the number of non-zeros elements inside the band of width $k$ being bounded by the number of non-zero elements of $A$, we have $n_{\text{in}}(k) \geq 2 \left(n - \delta - 1\right)- s \geq (n - \delta - 1) \geq s$.
	Similarly, for $k \leq \delta$, $n_{\text{out}}(k) \geq s$.
	For $\Pi = \idm$, as long as $k \leq \delta$, $n_{\text{out}}(k)\leq s$ decreases with $k$ and $n_{\text{in}}(k)=0$. For $k = \delta$, $n_{\text{in}}(k)=0$ and $n_{\text{out}}(k) \leq s$ (it is equal to $s$ minus the number of elements in the $\delta+1$-th diagonal).
	Thus, $\Pi = \idm$ is optimal, and $k = \delta$.
\end{proof}

Note that in practice, one has to chose the parameter $\lambda$ without observing $\delta$ before trying to solve \ref{eqn:robust2SUM2}. Yet, for matrices $A$ satisfying the hypothesis of \ref{prop:RSvsR2SUM}, the number of non-zero values of $A$ (which is observed even when $A$ is permuted) provides a way to estimate $\delta$. We compute it as the smallest integer $\delta$ such that the number of non-zero elements in a band matrix of size $\delta$ is larger than nnz$(A)$.
Also remark that the proof of \ref{prop:RSvsR2SUM} is conservative: it only involves reasoning about the location of non-zero values of a vectorized version of $\PAP$. Permuting rows and columns of a matrix adds constraints on the locations of these non-zero values that we did not take into account.

\section{Robust Seriation Algorithms}\label{sec:rob-ser-algo}
We compare several methods to address the \ref{eqn:robust2SUM2} problem. We describe them in what follows and provide experimental results in Section~\ref{sec:experiments}.

\subsection{QAP solvers (FAQ and PHCD)}
The first strategy is to directly minimize the objective of \ref{eqn:robust2SUM2} using QAP solvers. Indeed, the problem matches \ref{eqn:QAP} with $B_{ij} = \min(\lambda, |i-j|^2)$.
We test the aforementioned \citet{vogelstein2011fast} and \citet{lim2016box} methods for solving the QAP.

The first, which we refer to as FAQ \citep{vogelstein2011fast}, uses the matrix representation of permutations with a relaxation in the convex hull of permutation matrices, $\cB$, where the \ref{eqn:QAP} objective is optimized with the conditional gradient (a.k.a. Frank-Wolfe) algorithm. Each step of Frank-Wolfe involves an assignment problem solved with a Hungarian algorithm \citep{kuhn1955hungarian}.

The latter, which we refer to as PHCD \citep{lim2016box}, uses the sorting-network based representation of permutation vectors of \citet{Goem09} and performs coordinate descent in the convex hull of permutation vectors $\cPH$.

For completeness, we also used these QAP solvers in the experiments to solve \ref{eqn:2sum} (i.e. \ref{eqn:QAP} with $B_{ij}=|i-j|^2$), and \ref{eqn:huberSUM} ($B_{ij}= h_{\delta}(|i-j|)$), which is described below.

\subsection{Huber Loss Relaxation of \ref{eqn:robust2SUM2}}
The objective of \ref{eqn:robust2SUM2} is not convex. In order to use convex optimization algorithms, it can be relaxed to its convex envelope, resulting in the following problem,
\begin{align}\tag{HuberSUM($\delta$)} \label{eqn:huberSUM}
\BA{ll}
\mbox{minimize} &  \sum_{i,j=1}^n 
A_{ij} h_{\delta}( |\pi_i - \pi_j|)\\
\st & \pi \in \cP.
\EA
\end{align}
where $h_{\delta}(x)$ is the Huber function, which equals $x^2$ when $|x| \leq \delta$, and $\delta (2 |x| - \delta)$ otherwise.

\subsection{Relaxations in $\cPH$}
A typical convex relaxation work-flow involves relaxing both the objective function to its convex envelope, and relaxing the constrained set to its convex hull, in order to use of the arsenal of convex optimization, including scalable first order methods.
Here, we seek to optimize the objective functions of \ref{eqn:2sum} and \ref{eqn:huberSUM}, $f_{\text{2SUM}}$ and $f_{\text{Huber}}$, on the convex hull of $\cP_n$, the polyhedron $\cPH_n$.

\subsubsection{Symmetry Issues}\label{sssec:sym-issue}
Unfortunately, the solution of a relaxation $\tilde{x} \in \cPH_n$ does not necessarily (and most of the time, not) lie in $\cP_n$. To retrieve a solution in $\cP_n$, one must project the relaxed solution $\tilde{x}$ onto the set of permutations $\cP_n$, which may be challenging.
Here, the flat vector $c_n \triangleq \frac{n+1}{2} \ones_n \in \cPH_n$ minimizes $f_{\text{2SUM}}$ and $f_{\text{Huber}}$
 in $\cPH_n$. Indeed, all its entries being equal, 
$f_{\text{2SUM}}(c_n) = f_{\text{Huber}}(c_n) = 0$, which is optimal since these sums involve only non-negative terms.
Yet, this optimum is non-informative. Any permutation $\pi \in \cP_n$ has the same distance to $c_n$, $d = \sum_{i=1}^{n} (\frac{n+1}{2} - i)^2$, thus projecting back $c_n$ to $\cP_n$ is completely degenerate. 

This is illustrated in Figure~\ref{fig:goodVsBadTieBreakPH3}, where $\cPH_3$ is a salmon-colored hexagone centered around~$c_3$ (red circled dot), and whose vertices are the permutations.
$\cPH_3$ is represented on a planar figure since $\cPH_n$ lies in a hyperplane of dimension $n-1$, $ \cH_n = \{ x \in \reals^n | x^T \ones = \frac{n (n+1)}{2} \}$. Indeed, all permutation vectors have the same set of elements, hence the same sum, and also the same norm, as one can see from the black dashed circle of fixed norm in Figure~\ref{fig:goodVsBadTieBreakPH3} on which all permutations lie.
The symmetry of center $c_n$, formally defined by $T(x) - c_n = - (x - c_n)$, is visible from the level lines of $f_{\text{2SUM}}$ (blue ellipses). The objectives from \ref{eqn:2sum}
and \ref{eqn:huberSUM} are invariant under the ``flipping'' operator $T$. For instance, the permutation $\pi=(1,3,2)^T$ and its symmetric $T \pi = (n+1) \ones - \pi = (3,1,2)$ are on the same level line. This is the fundamental reason why the minimum of \ref{eqn:2sum} and \ref{eqn:huberSUM} lies in the center, making the basic convex relaxation in $\cPH_n$ useless.

\begin{figure}[htb]
	\begin{center}
		\centerline{\includegraphics[width=0.5\columnwidth]{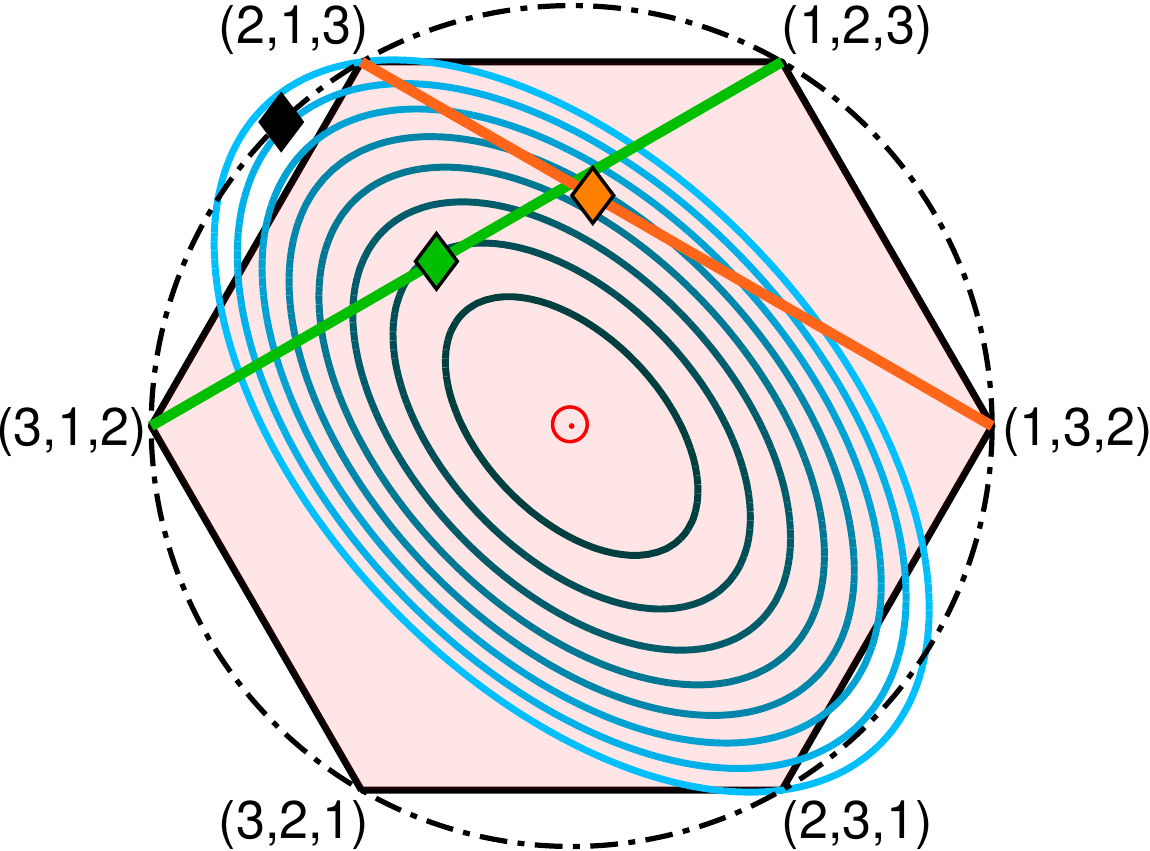}}
		\caption{
			View of the 3-Permutahedron $\cPH_3$ (filled polygon) in the 2D plane $\cH_3$ (orthogonal to the vector $\ones_3$ represented by the red pointing arrow (circled dot)).
			The blue ellipses are the level curves of $f_{\text{2SUM}}$.
			The black dashed circle represents the set of points having the same norm as the permutation vectors, and the black diamond is the minimizer of \ref{eqn:2sum} among them.
			The green (resp. orange) line is where the ``good'' (resp. ``bad'') tie-breaking constraint $\pi_2 +1 \leq \pi_3$ (resp. $\pi_1 +1 \leq \pi_3$) is active, and the green (resp. orange) diamond is the minimizer of $f_{\text{2SUM}}$ on the corresponding constrained set, the triangle $\left((2,1,3), (1,2,3), (1,3,2)\right)$ [resp. $\left((3,1,2), (2,1,3), (1,2,3)\right)$].
			The closest permutation to the green diamond is $(2,1,3)$, which is the correct solution (minimizer of $f_{\text{2SUM}}$ on $\cP_3$), but the orange diamond is closer to $(1,2,3)$ because of the anisotropy induced by the tie-breaking constraint.
			Figure adapted from \citet{lim2014beyond}.
		}
		\label{fig:goodVsBadTieBreakPH3}
	\end{center}
	\vskip -0.2in
\end{figure}

To overcome this issue, \citet{Fogel,lim2014beyond} employ two strategies. One is to add a penalty in the objective that increases towards to the center $c$, \eg,  add the concave penalty $- \mu \| x -c \|^2$ to the objective. The other one is to add constraints that keep the center $c$ out of the feasible set, \eg, add the tie-breaking constraint $\pi_1 + 1 \leq \pi_n$.
This resolves the ambiguity about the direction of the ordering without removing any permutation from the search space (up to a flip), since, for any permutation $\pi \in \cP$, either $\pi$ satisfies the constraint, or its symmetric $T(\pi)$ does.
On Figure~\ref{fig:goodVsBadTieBreakPH3}, that tie-breaking constraint is active on the orange line, and the constrained set satisfying it is the top-right triangle of $\cPH_3$, $\left((2,1,3), (1,2,3), (1,3,2)\right)$.
We consider methods employing both strategies in what follows.

\subsubsection{Frank-Wolfe with tie-breaking constraint (FWTB)}

The conditional gradient (Frank-Wolfe) algorithm in $\cPH$ solves a linear problem at each step, $s_t \in \argmin_{v \in \cP} \{ g_{t}^{T} v\}$, where $g_{t}$ is the gradient at the current point, and then updates the iterate $x^{(t+1)}\gets~ \gamma~x^{(t)}+(1-\gamma)s_t$, thus keeping a (sparse) convex combination of points of $\cP$ as a solution. Here, solving the linear problem boils down to sorting the entries of $g_{t}$ (see \S~\ref{sssec:FWTB-LMO}).
Adding the tie-breaking constraint $\pi_1 + 1 \leq \pi_n$ slightly modifies the linear minimization without affecting the $O(n \log n)$ algorithmic complexity \citep{lim2014beyond}.
We implemented the Frank-Wolfe algorithm with tie-break, using away-steps as in \citet{lacoste2015global}.
Interestingly, this method performed poorly in the experiments. We observed that the tie-break induces a bias in the problem, and that the choice of the tie-break $\pi_i + 1 \leq \pi_j, \: 1\leq i \neq j \leq n $ plays a key role in the performances, as we illustrate in Figure~\ref{fig:goodVsBadTieBreakPH3} with the ``good'' (green) and ``bad'' (orange) tie-breaks.
In Supplementary \S\ref{ssec:FWTB-biased}, we provide details about this bias issue,
and propose a linear minimization oracle for any tie-break of the form $\pi_i + 1 \leq \pi_j$ adapted from the idea of \citet{lim2014beyond} in \S\ref{sssec:FWTB-LMO}.
\subsubsection{Graduated Non-Convexity : Frank-Wolfe Algorithm with Concave Penalty (GnCR and HGnCR)}
In \citet{Fogel,lim2014beyond}, the parameter $\mu$ controlling the amplitude of the penalty $-\mu~\|x-c\|^2$ is limited in order to keep the objective convex.
Precisely, $f_{\text{2SUM}} = x^T L_A x$ becomes $\tilde{f}(x) = x^T L_A x - \mu \| P x \|^2 = x^T ( L_A - \mu P) x $, where $L_A =\diag(A\ones)-A$ is the Laplacian of A and $P = \idm - \frac{1}{n} \ones \ones^T$ projects on the subspace orthogonal to $\ones$.
To keep the problem convex, $\mu$ needs to be smaller than $\lambda_2$, the smallest non-zero eigenvalue of $L_A$.
Still, for small values of $\lambda_2$, this may lead to solutions lying close to the center $c$ up to numerical precision. Also, for $f_{\text{Huber}}$ , the convexity is broken for any positive value of $\mu$.

\citet{evangelopoulos2017graduated} proposed a graduated non-convexity scheme called GnCR to solve \ref{eqn:2sum}, where $\mu$ is gradually increased in outer iterations of the problem, starting with a small value ($\mu \leq \lambda_2$) preserving convexity, and moving towards high values of ($\mu \geq \lambda_{\max}$)
, making the objective concave. This strategy aims at finding a sequence of solutions to the subproblems that follow a path from near $c_n$ (when the objective is convex) towards a permutation (when it is concave). To solve each subproblem, GnCR uses the Frank-Wolfe algorithm in $\cPH_n$ without tie-breaking constraint.
In \citet{evangelopoulos2017unpublished}, the approach is extended to a pseudo-Huber loss, thus approximately solving \ref{eqn:huberSUM}, with a method called HGnCR.
We include both methods in the experiments.

\subsubsection{Unconstrained Optimization in $\cH_n$ with Iterative Bias (UBI)}

We propose a method where we also add a penalty to $f_{\text{Huber}}$ in order to avoid the center $c$ : $\tilde{f}_{\text{Huber}}(x) = f_{\text{Huber}}(x) - \mu h(x)$, where $h$ is a penalty function pushing away from $c$, but perform \emph{unconstrained} optimization of $\tilde{f}_{\text{Huber}}$, \ie, we no longer restrict the search space to $\cPH_n$.
Still, if the penalty becomes negligible compared to $f_{\text{Huber}}(x)$ when $\|x-c\|$ gets large, the global solution will be bounded, and, up to a scaling of $\mu$, it will lie in $\cPH_n$.

We use a sigmoidal penalty, $h_{\lambda, w}(x) = 1/\left(1 + \exp\left(-\lambda (x-c)^T (w-c)\right)\right)$.
It breaks the symmetry by adding a bias in a given direction $w$.
We thus propose an iterative method where each outer iteration $t$ solves a subproblem biased towards a direction $w_{t}=x^*_{(t-1)}$ given by the previous iteration, using an unconstrained optimization descent method (LBFGS, implementation from \citet{schmidt2005minfunc}).

The method is summarized in the Supp. Mat., Algorithm~\ref{alg:Uncons}. We also provide technical details in \S~\ref{sssec:Uncons}.
Importantly, we do not perform unconstrained optimization in $\reals^n$ but in $\cH_n$. However, the hyperplane $\cH_n$ can be expressed as an affine transformation of $\reals^{n-1}$, thus we still use unconstrained optimization techniques (in $\reals^{n-1}$) to optimize $\tilde{f}_{\text{Huber}}$ in $\cH_n$ (see \S\ref{sssec:Hn-vs-Rn-1}).

\subsection{Relaxations on the Sphere}
The spectral algorithm minimizes $f_{\text{2SUM}}$ on a sphere of given norm by computing a second extremal eigenvector, thus resolving the center issue. It is summarized in Algorithm~\ref{alg:spectral} and the derivation is detailed in Supp. \S~\ref{sssec:spectral-ordering}.
\begin{algorithm}[ht]
	\caption{Spectral ordering}\label{alg:spectral}
	\begin{algorithmic} [1]
		\REQUIRE Connected similarity matrix $A \in \mathbb{R}^{n \times n}$
		\STATE Compute Laplacian $L_A=\diag(A\ones)-A$
		\STATE Compute second smallest eigenvector of $L_A$, $\mathbf{x^*}$
		\STATE Sort the values of $\mathbf{x^*}$
		\ENSURE Permutation $\pi : \mathbf{x^*}_{\pi(1)} \leq \mathbf{x^*}_{\pi(2)} \leq ... \leq \mathbf{x^*}_{\pi(n)}$
	\end{algorithmic}
\end{algorithm}
Up to a translation and dilatation of this sphere, it is represented by the black dashed circle in Figure~\ref{fig:goodVsBadTieBreakPH3}.
However, this was only possible because the \ref{eqn:2sum} objective takes the form of a quadratic form $x^T L_A x$. Optimizing \ref{eqn:huberSUM} over a sphere is more challenging.
We propose two methods to address this task.

\subsubsection{Spectral Relaxation}
We propose to extend the spectral Algorithm~\ref{alg:spectral} to \ref{eqn:huberSUM} through the variational form of the Huber loss (so-called $\eta$-trick).
The absolute value of a real number can be expressed as $|x| = \argmin_{\eta \geq 0} \frac{x^2}{\eta} + \eta$.
The analog for Huber is, $h_{\delta}(x) = \argmin_{\eta \geq \delta} \frac{x^2}{\eta} + \eta$.

We propose an alternating minimization scheme called $\eta$-Spectral, based on this variational form
and summarized in Algorithm~\ref{alg:etaTrickHuberSUM}. A detailed description is given in Supp. \S\ref{sssec:eta-spectral}.
\begin{algorithm}[ht]
	\caption{$\eta$-Spectral Alternate Minimization Scheme for \ref{eqn:huberSUM}.}
	\label{alg:etaTrickHuberSUM}
	\begin{algorithmic} [1]
		\REQUIRE A similarity matrix $A \in \symm_{n}^+$, a maximum number of iterations $T$.
		\STATE Set $t = 0$, $\eta^{(0)} = \ones_n \ones_{n}^T$.
		\WHILE{$t \leq T$}
		\STATE Compute $\pi^{(t)} \in \argmin_{\pi \in \cP} \left\{ \sum_{i,j=1}^n A_{ij} \left ( \frac{(\pi_i - \pi_j)^2}{\eta_{ij}^{(t)}} + \eta_{ij}^{(t)} \right) \right\}$, \ie, $\pi^{(t)}$ is solution of \eqref{eqn:2sum} for the matrix $A./\eta$ where $./$ denotes the Hadamard (entrywise) division.
		\STATE Compute $\eta^* \in \argmin_{0 \leq \eta \leq \delta} \left\{  \sum_{i,j=1}^n A_{ij} \left ( \frac{(\pi_{i}^{(t)} - \pi_{j}^{(t)})^2}{\eta_{ij}} + \eta_{ij} \right) \right\}$, \ie, $\eta_{ij}^* = h_{\delta}(|\pi_{i}^{(t)} - \pi_{j}^{(t)}|)$
		\STATE Update $\eta^{(t)} \gets \gamma \eta^{(t-1)} + (1 - \gamma) \eta^{*}$.
		\STATE  $t \gets t+1$.
		\ENDWHILE
		\ENSURE A permutation $\pi^{(T)}$.
	\end{algorithmic}
\end{algorithm}

\subsubsection{First Order Optimization on Manifold}
Finally, we used a manifold optimization toolbox \citep{boumal2014manopt} as a black-box, to which we provide the expression of the objective and gradient of \ref{eqn:huberSUM} and ask for the minimum over the sphere (computed with a trust-regions algorithm). Some additional details about the formulation of the sphere are given in Supp. \S\ref{sssec:Hn-vs-Rn-1}.

\section{Seriation with Duplications}
\label{sec:SeriationDuplications}
The reformulation of \emph{de novo} sequencing as a (robust) seriation problem is based on the assumption that, up to noise, the bins can be reordered to form a long chain. While this hypothesis is relevant when a normal genome or chromosome is sequenced with long reads, it clearly fails to hold in an important case: cancer genomes. Indeed cancer cells typically harbour so-called \emph{structural variations} where large portions of the genome, up to whole chromosomes, are duplicated or deleted, and where new chromosomes are formed by fusing two pieces of chromosomes which are not connected in a normal genome. For example, Figure~\ref{fig:cancer} shows the 1D structure of a breast cancer cell line. Different colors correspond to DNA fragments normally in different chromosomes. Instead of 23 pairs of  chromosomes with each pair in a single uniform color, expected in a normal cell, we observe various mosaics of colors indicating various duplication and fusion events.
\begin{figure}[ht]
	\begin{center}
		\centerline{\includegraphics[width=0.75\columnwidth]{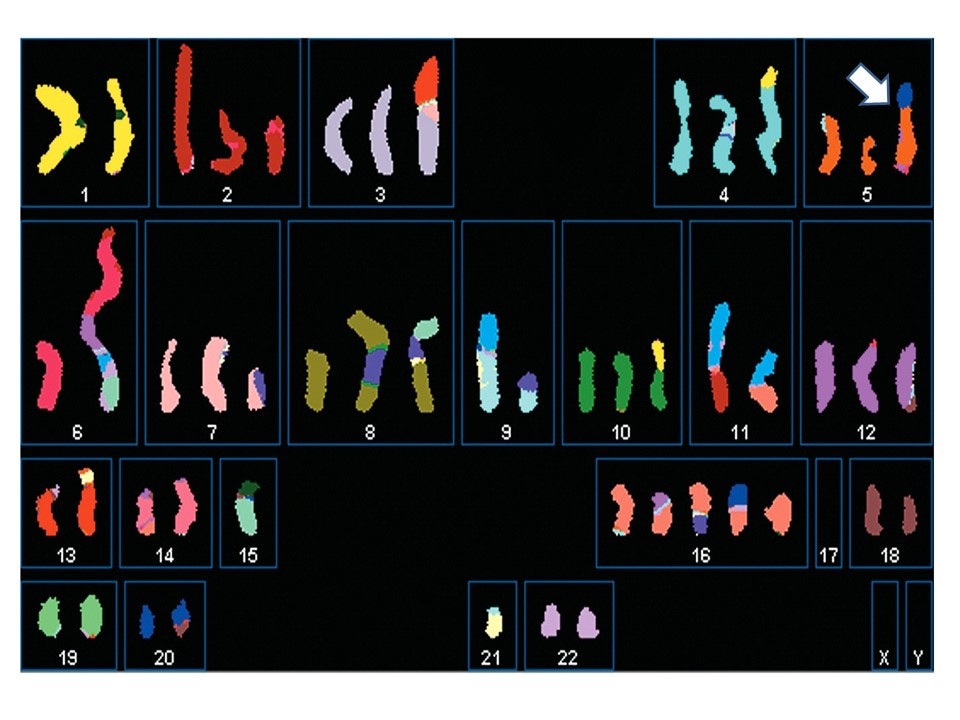}}
		\caption{Structure of a typical cancer genome (breast cancer cell line). Instead of the standard 23 pairs of chromosomes, cancer cells often harbour large structural variants, such as changes in copy number and translocations. Reconstructing this 1D map from high-throughput Hi-C or sequencing data is an important problem that motivates the definition of seriation with duplications. Figure from~\citet{Karp}.}
		\label{fig:cancer}
	\end{center}
	\vskip -0.2in
\end{figure}

Reconstructing the 1D structure of a cancer genome from experimental data is an important problem. Besides standard DNA sequencing techniques, an interesting recent development called Hi-C and based on the chromosome conformation capture (3C) technology allows to measure experimentally the frequency of physical interactions in 3D between all pairs of positions in the genome \citep{Lieberman-Aiden2009Comprehensive}. In short, if we split the full human genome into $n$ bins (of typical length $10^4-10^6$ basepairs each), an Hi-C experiment produces an $n\times n$ interaction matrix $A$ such that $A_{ij}$ is the frequency of interactions between DNA fragments in bins $i$ and $j$. Interestingly, most 3D interactions take place between DNA fragments which are on the same chromosome, and the frequency of 3D interactions tends to decrease with the distance between the fragments when they are on the same chromosome; hence Hi-C data can be used to perform genome assembly, using \eg, a seriation algorithm to obtain the layout \citep{Korbel2013Genome}.

A Hi-C experiment roughly proceeds as follows. Freeze the DNA in its current 3D conformation, and collect pairs of DNA fragments that lie close to each other in this spatial conformation.
For every such pair $(k,l)$, map each of the two fragments to a normal reference genome, providing their positions, $p_k$ and $p_l$. Add +1 to the interaction matrix entry $A_{ij}$ corresponding to the two bins $i$ and $j$ that respectively span $p_k$ and $p_l$. This process is repeated to statistically obtain an average proximity (frequency) between two bins.

Because of duplications, deletions and translocations in cancer genome, each bin (defined according to a normal reference genome) may be included in several fragments of different chromosomes in a cancer genome, and it may therefore not be possible nor relevant to order the bins. Instead, since it is possible to estimate from Hi-C data the total number of DNA copies for each bin, it makes more sense to first associate to each bin a corresponding number of fragments (e.g. two fragments per bin in a normal diploid genome), and then reconstruct an ordering of fragments into a number of chains to estimate the 1D structure of a cancer genome (Figure~\ref{fig:cancer}). 

The difficulty to apply a seriation algorithm is that Hi-C data provide {\em cumulative} information at the bin level, not at the fragment level. More precisely, if we denote $S_{kl}$ the (unobserved) frequency of interactions between fragments $k$ and $l$, respectively extracted from bins $b_i$ and $b_j$, what Hi-C measures as interactions between $b_i$ and $b_j$ is the sum of $S_{k^{\prime}l^{\prime}}$ where $k^{\prime}$ and $l^{\prime}$ are fragments contained in $b_i$ and $b_j$, respectively. This motivates the definition of the seriation with duplication problem formalized below.

\subsection{Problem setting}\label{ssec:formalizationSerDupl}
For clarity, let us begin by an example with $n=3$, $N=4$. Consider a simplified reference genome split in 3 subsequences, $g = (\varheartsuit, \Diamondblack, \clubsuit )$.
In a cancer genome, the $\varheartsuit$ sequence is duplicated and also appears at the end of the genome. Using the symbol $\heartsuit$ to denote the duplicated sequence of DNA, the cancer genome can be written $\tilde{g} = (\varheartsuit, \Diamondblack, \clubsuit, \heartsuit )$. 
The true interaction matrix between the fragments $(\varheartsuit, \Diamondblack, \clubsuit, \heartsuit )$ is a $\cR$ matrix,
\begin{align*}
S_* = 
\begin{blockarray}{ccccc}
&\varheartsuit & \Diamondblack & \clubsuit & \heartsuit \\
\begin{block}{c(cccc)}
\varheartsuit & 3 & 2 & 1 & 0\\ 
\Diamondblack & 2 & 3 & 2 & 1\\ 
\clubsuit & 1 & 2 & 3 & 2\\
\heartsuit & 0 & 1 & 2 & 3\\
\end{block}
\end{blockarray}
\end{align*}
Yet, interactions between $(\clubsuit, \heartsuit)$ and $(\clubsuit, \varheartsuit)$ are both attributed to $(\clubsuit, \varheartsuit)$ by the Hi-C experiment, resulting in the following observed interaction matrix and duplication count vector, 
\begin{align*}
A = 
\begin{blockarray}{cccc}
&\varheartsuit & \Diamondblack & \clubsuit \\
\begin{block}{c(ccc)}
\varheartsuit & 6 & 3 & 3 \\ 
 \Diamondblack & 3 & 3 & 2 \\ 
 \clubsuit & 3 & 2 & 3\\
\end{block}
\end{blockarray}
~,
\quad c = (2, 1, 1)^T.
\end{align*}
Observing $A$, the sequence we wish to reconstruct is in fact $\pi_* = (1, 2, 3, 1)^T$.

Given a matrix $A \in \symm_n$ of similarity between $n$ bins, and a vector $c \in \posints^n$ (the ``counts'' of the bins), with total $N = \sum_{i=1}^{n} c_i$, Seriation with Duplications aims at finding a sequence $\tilde{\pi} \in [1,n]^N$ of $N$ integers such that $i$ appears $c_i$ times in $\tilde{\pi}$, at positions $L_i \subset [1,N]$ with $|L_i|=c_i$, and a matrix $S \in \cR^N$ such that 
\[
A_{ij} = \sum_{k \in L_i , l \in L_j} S_{kl}\quad \mbox{ for all $i, j \in [1,n]$.}
\]
Remark that if $c = \ones_n$ (the vector of $\reals^n$ with all entries equal to $1$), the problem is equivalent to seriation and~$\tilde{\pi}$ is a permutation vector.

To represent the subsets $\{ L_i \}_{i \in [1,n]}$, we use assignment matrices $Z \in \{0,1\}^{n \times N}$ such that $Z_{ik} = 1$ iff $k \in L_i$ (as in clustering problems).
Such an assignment matrix is linked to the vector-based notation $\tilde{\pi} \in [1,n]^N$ from above through $\tilde{\pi} = Z^T (1, 2, \ldots, n)^T$.
We write $\cZ_c$ the set of assignment matrices for a given duplication count vector $c \in \posints^n$,
\begin{equation*}
\cZ_c = \left\{
Z \in \{0,1\}^{n \times N} \: \middle| \: Z \ones_N = c \: , \: Z^T \ones_n = \ones_N
\right\}
\end{equation*}
where $N = c^T \ones_n$, and the constraints indicate that each bin $i \in [1,n]$ has $c_i$ duplicates, and that each element $k \in [1,N]$ comes from one single bin.
Observe that given an initial assignment matrix $Z_0 \in \cZ_c$, any other $Z \in \cZ_c$ can be expressed as $Z_0$ whose columns have been permuted, \ie~there exists $\Pi \in \cP_N$ such that $Z = Z_0 \Pi$.
As in the \ref{eqn:seriation} formulation, the problem of Seriation with Duplications can be written
\begin{align}\tag{SD} \label{eqn:seriationdupli}
\BA{ll}
\find & \Pi \in \cP_N \: , \: S \in \cR^N \\
\st & Z_0 \Pi S \Pi^T Z_0^T = A.
\EA
\end{align}
where $Z_0$ is an initial assignment matrix. Like \ref{eqn:seriation}, \ref{eqn:seriationdupli} may not be feasible. The analog of \ref{eqn:robustseriation1} is then written
\begin{align}\tag{RSD} \label{eqn:robustseriationdupli1}
\BA{ll}
\mbox{minimize} &  \| Z_0 \Pi S \Pi^T Z_0^T - A \| \\
\st & \Pi \in \cP_N, \: S \in \cR^n.
\EA
\end{align}
Note again that if $c = \ones_n$, then $N=n$, $Z_0 = \idm_n$, and \ref{eqn:seriationdupli} (respectively \ref{eqn:robustseriationdupli1}) is equivalent to \ref{eqn:seriation} (resp. \ref{eqn:robustseriation1}).

\subsection{Algorithms}

Let us assume that we are able to project on the set of pre-strong-R matrices, that is to say, given $S$, we can compute the couple $(\Pi_{*}, S_*) \in \cP \times \cR$ that minimizes $ \| \Pi R \Pi^T - S \|$ (note that the projection on the set of pre-strong-R matrices is nothing but the \ref{eqn:robustseriation1} problem).
we can then use alternationg projections to optimize \ref{eqn:robustseriationdupli1}  (although the set of pre-strong-R matrices is not convex, so convergence to a global optimum is not garanteed). We detail this method Algorithm~\ref{alg:AltProjBase}.
\begin{algorithm}[h!]
	\caption{General Alternating Projection Scheme for Seriation with Duplications.}
	\label{alg:AltProjBase}
	\begin{algorithmic} [1]
		\REQUIRE A matrix $A \in \symm_n$, a duplication count vector $c \in \posints^n$, a maximum number of iterations $T$.
		\STATE Set $N= \sum_{i=1}^n c_i, \: Z^{(0)} \in \cZ_c$ and $S^{(0)} = Z^{(0) T} \diag(c^{-1}) A \diag(c^{-1})^T Z^{(0)}$, \ie, $S^{(0)}_{kl} = \frac{A_{ij}}{c_i c_j}$ with $k \in L_i$ and $l \in L_j$.
		\WHILE{$t \leq T$}
		\STATE Compute $(\Pi_*, S_*)$, solution of \eqref{eqn:robustseriation1} for $S^{(t)}$, and set\\
		$S^{(t+\frac{1}{2})} \gets S_*$ \\
		$Z^{(t+1)} \gets Z^{(t)} \Pi_*$ \label{lst:robustseriationinalg}
		\STATE Compute $S_A$, projection of $S^{(t+\frac{1}{2})}$ on the set of matrices that satisfy $Z^{(t+1)}  S  Z^{(t+1) T} = A$, and set\\
		$S^{(t+1)} \gets S_A$
		\STATE  $t \gets t+1$.
		\IF{$Z^{(t+1)} = Z^{(t)}$}
		\STATE {\bf break}
		\ENDIF
		\ENDWHILE
		\ENSURE A matrix $S^{(T)}$, an assignment matrix $Z^{(T)}$
	\end{algorithmic}
\end{algorithm}

In fact, we can use any method presented in Section~\ref{sec:RobustSeriation} to solve the projection step 3 in Algorithm~\ref{alg:AltProjBase}. In our experiments here, we use $\eta$-spectral and Uncons, which are the most efficient, and spectral as a baseline.
From the permutation $\Pi_*$ obtained by, \eg, solving \ref{eqn:huberSUM} with $\eta$-Spectral,
we compute $S_*$ by doing a $\ell_1$ projection of  $\Pi_* S^{(t)} \Pi_{*}^T$ onto $\cR$ through linear programming.
Indeed, the membership to $\cR$ can be described by a set of linear inequalities. We can also add upper bounds on the matrix entries belonging to a given diagonal, if we have {\it a priori} knowledge on the law by which the entries decrease when moving away from the diagonal, which is the case for Hi-C genome reconstruction. We detail these steps in the Supplementary Material, \S~\ref{sssec:dupli-detail-proj-on-strongR}.
Projecting onto the set of matrices satisfying linear equality constraints in step 4 can also be done with a convex programming solver, but the problem is actually separable on the values $(i,j) \in [1,n] \times [1,n]$ and has a closed form solution detailed in the Supplementary Material, \S~\ref{sssec:dupli-detail-proj-on-affine}.

\section{Numerical Experiments}\label{sec:experiments}
In this section, we test the algorithms detailed above on both synthetic and real data sets.
\subsection{Robust Seriation} 
\subsubsection{Synthetic data}
We performed experiments with matrices from $\cM_n(\delta, s)$ with $n=100, \:200, \:500$, $\delta=n/10, \: n/20$, and $s/s_{\text{lim}} =0.5,1,2.5,5,7.5, 10$, with $s$ is the number of out-of-band terms as in Definition~\ref{def:bd+noise-mat} and $s_{\text{lim}} = (n - \delta - 1)$ is the value appearing in Proposition~\ref{prop:RSvsR2SUM}, where \ref{eqn:robust2SUM2} and \ref{eqn:robustseriation1} coincide when $s \leq s_{\text{lim}}$.
In Table~\ref{tb:rob-ser-synth}, we show the seriation results of the different methods described in \S\ref{sec:rob-ser-algo}. When an algorithm can be used for \ref{eqn:2sum}, but also with \ref{eqn:robust2SUM2} (or \ref{eqn:huberSUM}, respectively), we pre-pend -H (or -R, resp.) to its name in the Huber (or R-2SUM, resp.) corresponding row of the Table . In Table~\ref{tb:KT-vs-s}, we show the Kendall-$\tau$ score for different values of $s/s_{\text{lim}}$.
For a given set of parameters $(n, \delta, s)$, we generated 100 experiments with random locations for the out-of-band entries. The results displayed in Tables~\ref{tb:rob-ser-synth} and \ref{tb:KT-vs-s} are averaged over these experiments, with the standard deviation given after the $\pm$ sign.
The experiments with different values of $n$ and $\delta$ exhibit similar trends, as one can see in Tables~\ref{tb:KT-vs-s-delta01n} and \ref{tb:KT-vs-s-delta005n}.
Overall, $\eta$-Spect. finds the best ordering, and is also time efficient. Uncons is also competitive.
Some methods such as HGnCR do not perform as good in average, but have a higher standard deviation over the 100 simulations. They actually perform well on most simulations, but fail on a few ones. Overall this results in a lower mean Kendall-$\tau$ score and a higher standard deviation. 

\begin{table}[t]
	\caption{Kendall-$\tau$, \ref{eqn:huberSUM}, \ref{eqn:robust2SUM2}, \ref{eqn:robustseriation1} (with Froebenius norm) scores for the different methods for $n=200$, $\delta=20$, and $s/s_{\text{lim}}=5$. The results are averaged over 100 instances of $A \in \cM_n(\delta,s)$.
		The first seven methods are used with the \ref{eqn:2sum} loss, the six middle ones with the \ref{eqn:huberSUM} loss, where $\delta$ was chosen following the rule described at the end of~\S\ref{ssec:robust2sum}, and the two last middle ones with the \ref{eqn:robust2SUM2} loss. and  Some scores are scaled to simplify the table.}
	\label{tb:rob-ser-synth}
	\begin{center}
		\begin{small}
			\begin{sc}
				\begin{tabular}{ccccccr}

					\toprule
					& \parbox{1.5cm}{Kendall-$\tau$} & \parbox{1.3cm}{Huber {$\scriptstyle \times1e-6$}} & \parbox{1.3cm}{R2SUM {$\scriptstyle \times 1e-6$}} & Dist2R & \parbox{1.3cm}{2SUM {$\scriptstyle \times 1e-6$}} & Time (s) \\ 
					\midrule 
					spectral & 0.86 {$\scriptstyle \pm 0.06$} & 7.76 {$\scriptstyle \pm 0.61$} & 2.67 {$\scriptstyle \pm 0.19$} & 73.6 {$\scriptstyle \pm 5.3$} & 7.7 {$\scriptstyle \pm 0.4$} & 3.54e-01  \\ 
					GnCR & 0.87 {$\scriptstyle \pm 0.15$} & 7.21 {$\scriptstyle \pm 0.40$} & 2.47 {$\scriptstyle \pm 0.17$} & 67.6 {$\scriptstyle \pm 4.7$} & 7.5 {$\scriptstyle \pm 0.3$} & 6.99e-01  \\ 
					FAQ & 0.89 {$\scriptstyle \pm 0.08$} & 7.19 {$\scriptstyle \pm 0.31$} & 2.46 {$\scriptstyle \pm 0.14$} & 67.6 {$\scriptstyle \pm 4.1$} & {\bf 7.4 {$\scriptstyle \pm 0.2$}} & 3.37e+00  \\ 
					LWCD & 0.89 {$\scriptstyle \pm 0.08$} & 7.18 {$\scriptstyle \pm 0.30$} & 2.46 {$\scriptstyle \pm 0.14$} & 67.5 {$\scriptstyle \pm 3.9$} & {\bf7.4 {$\scriptstyle \pm 0.2$} } & 2.99e+00  \\ 
					UBI & 0.89 {$\scriptstyle \pm 0.06$} & 7.32 {$\scriptstyle \pm 0.31$} & 2.52 {$\scriptstyle \pm 0.12$} & 69.5 {$\scriptstyle \pm 3.3$} & 7.5 {$\scriptstyle \pm 0.2$} & 1.45e+00  \\ 
					Manopt & 0.86 {$\scriptstyle \pm 0.06$} & 7.72 {$\scriptstyle \pm 0.58$} & 2.66 {$\scriptstyle \pm 0.18$} & 73.2 {$\scriptstyle \pm 5.2$} & 7.6 {$\scriptstyle \pm 0.4$} & 3.90e+00  \\ 
					\midrule 
					$\eta$-Spectral & {\bf 0.97 {$\scriptstyle \pm 0.00$} } & {\bf 6.74 {$\scriptstyle \pm 0.13$} } & 2.03 {$\scriptstyle \pm 0.02$} & 50.8 {$\scriptstyle \pm 0.8$} & 7.6 {$\scriptstyle \pm 0.2$} & 1.07e+00  \\ 
					HGnCR & 0.89 {$\scriptstyle \pm 0.22$} & 6.91 {$\scriptstyle \pm 0.52$} & 2.11 {$\scriptstyle \pm 0.26$} & 53.6 {$\scriptstyle \pm 8.6$} & 7.7 {$\scriptstyle \pm 0.4$} & 9.06e+00  \\ 
					H-FAQ & 0.95 {$\scriptstyle \pm 0.08$} & 6.84 {$\scriptstyle \pm 0.32$} & 2.01 {$\scriptstyle \pm 0.08$} & 49.0 {$\scriptstyle \pm 3.9$} & 7.7 {$\scriptstyle \pm 0.3$} & 4.28e-01  \\ 
					H-LWCD & 0.94 {$\scriptstyle \pm 0.09$} & 6.88 {$\scriptstyle \pm 0.34$} & 2.03 {$\scriptstyle \pm 0.11$} & 49.7 {$\scriptstyle \pm 5.0$} & 7.7 {$\scriptstyle \pm 0.3$} & 3.00e+00  \\ 
					H-UBI & {\bf 0.97 {$\scriptstyle \pm 0.00$} } & {\bf 6.74 {$\scriptstyle \pm 0.13$} } & 2.05 {$\scriptstyle \pm 0.02$} & 51.4 {$\scriptstyle \pm 1.1$} & 7.6 {$\scriptstyle \pm 0.2$} & 3.08e+00  \\ 
					H-Manopt & 0.92 {$\scriptstyle \pm 0.06$} & 7.05 {$\scriptstyle \pm 0.39$} & 2.26 {$\scriptstyle \pm 0.15$} & 59.7 {$\scriptstyle \pm 5.2$} & 7.6 {$\scriptstyle \pm 0.3$} & 9.22e+00  \\ 
					\midrule 
					R-FAQ & 0.95 {$\scriptstyle \pm 0.10$} & 6.97 {$\scriptstyle \pm 0.40$} & {\bf 1.99 {$\scriptstyle \pm 0.08$} } & {\bf 44.9 {$\scriptstyle \pm 4.3$} } & 7.9 {$\scriptstyle \pm 0.4$} & 3.39e-01  \\ 
					R-LWCD & 0.94 {$\scriptstyle \pm 0.09$} & 7.03 {$\scriptstyle \pm 0.42$} & 2.01 {$\scriptstyle \pm 0.09$} & 46.0 {$\scriptstyle \pm 4.8$} & 8.0 {$\scriptstyle \pm 0.4$} & 3.32e+00  \\ 
					
					\bottomrule
				\end{tabular}
			\end{sc}
		\end{small}
	\end{center}
	\vskip -0.1in
\end{table}

\begin{table}[t]
	\caption{Kendall-$\tau$ score for different values of $s/s_{\text{lim}}$, for the same methods as in Table~\ref{tb:rob-ser-synth}, and $n=200$, $\delta=20$.}
	\label{tb:KT-vs-s}
	\vskip .15in
	\begin{center}
		\begin{small}
			\begin{sc}
				\begin{tabular}{lcccccr}

					\toprule
					& $s/s_{\text{lim}}=0.5$ & $s/s_{\text{lim}}=1$ & $s/s_{\text{lim}}=2.5$ & $s/s_{\text{lim}}=5$ & $s/s_{\text{lim}}=7.5$ & $s/s_{\text{lim}}=10$ \\ 
					\midrule 
					spectral  & 0.96 {$\scriptstyle \pm 0.01$} & 0.95 {$\scriptstyle \pm 0.01$} & 0.91 {$\scriptstyle \pm 0.03$} & 0.86 {$\scriptstyle \pm 0.06$} & 0.84 {$\scriptstyle \pm 0.06$} & 0.80 {$\scriptstyle \pm 0.09$}\\
					GnCR  & 0.98 {$\scriptstyle \pm 0.00$} & 0.96 {$\scriptstyle \pm 0.04$} & 0.93 {$\scriptstyle \pm 0.07$} & 0.87 {$\scriptstyle \pm 0.15$} & 0.81 {$\scriptstyle \pm 0.20$} & 0.80 {$\scriptstyle \pm 0.18$}\\
					FAQ  & 0.98 {$\scriptstyle \pm 0.00$} & 0.97 {$\scriptstyle \pm 0.00$} & 0.94 {$\scriptstyle \pm 0.02$} & 0.89 {$\scriptstyle \pm 0.08$} & 0.87 {$\scriptstyle \pm 0.08$} & 0.82 {$\scriptstyle \pm 0.13$}\\
					LWCD  & 0.98 {$\scriptstyle \pm 0.00$} & 0.97 {$\scriptstyle \pm 0.00$} & 0.94 {$\scriptstyle \pm 0.02$} & 0.89 {$\scriptstyle \pm 0.08$} & 0.87 {$\scriptstyle \pm 0.08$} & 0.82 {$\scriptstyle \pm 0.13$}\\
					UBI  & 0.97 {$\scriptstyle \pm 0.00$} & 0.96 {$\scriptstyle \pm 0.01$} & 0.92 {$\scriptstyle \pm 0.03$} & 0.89 {$\scriptstyle \pm 0.06$} & 0.86 {$\scriptstyle \pm 0.07$} & 0.82 {$\scriptstyle \pm 0.12$}\\
					Manopt  & 0.97 {$\scriptstyle \pm 0.00$} & 0.95 {$\scriptstyle \pm 0.01$} & 0.91 {$\scriptstyle \pm 0.03$} & 0.86 {$\scriptstyle \pm 0.06$} & 0.84 {$\scriptstyle \pm 0.06$} & 0.80 {$\scriptstyle \pm 0.09$}\\
					\midrule 
					$\eta$-Spectral  & 0.99 {$\scriptstyle \pm 0.00$} & 0.99 {$\scriptstyle \pm 0.00$} & 0.98 {$\scriptstyle \pm 0.00$} & {\bf 0.97 {$\scriptstyle \pm 0.00$}} & {\bf 0.96 {$\scriptstyle \pm 0.00$}} & {\bf 0.94 {$\scriptstyle \pm 0.06$}}\\
					HGnCR  & {\bf 1.00 {$\scriptstyle \pm 0.00$}} & 0.99 {$\scriptstyle \pm 0.00$} & {\bf 0.99 {$\scriptstyle \pm 0.00$}} & 0.89 {$\scriptstyle \pm 0.22$} & 0.85 {$\scriptstyle \pm 0.23$} & 0.83 {$\scriptstyle \pm 0.25$}\\
					H-FAQ  & {\bf 1.00 {$\scriptstyle \pm 0.00$}} & {\bf1.00 {$\scriptstyle \pm 0.00$} } & 0.99 {$\scriptstyle \pm 0.01$} & 0.95 {$\scriptstyle \pm 0.08$} & 0.94 {$\scriptstyle \pm 0.09$} & 0.91 {$\scriptstyle \pm 0.13$}\\
					H-LWCD  & {\bf 1.00 {$\scriptstyle \pm 0.00$} } & {\bf 1.00 {$\scriptstyle \pm 0.00$}} & 0.99 {$\scriptstyle \pm 0.02$} & 0.94 {$\scriptstyle \pm 0.09$} & 0.94 {$\scriptstyle \pm 0.09$} & 0.90 {$\scriptstyle \pm 0.14$}\\
					H-UBI  & 0.99 {$\scriptstyle \pm 0.00$} & 0.99 {$\scriptstyle \pm 0.00$} & 0.98 {$\scriptstyle \pm 0.00$} & {\bf 0.97 {$\scriptstyle \pm 0.00$}} & {\bf 0.96 {$\scriptstyle \pm 0.01$}} & {\bf 0.94 {$\scriptstyle \pm 0.03$}}\\
					H-Manopt  & {\bf 1.00 {$\scriptstyle \pm 0.00$}} & 0.99 {$\scriptstyle \pm 0.00$} & 0.97 {$\scriptstyle \pm 0.02$} & 0.92 {$\scriptstyle \pm 0.06$} & 0.89 {$\scriptstyle \pm 0.07$} & 0.84 {$\scriptstyle \pm 0.10$}\\
					\midrule 
					R-FAQ  & {\bf 1.00 {$\scriptstyle \pm 0.00$} } & {\bf 1.00 {$\scriptstyle \pm 0.00$}} & 0.99 {$\scriptstyle \pm 0.04$} & 0.95 {$\scriptstyle \pm 0.10$} & 0.94 {$\scriptstyle \pm 0.10$} & 0.90 {$\scriptstyle \pm 0.15$}\\
					R-LWCD  & 0.99 {$\scriptstyle \pm 0.00$} & {\bf 1.00 {$\scriptstyle \pm 0.00$} } & 0.99 {$\scriptstyle \pm 0.04$} & 0.94 {$\scriptstyle \pm 0.09$} & 0.94 {$\scriptstyle \pm 0.10$} & 0.90 {$\scriptstyle \pm 0.16$}\\
					
					\bottomrule
				\end{tabular}
			\end{sc}
		\end{small}
	\end{center}
	\vskip -.1in
\end{table}

\subsubsection{{\it E. coli} genome reconstruction}\label{sssec:ecoli}
We performed experiments with real data from {\it Escherichia coli} reads sequenced by \citet{Loman15} with the Oxford Nanopore Technology MinION's device. We used the minimap \citep{Li:Miniasm} tool that measures the overlap between reads to build the similarity matrix.
The resulting matrix is sparse and of size $n \simeq 10^4$. Removing its smallest values with a threshold can help in the task of reordering it \citep{recanati2016spectral}. We performed a grid search with 24 threshold values linearly spaced in a reasonable range (500-700) that kept the whole matrix connected. For each of them, we compute $\sqrt{\lambda}=\delta$ from the number of non-zero entries of the matrix as explained in~\ref{ssec:robust2sum}. We only show the results yielding the best \ref{eqn:robust2SUM2} score in Figure~\ref{fig:EcoliOrdVsTrue}.

We only tested the $\eta$-Spectral and UnCons methods since they exhibit the best performance on synthetic data and are scalable (one boils down to computing the highest eigenvalues of a sparse matrix, the other follows a first order unconstrained optimization scheme).
The UnCons method failed to find an approximately correct ordering.
The $\eta$-Spectral method outperformed UnCons in the \ref{eqn:robust2SUM2} score sense, and successfully reordered the genome. Nevertheless, it is quite sensitive to variations in the similarity matrix. Among the 24 threshold parameters used in the grid-search, only a few (the ones with lowest \ref{eqn:robust2SUM2} score) yielded a correct ordering.

Figure~\ref{fig:EcoliOrdVsTrue} displays the ordering found by our method versus one given by mapping the reads to a reference genome with the BWA sequence alignment tool \citep{LiBWA}.
Some reads could not be aligned to the reference genome with BWA, probably due to high sequencing errors. For these reads, the software outputs an ``infinite'' position. This results in the cloud of points located at the top of the figure.
Observe also a small zone where the ordering is reversed, which is zoomed-in in Figure~\ref{subfig:EcoliWrongContig}.

\begin{figure}[ht]
	\begin{center}
		\centerline{\includegraphics[width=0.5\columnwidth]{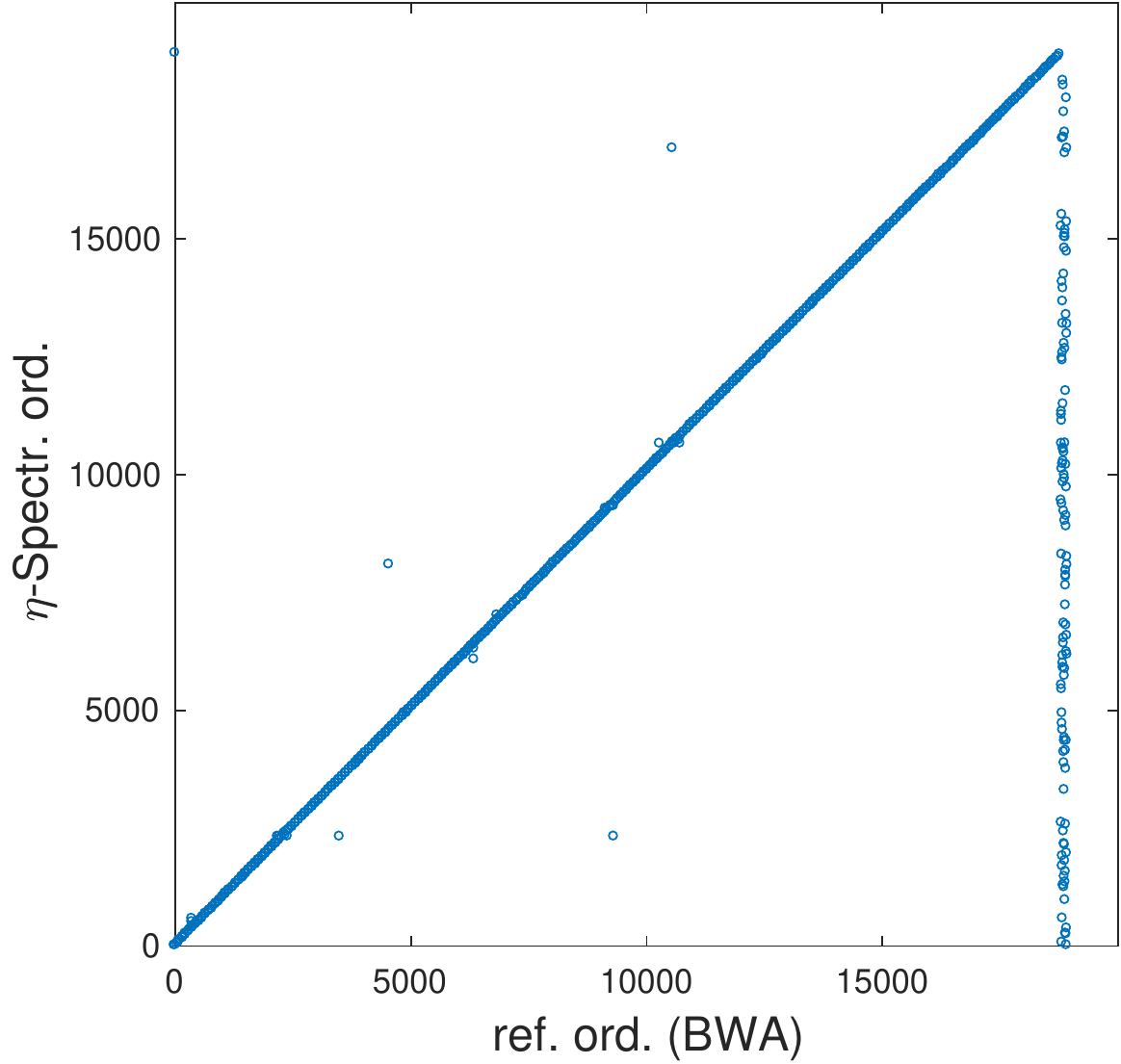}}
		\caption{Ordering found by $\eta$-Spectral \textit{vs} true ordering (obtained by mapping the \textit{reads} to a reference genome with BWA.}
		\label{fig:EcoliOrdVsTrue}
	\end{center}
	\vskip -0.2in
\end{figure}

\subsection{Seriation with Duplication}\label{ssec:exps-ser-dupli}
We performed synthetic experiments in which we generate the data as follows. We first build a strong-R matrix $S$ of size $N$, and a random duplication count vector $c \in \posints^n$ such that $N = \sum_{i=1}^{n} c_i$. We generate a random assignment matrix $Z \in \cZ_c$, and the corresponding observed matrix $A = Z S Z^T$.
We then test Algorithm~\ref{alg:AltProjBase} by providing it with $A$ and $c$ and comparing its output $Z^{{\text out}}$ and $S^{{\text out}}$ to the ground truth.

Specifically, we compute the relative Froebenius distance between $S$ and $S^{{\text out}}$, ${\text d2R} = \|S - S^{{\text out}}\|_{F} / \|S\|_F$, and we compute a distance between the assignment matrices as follows.
For a given bin index $i \in [1,n]$ (\ie~a row $Z_i$), there are $c_i$ locations for the non-zeros of the $i$-th row of $Z$ and of $Z^{{\text out}}$ (which can also be viewed as two subsets $L_i$ and $L_i^{{\text out}}$ of $[1,N]$). To compute the distance between these positions, we first compute a matching between the elements of $L_i$ and $L_i^{{\text out}}$ using the Hungarian algorithm \citep{kuhn1955hungarian}.
Then, we compute the distance between each matched pair of elements $(k, k^{{\text out}}) \in L_i \times L_i^{{\text out}}$, and store the average distance between matching pairs for row $i$. Supplementary Figures~\ref{fig:matchingHungarian} and~\ref{fig:ZvsZtrue} illustrates this process. The average over all rows of this average distance is given in Table~\ref{tb:SerDupliSparseBD40} as meanDist, and we also provide its standard deviation and median.

In the experiments, we built dense strong-R, Toeplitz matrices $S$ where the entries follow a power law of the distance to the diagonal, $S_{kl} = |k-l|^{-\gamma}$, which is consistent with the observed frequency of intra-chromosomal interactions \citep{Lieberman2009Comprehensive}. We used $N=200$ and tried several values for the exponent $\gamma$ and the ratio $N/n$, namely $\gamma \in \{ 0.1, \: 0.5, \:1\}$ and $N/n \in \{1.33, \: 2, \: 4 \}$. The results are shown in Tables~\ref{tb:SerDupliDense}, \ref{tb:SerDupliDenseNoise5}, \ref{tb:SerDupliDenseNoise10}, and some qualitative results are shown in Figure~\ref{fig:SerDuplDenseQualitResults}.
We also conducted experiments 
with sparse, band matrices $S \in \cM_N(\delta,s)$ as in Section~\ref{sec:RobustSeriation}. The results are shown in Tables~\ref{tb:SerDupliSparseBD40}, \ref{tb:SerDupliSparseBD40-full}, \ref{tb:SerDupliSparseBD20-full}, and some qualitative results are shown in Figure~\ref{fig:SerDuplSparseQualitResults}.
The $\eta$-Spectral method works best for dense matrices, and is outperformed by H-UBI for maatrices in $\cM_N(\delta,s)$.
We observe that, as expected, the recovered assignment $Z^{\text{out}}$ is closer to $Z$ when $N/n$ is smaller. However, the D2S scores and the qualitative Figures~\ref{fig:SerDuplDenseQualitResults} and \ref{fig:SerDuplSparseQualitResults} suggest that for large $N/n$, the recovered matrix $S^{\text{out}}$ may be close to $S$ although the assignment is not well recovered. Intuitively, this means the problem is degenerate, with several assignment matrices roughly leading to the same matrix $S$, and the algorithm finds one of these.

\begin{table}[t]
	\caption{
		Results of synthetic experiments for Seriation with Duplications from matrices $S \in \cM_N(\delta,s)$ with $n=200$, $\delta=n/5$, $s=0$, and various values of $N/n$, where the \eqref{eqn:robustseriation1} problem is tackled with either Spectral, $\eta$-Spectral or H-UBI within Algorithm~\ref{alg:AltProjBase}.
		From the output $S^{{\text out}}$ and $Z^{{\text out}}$ of Algorithm~\ref{alg:AltProjBase} and the ground truth $S$ and $Z$ from which the data $A$ is generated,
		D2S is the relative Froebenius distance between $S$ and $S^{{\text out}}$,
		Huber is the \eqref{eqn:huberSUM} loss on $S$,
		meanDist, stdDist and medianDist are the average, standard deviation and median of the distance between the positions assigned to a index $k$ by $Z$ and $Z^{out}$ (see main text for details).
		Time is the amount of CPU time elapsed until convergence of Algorithm~\ref{alg:AltProjBase}.
		}
	\label{tb:SerDupliSparseBD40}
	\begin{center}
		\begin{small}
			\begin{sc}
				\begin{tabular}{llccccc}
					\toprule
	$N/n$ & method & d2S & Huber  & meanDist & stdDist &  Time  \\ 
	      &        &     & {$\scriptstyle(\times1e-7)$} &        &       & {$\scriptstyle(\times1e-3s)$}\\
	\midrule 
 \multirow{3}*{\parbox{1.3cm}{$1.33$}}
 & spectral & 0.53 {$\scriptstyle \pm0.08 $} & 1.67 {$\scriptstyle \pm0.33 $} & 11.8 {$\scriptstyle \pm3.5 $} & 13.2 {$\scriptstyle \pm1.7 $} &  7.45 {$\scriptstyle \pm4.08 $} \\ 
 & $\eta$-Spectral & 0.12 {$\scriptstyle \pm0.06 $} & 0.76 {$\scriptstyle \pm0.06 $} & 0.8 {$\scriptstyle \pm0.8 $} & 2.4 {$\scriptstyle \pm2.2 $} & 2.85 {$\scriptstyle \pm1.78 $} \\ 
 & H-UBI & {\bf 0.09 {$\scriptstyle \pm0.06 $}} & {\bf 0.74 {$\scriptstyle \pm0.05 $}} & {\bf 0.6 {$\scriptstyle \pm0.6 $}} & 1.8 {$\scriptstyle \pm1.9 $} &  3.99 {$\scriptstyle \pm2.76 $} \\ 
 \midrule 
 \multirow{3}*{\parbox{1.3cm}{$2$}}
 & spectral & 0.38 {$\scriptstyle \pm0.05 $} & 1.48 {$\scriptstyle \pm0.26 $} & 10.3 {$\scriptstyle \pm4.2 $} & 10.5 {$\scriptstyle \pm2.8 $} &  1.30 {$\scriptstyle \pm0.25 $} \\ 
 & $\eta$-Spectral & 0.21 {$\scriptstyle \pm0.04 $} & 0.99 {$\scriptstyle \pm0.12 $} & 4.1 {$\scriptstyle \pm4.1 $} & 6.9 {$\scriptstyle \pm3.9 $} &  0.50 {$\scriptstyle \pm0.19 $} \\ 
 & H-UBI & {\bf 0.19 {$\scriptstyle \pm0.05 $}} & {\bf 0.96 {$\scriptstyle \pm0.14 $}} & {\bf 4.0 {$\scriptstyle \pm5.8 $}} & 6.2 {$\scriptstyle \pm4.6 $} &  0.79 {$\scriptstyle \pm0.31 $} \\ 
 \midrule
 \multirow{3}*{\parbox{1.3cm}{$4$}}
 & spectral & 0.29 {$\scriptstyle \pm0.02 $} & 1.45 {$\scriptstyle \pm0.09 $} & 18.4 {$\scriptstyle \pm4.5 $} & 11.8 {$\scriptstyle \pm3.1 $} &  1.34 {$\scriptstyle \pm0.23 $} \\ 
 & $\eta$-Spectral & 0.22 {$\scriptstyle \pm0.02 $} & 1.29 {$\scriptstyle \pm0.06 $} & 16.3 {$\scriptstyle \pm6.8 $} & 12.2 {$\scriptstyle \pm5.1 $} &  0.61 {$\scriptstyle \pm0.14 $} \\ 
 & H-UBI & {\bf 0.22 {$\scriptstyle \pm0.02 $}} & {\bf 1.26 {$\scriptstyle \pm0.06 $}} & {\bf 15.9 {$\scriptstyle \pm7.2 $}} & 12.0 {$\scriptstyle \pm5.6 $} &  0.91 {$\scriptstyle \pm0.25 $} \\

					\bottomrule
				\end{tabular}
			\end{sc}
		\end{small}
	\end{center}
	\vskip -.1in
\end{table}

\begin{table}[t]
	\caption{
		Results of synthetic experiments for Seriation with Duplications from dense, strong-R matrices of size $n=200$, with the same metrics and methods as in Table~\ref{tb:SerDupliSparseBD40}, with $\gamma=0.5$.
	}
	\label{tb:SerDupliDense}
	\begin{center}
		\begin{small}
			\begin{sc}
				\begin{tabular}{llcccccc}
					\toprule
					$N/n$ & method & d2S & Huber  & meanDist & stdDist & Time \\ 
					    &      &   & {$\scriptstyle(\times 1e-7)$}&       &      & {$\scriptstyle(\times 1e-2s)$}\\
				    \midrule
					\multirow{3}*{\parbox{1.3cm}{$1.33$}}
					& spectral & 0.25 {$\scriptstyle \pm0.04$} & 1.36 {$\scriptstyle \pm0.03$} & 6.1 {$\scriptstyle \pm1.8$} & 7.9 {$\scriptstyle \pm1.6$} & 8.74 {$\scriptstyle \pm4.85$} \\ 
					& $\eta$-Spectral & {\bf 0.15 {$\scriptstyle \pm0.02$}} & {\bf 1.30 {$\scriptstyle \pm0.01$}} & {\bf 2.2 {$\scriptstyle \pm0.7$}} & 3.7 {$\scriptstyle \pm1.1$} & 6.12 {$\scriptstyle \pm4.84$} \\ 
					& H-UBI & 0.24 {$\scriptstyle \pm0.04$} & 1.35 {$\scriptstyle \pm0.03$} & 5.5 {$\scriptstyle \pm1.6$} & 7.3 {$\scriptstyle \pm1.4$} & 11.06 {$\scriptstyle \pm7.56$} \\ 
					
					\midrule
					\multirow{3}*{\parbox{1.3cm}{$2$}} 
					& spectral & 0.27 {$\scriptstyle \pm0.02$} & 1.41 {$\scriptstyle \pm0.02$} & 9.5 {$\scriptstyle \pm1.6$} & 8.4 {$\scriptstyle \pm1.3$} & 7.47 {$\scriptstyle \pm3.20$} \\ 
					& $\eta$-Spectral & {\bf 0.22 {$\scriptstyle \pm0.02$}} & {\bf 1.37 {$\scriptstyle \pm0.02$}} & {\bf 6.6 {$\scriptstyle \pm1.5$}} & 6.7 {$\scriptstyle \pm1.9$} & 7.89 {$\scriptstyle \pm3.89$} \\ 
					& H-UBI & 0.26 {$\scriptstyle \pm0.02$} & 1.40 {$\scriptstyle \pm0.02$} & 9.0 {$\scriptstyle \pm1.5$} & 8.1 {$\scriptstyle \pm1.2$} & 10.09 {$\scriptstyle \pm4.90$} \\ 
					
					\midrule
					\multirow{3}*{\parbox{1.3cm}{$4$}}
					& spectral & 0.18 {$\scriptstyle \pm0.01$} & 1.35 {$\scriptstyle \pm0.01$} & 14.4 {$\scriptstyle \pm2.8$} & 8.7 {$\scriptstyle \pm2.7$} & 6.53 {$\scriptstyle \pm1.90$} \\ 
					& $\eta$-Spectral & {\bf 0.18 {$\scriptstyle \pm0.01$}} & {\bf 1.35 {$\scriptstyle \pm0.01$}} & {\bf 14.3 {$\scriptstyle \pm2.9$}} & 8.9 {$\scriptstyle \pm2.9$} & 7.59 {$\scriptstyle \pm2.28$} \\ 
					& H-UBI & 0.19 {$\scriptstyle \pm0.01$} & 1.35 {$\scriptstyle \pm0.01$} & 14.8 {$\scriptstyle \pm2.5$} & 8.8 {$\scriptstyle \pm2.1$} & 8.62 {$\scriptstyle \pm2.46$} \\ 

					\bottomrule
				\end{tabular}
			\end{sc}
		\end{small}
	\end{center}
	\vskip -.1in
\end{table}

\section{Conclusion}
We introduced the Robust Seriation problem, which arises in \eg~\textit{de novo} genome assembly. We show that for a class of similarity matrices modeling those observed in genome assembly, the problem of Robust Seriation is equivalent to a modified 2-SUM problem. This modified problem can be relaxed, with an objective function using a Huber loss instead of the squared loss present in 2-SUM. We adapt several relaxations of permutation problems to this 2-SUM problem with Huber loss and also introduce new relaxations, including the $\eta$-Spectral method, which is computationally efficient and performs well in our experiments. Notably, it successfully reorders a bacterial genome from third generation sequencing data.
Finally, we introduced the framework of Seriation with Duplications, which is a generalization of Robust Seriation, with the aim of analyzing Hi-C data from cancer genomes. We show promising qualitative results on synthetic data.

\section*{Acknowledgements}
AA is at CNRS, d\'epartement d'informatique, \'Ecole normale sup\'erieure, UMR CNRS 8548, 45 rue d'Ulm 75005 Paris, France,  INRIA and PSL Research University. The authors would like to acknowledge support from the {\em data science} joint research initiative with the {\em fonds AXA pour la recherche} and Kamet Ventures.
AR would like to thank Vincent Roulet for fruitful discussions.

{\small \bibliographystyle{plainnat}
\bibsep 1ex
\bibliography{RobustSeriationDuplications.bib}}

\begin{thebibliography}{41}
\providecommand{\natexlab}[1]{#1}
\providecommand{\url}[1]{\texttt{#1}}
\expandafter\ifx\csname urlstyle\endcsname\relax
  \providecommand{\doi}[1]{doi: #1}\else
  \providecommand{\doi}{doi: \begingroup \urlstyle{rm}\Url}\fi

\bibitem[Atkins and Middendorf(1996)]{atkins1996physical}
Jonathan~E Atkins and Martin Middendorf.
\newblock On physical mapping and the consecutive ones property for sparse
  matrices.
\newblock \emph{Discrete Applied Mathematics}, 71\penalty0 (1-3):\penalty0
  23--40, 1996.

\bibitem[Atkins et~al.(1998)Atkins, Boman, and Hendrickson]{Atkins}
Jonathan~E Atkins, Erik~G Boman, and Bruce Hendrickson.
\newblock A spectral algorithm for seriation and the consecutive ones problem.
\newblock \emph{SIAM Journal on Computing}, 28\penalty0 (1):\penalty0 297--310,
  1998.

\bibitem[Barnard et~al.(1995)Barnard, Pothen, and Simon]{Barn95}
Stephen~T Barnard, Alex Pothen, and Horst Simon.
\newblock A spectral algorithm for envelope reduction of sparse matrices.
\newblock \emph{Numerical linear algebra with applications}, 2\penalty0
  (4):\penalty0 317--334, 1995.

\bibitem[Boumal et~al.(2014)Boumal, Mishra, Absil, and
  Sepulchre]{boumal2014manopt}
Nicolas Boumal, Bamdev Mishra, P-A Absil, and Rodolphe Sepulchre.
\newblock Manopt, a matlab toolbox for optimization on manifolds.
\newblock \emph{The Journal of Machine Learning Research}, 15\penalty0
  (1):\penalty0 1455--1459, 2014.

\bibitem[Cheema et~al.(2010)Cheema, Ellis, and Dicks]{cheema2010thread}
Jitender Cheema, TH~Noel Ellis, and Jo~Dicks.
\newblock Thread mapper studio: a novel, visual web server for the estimation
  of genetic linkage maps.
\newblock \emph{Nucleic acids research}, 38\penalty0 (suppl\_2):\penalty0
  W188--W193, 2010.

\bibitem[Dudchenko et~al.(2017)Dudchenko, Batra, Omer, Nyquist, Hoeger, Durand,
  Shamim, Machol, Lander, Aiden, et~al.]{dudchenko2017novo}
Olga Dudchenko, Sanjit~S Batra, Arina~D Omer, Sarah~K Nyquist, Marie Hoeger,
  Neva~C Durand, Muhammad~S Shamim, Ido Machol, Eric~S Lander, Aviva~Presser
  Aiden, et~al.
\newblock De novo assembly of the aedes aegypti genome using hi-c yields
  chromosome-length scaffolds.
\newblock \emph{Science}, 356\penalty0 (6333):\penalty0 92--95, 2017.

\bibitem[Evangelopoulos et~al.(2017{\natexlab{a}})Evangelopoulos, Brockmeier,
  Mu, and Goulermas]{evangelopoulos2017graduated}
Xenophon Evangelopoulos, Austin~J Brockmeier, Tingting Mu, and John~Y
  Goulermas.
\newblock A graduated non-convexity relaxation for large scale seriation.
\newblock In \emph{Proceedings of the 2017 SIAM International Conference on
  Data Mining}, pages 462--470. SIAM, 2017{\natexlab{a}}.

\bibitem[Evangelopoulos et~al.(2017{\natexlab{b}})Evangelopoulos, Brockmeier,
  Mu, and Goulermas]{evangelopoulos2017unpublished}
Xenophon Evangelopoulos, Austin~J Brockmeier, Tingting Mu, and John~Y
  Goulermas.
\newblock Approximation methods for large scale object sequencing.
\newblock \emph{Submitted to Machine Learning (unpublished yet)},
  2017{\natexlab{b}}.

\bibitem[Fogel et~al.(2013)Fogel, Jenatton, Bach, and d'Aspremont]{Fogel}
Fajwel Fogel, Rodolphe Jenatton, Francis Bach, and Alexandre d'Aspremont.
\newblock Convex relaxations for permutation problems.
\newblock In \emph{Advances in Neural Information Processing Systems}, pages
  1016--1024, 2013.

\bibitem[Garriga et~al.(2011)Garriga, Junttila, and Mannila]{Garr11}
Gemma~C Garriga, Esa Junttila, and Heikki Mannila.
\newblock Banded structure in binary matrices.
\newblock \emph{Knowledge and information systems}, 28\penalty0 (1):\penalty0
  197--226, 2011.

\bibitem[George and Pothen(1997)]{George:1997aa}
Alan George and Alex Pothen.
\newblock An analysis of spectral envelope reduction via quadratic assignment
  problems.
\newblock \emph{SIAM Journal on Matrix Analysis and Applications}, 18\penalty0
  (3):\penalty0 706--732, 1997.

\bibitem[Goemans(2015)]{Goem09}
Michel~X Goemans.
\newblock Smallest compact formulation for the permutahedron.
\newblock \emph{Mathematical Programming}, 153\penalty0 (1):\penalty0 5--11,
  2015.

\bibitem[Hahsler et~al.(2008)Hahsler, Hornik, and Buchta]{hahsler2008getting}
Michael Hahsler, Kurt Hornik, and Christian Buchta.
\newblock Getting things in order: an introduction to the r package seriation.
\newblock \emph{Journal of Statistical Software}, 25\penalty0 (3):\penalty0
  1--34, 2008.

\bibitem[Jones et~al.(2012)Jones, Rajaraman, Tannier, and
  Chauve]{jones2012anges}
Bradley~R Jones, Ashok Rajaraman, Eric Tannier, and Cedric Chauve.
\newblock Anges: reconstructing ancestral genomes maps.
\newblock \emph{Bioinformatics}, 28\penalty0 (18):\penalty0 2388--2390, 2012.

\bibitem[Karp et~al.(2015)Karp, Iwasa, and Marshall]{Karp}
Gerald Karp, Janet Iwasa, and Wallace Marshall.
\newblock \emph{Cell and Molecular Biology: Concepts and Experiments}.
\newblock Wiley, 2015.

\bibitem[Koopmans and Beckmann(1957)]{koopmans1957QAP}
Tjalling~C Koopmans and Martin Beckmann.
\newblock Assignment problems and the location of economic activities.
\newblock \emph{Econometrica: journal of the Econometric Society}, pages 53--76

\bibitem[Korbel and Lee(2013)]{Korbel2013Genome}
Jan~O Korbel and Charles Lee.
\newblock Genome assembly and haplotyping with hi-c.
\newblock \emph{Nature biotechnology}, 31\penalty0 (12):\penalty0 1099, 2013.

\bibitem[Koren and Phillippy(2015)]{KorenOneChr}
Sergey Koren and Adam~M Phillippy.
\newblock One chromosome, one contig: complete microbial genomes from long-read
  sequencing and assembly.
\newblock \emph{Current opinion in microbiology}, 23:\penalty0 110--120, 2015.

\bibitem[Koren et~al.(2017)Koren, Walenz, Berlin, Miller, Bergman, and
  Phillippy]{koren2017canu}
Sergey Koren, Brian~P Walenz, Konstantin Berlin, Jason~R Miller, Nicholas~H
  Bergman, and Adam~M Phillippy.
\newblock Canu: scalable and accurate long-read assembly via adaptive k-mer
  weighting and repeat separation.
\newblock \emph{Genome research}, 27\penalty0 (5):\penalty0 722--736, 2017.

\bibitem[Kuhn(1955)]{kuhn1955hungarian}
Harold~W Kuhn.
\newblock The hungarian method for the assignment problem.
\newblock \emph{Naval Research Logistics (NRL)}, 2\penalty0 (1-2):\penalty0
  83--97, 1955.

\bibitem[Lacoste-Julien and Jaggi(2015)]{lacoste2015global}
Simon Lacoste-Julien and Martin Jaggi.
\newblock On the global linear convergence of frank-wolfe optimization
  variants.
\newblock In \emph{Advances in Neural Information Processing Systems}, pages
  496--504, 2015.

\bibitem[Laurent and Seminaroti(2015)]{laurent2015QAPsolvable}
Monique Laurent and Matteo Seminaroti.
\newblock The quadratic assignment problem is easy for robinsonian matrices
  with toeplitz structure.
\newblock \emph{Operations Research Letters}, 43\penalty0 (1):\penalty0
  103--109, 2015.

\bibitem[Li(2016)]{Li:Miniasm}
Heng Li.
\newblock Minimap and miniasm: fast mapping and de novo assembly for noisy long
  sequences.
\newblock \emph{Bioinformatics}, 32\penalty0 (14):\penalty0 2103--2110, 2016.

\bibitem[Li and Durbin(2010)]{LiBWA}
Heng Li and Richard Durbin.
\newblock Fast and accurate long-read alignment with burrows--wheeler
  transform.
\newblock \emph{Bioinformatics}, 26\penalty0 (5):\penalty0 589--595, 2010.

\bibitem[Lieberman-Aiden et~al.(2009{\natexlab{a}})Lieberman-Aiden, van Berkum,
  Williams, Imakaev, Ragoczy, Telling, Amit, Lajoie, Sabo, Dorschner,
  Sandstrom, Bernstein, Bender, Groudine, Gnirke, Stamatoyannopoulos, Mirny,
  Lander, and Dekker]{Lieberman-Aiden2009Comprehensive}
E.~Lieberman-Aiden, N.~L. van Berkum, L.~Williams, M.~Imakaev, T.~Ragoczy,
  A.~Telling, I.~Amit, B.~R. Lajoie, P.~J. Sabo, M.~O. Dorschner, R.~Sandstrom,
  B.~Bernstein, M.~A. Bender, M.~Groudine, A.~Gnirke, J.~Stamatoyannopoulos,
  L.~A. Mirny, E.~S. Lander, and J.~Dekker.
\newblock Comprehensive mapping of long-range interactions reveals folding
  principles of the human genome.
\newblock \emph{Science}, 326\penalty0 (5950):\penalty0 289--293, Oct
  2009{\natexlab{a}}.

\bibitem[Lieberman-Aiden et~al.(2009{\natexlab{b}})Lieberman-Aiden, Van~Berkum,
  Williams, Imakaev, Ragoczy, Telling, Amit, Lajoie, Sabo, Dorschner,
  et~al.]{Lieberman2009Comprehensive}
Erez Lieberman-Aiden, Nynke~L Van~Berkum, Louise Williams, Maxim Imakaev,
  Tobias Ragoczy, Agnes Telling, Ido Amit, Bryan~R Lajoie, Peter~J Sabo,
  Michael~O Dorschner, et~al.
\newblock Comprehensive mapping of long-range interactions reveals folding
  principles of the human genome.
\newblock \emph{science}, 326\penalty0 (5950):\penalty0 289--293,
  2009{\natexlab{b}}.

\bibitem[Lim and Wright(2014)]{lim2014beyond}
Cong~Han Lim and Stephen Wright.
\newblock Beyond the birkhoff polytope: Convex relaxations for vector
  permutation problems.
\newblock In \emph{Advances in Neural Information Processing Systems}, pages
  2168--2176, 2014.

\bibitem[Lim and Wright(2016)]{lim2016box}
Cong~Han Lim and Steve Wright.
\newblock A box-constrained approach for hard permutation problems.
\newblock In \emph{International Conference on Machine Learning}, pages
  2454--2463, 2016.

\bibitem[Loman et~al.(2015)Loman, Quick, and Simpson]{Loman15}
Nicholas~J Loman, Joshua Quick, and Jared~T Simpson.
\newblock A complete bacterial genome assembled de novo using only nanopore
  sequencing data.
\newblock \emph{Nature methods}, 12\penalty0 (8):\penalty0 733, 2015.

\bibitem[Lyzinski et~al.(2016)Lyzinski, Fishkind, Fiori, Vogelstein, Priebe,
  and Sapiro]{lyzinski2016graph}
Vince Lyzinski, Donniell~E Fishkind, Marcelo Fiori, Joshua~T Vogelstein,
  Carey~E Priebe, and Guillermo Sapiro.
\newblock Graph matching: Relax at your own risk.
\newblock \emph{IEEE transactions on pattern analysis and machine
  intelligence}, 38\penalty0 (1):\penalty0 60--73, 2016.

\bibitem[Marie-Nelly et~al.(2014)Marie-Nelly, Marbouty, Cournac, Flot, Liti,
  Parodi, Syan, Guill{\'e}n, Margeot, Zimmer, et~al.]{Marie-Nelly2014}
Herv{\'e} Marie-Nelly, Martial Marbouty, Axel Cournac, Jean-Fran{\c{c}}ois
  Flot, Gianni Liti, Dante~Poggi Parodi, Sylvie Syan, Nancy Guill{\'e}n,
  Antoine Margeot, Christophe Zimmer, et~al.
\newblock High-quality genome (re) assembly using chromosomal contact data.
\newblock \emph{Nature communications}, 5:\penalty0 5695, 2014.

\bibitem[Meidanis et~al.(1998)Meidanis, Porto, and Telles]{Meid98}
Jo{\~a}o Meidanis, Oscar Porto, and Guilherme~P Telles.
\newblock On the consecutive ones property.
\newblock \emph{Discrete Applied Mathematics}, 88\penalty0 (1):\penalty0
  325--354, 1998.

\bibitem[Papadimitriou and Steiglitz(1998)]{papadimitriou1998combinatorial}
Christos~H Papadimitriou and Kenneth Steiglitz.
\newblock \emph{Combinatorial optimization: algorithms and complexity}.
\newblock Courier Corporation, 1998.

\bibitem[Pop(2004)]{Pop04}
Mihai Pop.
\newblock Shotgun sequence assembly.
\newblock \emph{Advances in computers}, 60:\penalty0 193--248, 2004.

\bibitem[Recanati et~al.(2017)Recanati, Br{\"u}ls, and
  d’Aspremont]{recanati2016spectral}
Antoine Recanati, Thomas Br{\"u}ls, and Alexandre d’Aspremont.
\newblock A spectral algorithm for fast de novo layout of uncorrected long
  nanopore reads.
\newblock \emph{Bioinformatics}, 33\penalty0 (20):\penalty0 3188--3194, 2017.

\bibitem[Robinson(1951)]{Robi51}
William~S Robinson.
\newblock A method for chronologically ordering archaeological deposits.
\newblock \emph{American antiquity}, 16\penalty0 (4):\penalty0 293--301, 1951.

\bibitem[Sahni and Gonzalez(1976)]{sahni1976p}
Sartaj Sahni and Teofilo Gonzalez.
\newblock P-complete approximation problems.
\newblock \emph{Journal of the ACM (JACM)}, 23\penalty0 (3):\penalty0 555--565,
  1976.

\bibitem[Schmidt(2005)]{schmidt2005minfunc}
Mark Schmidt.
\newblock minfunc: unconstrained differentiable multivariate optimization in
  matlab.
\newblock \emph{Software available at http://www. cs. ubc. ca/\~{}
  schmidtm/Software/minFunc. htm}, 2005.

\bibitem[Selvaraj et~al.(2013)Selvaraj, Dixon, Bansal, and Ren]{selvaraj2013}
Siddarth Selvaraj, Jesse~R Dixon, Vikas Bansal, and Bing Ren.
\newblock Whole-genome haplotype reconstruction using proximity-ligation and
  shotgun sequencing.
\newblock \emph{Nature biotechnology}, 31\penalty0 (12):\penalty0 1111, 2013.

\bibitem[Vogelstein et~al.(2011)Vogelstein, Conroy, Lyzinski, Podrazik,
  Kratzer, Harley, Fishkind, Vogelstein, and Priebe]{vogelstein2011fast}
Joshua~T Vogelstein, John~M Conroy, Vince Lyzinski, Louis~J Podrazik, Steven~G
  Kratzer, Eric~T Harley, Donniell~E Fishkind, R~Jacob Vogelstein, and Carey~E
  Priebe.
\newblock Fast approximate quadratic programming for large (brain) graph
  matching.
\newblock \emph{arXiv preprint arXiv:1112.5507}, 2011.

\bibitem[Von~Luxburg(2007)]{von2007tutorial}
Ulrike Von~Luxburg.
\newblock A tutorial on spectral clustering.
\newblock \emph{Statistics and computing}, 17\penalty0 (4):\penalty0 395--416,
  2007.

\end{thebibliography}

\clearpage

\beginsupplement
\onecolumn

\section{Supplementary Material}

\subsection{Seriation and Robust Seriation Algorithms}

\subsubsection{Spectral ordering algorithm}\label{sssec:spectral-ordering}
The spectral ordering algorithm \citep{Atkins} is the baseline for solving \ref{eqn:seriation}. It is closely related to the well-known spectral clustering algorithm \citep{von2007tutorial}. For any vector $\mathbf{x}$, the objective in~\eqref{eqn:2sum} reads
\begin{align}
\label{eqn:2sumQuad}
\sum_{i,j=1}^{n} A_{ij} \left(x_i - x_j \right)^2 = x^T L_A x
\end{align}
where $L_A =\diag(A\ones)-A$ is the Laplacian matrix of $A$. This means that the 2-SUM problem amounts to
\[
\min_{\pi \in \cP} \pi^T L_A \pi
\]
where $\pi$ is a permutation vector.
The most aggressive relaxation to the problem would be to fully drop the constraints $\pi \in \cP$ and solve the unconstrained problem,
\[
\min_{x \in \reals^n} x^T L_A x
\]
The solution to this relaxation is $x^* = 0$, for which the objective achieves the minimum $0$ (note that from the expression of \eqref{eqn:2sumQuad}, $L_A$ is positive semi-definite, since $A$ has non-negative entries).
Projecting $x^*$ on $\cP_n$ is a degenerate problem : all permutation vectors are at the same distance from $0$ (\ie, they all have the same norm, since they all have the same set of values).
Some additional constraints are thus necessary to make the relaxation useful.
Since all permutation vectors have the same norm $c = \| \pi \|$ (in other words, they all belong to a sphere $\Sigma_n$), a slightly more refined relaxation is to optimize over that sphere,
\[
\min_{x \: : \: \| x \| = c } x^T L_A x
\]
This problem is no other than an eigenvector problem. Its solution is  $x^* = c \ones_n$, which also achieves the minimum $0$ (see the expression from \eqref{eqn:2sumQuad} when all $x_i$ are equal).
Again, this optimum is non-informative. Any permutation $\pi$ has the same distance to $x^*$, $d = \sum_{i=1}^{n} (c - i)^2$, thus projecting back $x^*$ to $\cP_n$ is totally degenerate. This can be viewed geometrically on Figure~\ref{fig:goodVsBadTieBreakPH3}.
Now, we can add an orthogonality constraint to the smallest eigenvector, $\ones$,
\[
\min_{x \: : \: \| x \| = c, \: x \perp \ones } x^T L_A x
\]
It still is an eigenvector problem. The solution is the second smallest eigenvector, called the Fiedler vector. A permutation is recovered from this eigenvector by sorting its coefficients (which is impossible for $\ones$): given $\mathbf{x} = (x_1, x_2,  ..., x_n)$, the spectral algorithm outputs a permutation $\pi$ such that $x_{\pi(1)} \leq x_{\pi(2)} \leq ... \leq x_{\pi(n)}$. This procedure is summarized as Algorithm \ref{alg:spectral}.

\subsubsection{Spectral Relaxation for Robust Seriation}\label{sssec:eta-spectral}
We describe here the method $\eta$-Spectral, for reordering a matrix into an approximate R-matrix. The absolute value of a real number $x \in \reals$ can be expressed as the solution of an minimization problem over a real variable $\eta$,
\begin{equation*}
|x| = \argmin_{\eta \geq 0} \frac{x^2}{\eta} + \eta
\end{equation*}
Similarly, the Huber function defined by
\begin{equation*}
h_{\delta}(x) = \left\{
\BA{ll}
x^2  &  \mbox{if $ |x| \leq \delta$,}\\
\delta (2|x| - \delta) & \mbox{otherwise}\\
\EA
\right.
\end{equation*}
can be expressed as 
\begin{equation*}
h_{\delta}(x) = \argmin_{0 \leq \eta \leq \delta} \frac{x^2}{\eta} + \eta
\end{equation*}
Using this variational form, we can write \ref{eqn:huberSUM} as an optimization problem over variables $\pi$ and $\eta \in \reals^{n \times n}$,
\begin{align}\tag{$\eta$-HuberSUM}\label{eqn:etaTrickHuberSUM}
\BA{ll}
\mbox{minimize} &  \sum_{i,j=1}^n 
A_{ij} \left ( \frac{(\pi_i - \pi_j)^2}{\eta_{ij}} + \eta_{ij} \right) \\
\multirow{2}{2cm}{\st} & \pi \in \cP, \\
& 0 \leq \eta_{ij} \leq \delta, \: \mbox{for all } i,j.
\EA
\end{align}
The objective in \ref{eqn:etaTrickHuberSUM} is jointly convex in $(\pi, \eta)$ (sum and combination of linear functions with quadratic over linear). The constraint set for $\eta$ is convex, and although $\cP$ is not, it can be relaxed to $\cPH$. However we found empirically that an alternate minimization scheme that is not based on convex optimization but rather exploits the efficiency of the spectral algorithm shows good performances. We present it in Algorithm~\ref{alg:etaTrickHuberSUM}.
We use the spectral algorithm to (approximately) solve \eqref{eqn:2sum} in step 3. Here, $\gamma$ is a parameter that controls the influence of the previous iterates of $\eta$, the case $\gamma=0$ is just plain alternate minimization.
In practice, we evaluate the objective of \eqref{eqn:huberSUM} for $A$ and $\pi^{(t)}$ at each iteration, and keep the iterate $\pi$ with the lowest score.

\subsubsection{Unconstrained optimization with Iterative Bias (UBI)}\label{sssec:Uncons}

We propose an iterative method where each outer iteration $t$ solves a subproblem biased towards a direction $x^*_{(t-1)}$ given by the previous iteration, with an unconstrained optimization descent method (instead of Frank-Wolfe), in Algorithm~\ref{alg:Uncons}. 
The bias is added in order to avoid the trivial minimum in the center of $\cPH$.
\begin{algorithm}[ht]
	\caption{Iterative scheme with biased unconstrained optimization in $\cH_n$.}
	\label{alg:Uncons}
	\begin{algorithmic} [1]
		\REQUIRE An objective function $f$, an initial bias direction $\pi^{(0)} \in \cP_n$, an increasing bias function $h : \reals \rightarrow \reals$, a maximum number of outer iterations $T$, an optimization algorithm $\mathcal{A}$.
		\STATE Set $t = 0$
		\WHILE{$t < T$}
		\STATE Compute 
		\[
		x^{(t+1)}_* \in \argmin_{x \in \cH_n} \left\{ f(x) - h(x^T x^{(t)}_*) \right\}\quad \mbox{using algorithm $\mathcal{A}$.}
		\]
		\STATE Set $\pi^{(t+1)} = \mathop{\rm argsort} x^{(t+1)}_*$
		\STATE  $t \gets t+1$.
		\ENDWHILE
		\ENSURE A permutation $\pi^{(T)}$.
	\end{algorithmic}
\end{algorithm}

In practice, $f$ is the objective from \ref{eqn:2sum} or \ref{eqn:huberSUM}, and $h$ is a scaled sigmoid function, \[h(t) = \frac{1}{1 + \exp(-\lambda t)}.\]
The sigmoid function is non convex but smooth. It is also bounded, so the optimum of $\left\{ f(x) - h(x^T \pi^{(t)}) \right\}$ is always finite with $f=f_{\text{2SUM}}$ and $f=f_{\text{Huber}}$.

\begin{figure}[htb]
	\begin{center}
		\centerline{\includegraphics[width=0.5\columnwidth]{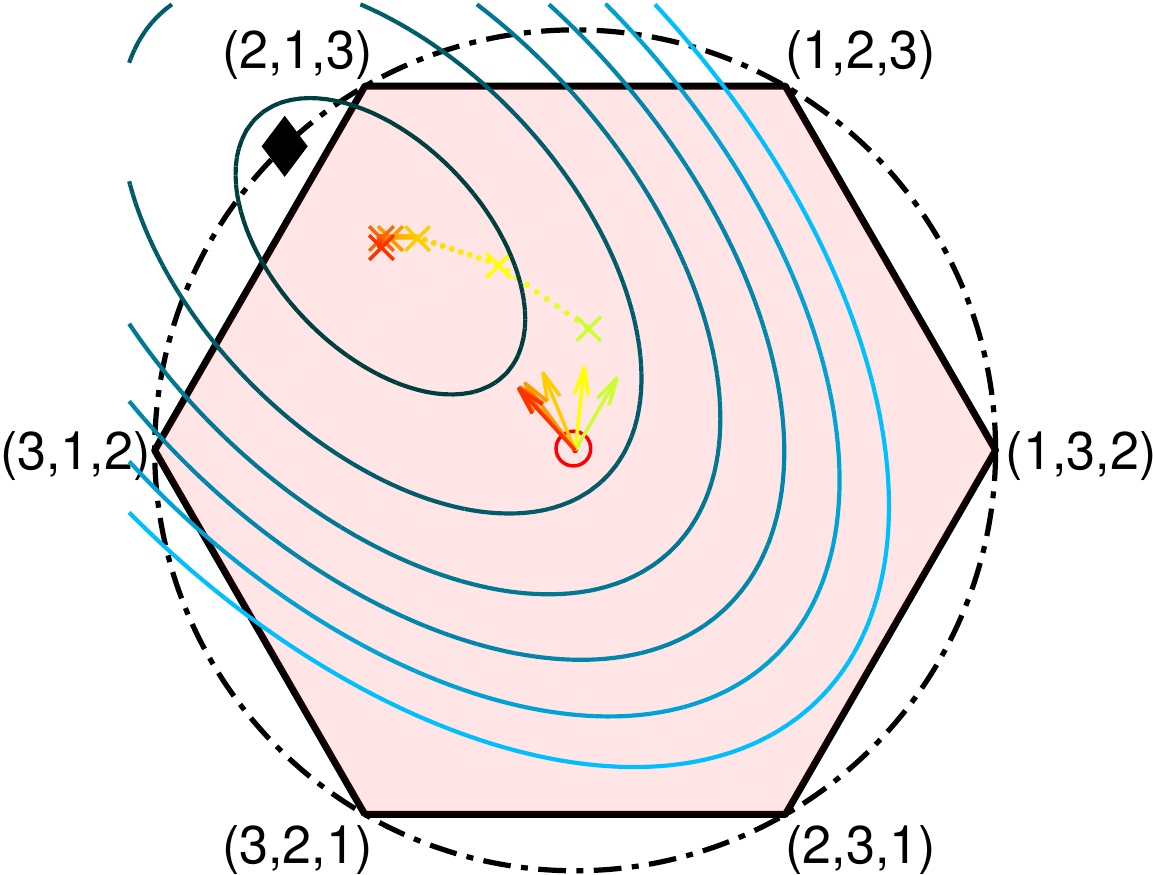}}
		\caption{
			Illustration of Algorithm~\ref{alg:Uncons} in
			the 3-Permutahedron $\cPH_3$ (filled polygon, same representation as in Figure~\ref{fig:goodVsBadTieBreakPH3}).
			The colored crosses (from flashy green (right) to red (left)) represent the solutions $x^{(t)}_*$ obtained at the outer loops of the algorithm, and the associated colored arrows in the center point towards the associated bias that was used at iteration $t$.
			In blue are the level lines of $\left\{ f(x) - h(x^T x^{(t)}_*) \right\}$ with $f=f_{\text{2SUM}}$ and $t=5$ (red arrow).
		}
		\label{fig:unconsPH3}
	\end{center}
	\vskip -0.2in
\end{figure}
Figure~\ref{fig:unconsPH3} illustrates the iterative procedure. The colored crosses indicate the minima of the sequence of biased functions. The last one (with the level lines) is biased towards the optimum. Empirically, we found better results by using $\pi^{(t)}$ rather than $x^{(t)}_*$ in step 3 of Algorithm~\ref{alg:Uncons}, which is what we eventually used in the experiments. Optimization of $f$ in $\cH_n$ is done through unconstrained optimization in $\reals^{n-1}$ of the composition of $f$ with an affine transformation described in the following paragraph (\S \ref{sssec:Hn-vs-Rn-1}).

\subsubsection{Implementation detail : from $\reals^{n-1}$ to $\cH_n$}\label{sssec:Hn-vs-Rn-1}

The idea of the methods UBI (\S~\ref{sssec:Uncons}) and Manopt is to relax the set of permutations $\cP_n$ to the whole space $\reals^n$, or to the sphere $\Sigma_n = \left\{x \in \reals^n \: : \: \|x\|_{2}^2 = \sum_{i=1}^{n} i^2 \right\}$, respectively.
However, we have seen in \S~\ref{sssec:sym-issue} and Figure~\ref{fig:goodVsBadTieBreakPH3} that the set of permutation lie in a hyperplane of dimension $n-1$, $\cH_n = \{ x \in \reals^n | x^T \ones =  \frac{n (n+1)}{2} \}$ (which simply means that all permutation vectors have the same sum).

Thus, we compute a basis of $\cH_n$ and use an affine transformation from $\reals^{n-1}$ to $\cH_n$ such that $0_{n-1}$ corresponds to $c_n = (n+1)/2 \ones_n \in \cH_n$.
In practice, we used,
$U = (u^{(1)}, \ldots, u^{(n-1)}) \in \reals^{n \times n-1}$, \eg, $u^{(j)} = \frac{\tilde{u}^{(j)}}{\| \tilde{u}^{(j)} \|}$, with
\[\left\{\BA{ll}
\tilde{u}^{(j)}_i = 0 & \mbox{if $i < j$,}\\
\tilde{u}^{(j)}_{j} = -j & \\
\tilde{u}^{(j)}_i = 1 & \mbox{if $i > j$,}
\EA\right.\]
The vectors $\{u^{(j)}\}_{1 \leq j \leq n-1}$ are orthonormal and are all orthogonal to $\ones_n$.
Any point $x \in \cH_n$ can be written as $x = \mathcal{A}(y) \triangleq Uy + c_n$ with $y \in \reals^{n-1}$.
For any $f \: : \reals^n \rightarrow \reals$, we define $f_{\cH_n} \: : \reals^{n-1} \rightarrow \reals$ by $f_{\cH_n}(y) = f(Uy + c_n)$ in order to perform unconstrained optimization (UBI) on $f_{\cH_n}$.
The intersection of the sphere $\Sigma_n$ with $\cH_n$, represented by the black dashed circle in Figure~\ref{fig:goodVsBadTieBreakPH3}, is the transformation of a sphere $\tilde{\Sigma}_{n-1}$ by $\mathcal{A}$, so we used (Manopt) with $f_{\cH_n}$ on a sphere in $\reals^{n-1}$ in the experiments.

\subsubsection{Frank-Wolfe with Tie-Break (FWTB) : Linear Minimization Oracle}\label{sssec:FWTB-LMO}
The conditional gradient (Frank-Wolfe) algorithm is suited to optimization in $\cPH_n$ since the linear minimization oracle (LMO) performed at each iteration boils down to sorting the entries of a vector $g \in \reals^n$ (hence, it has a computational complexity of $O(n \log n)$).
Specifically, the LMO solves,
\begin{align}
\tag{LMO} \label{eqn:LMO-PH}
\BA{ll}
\mbox{minimize}  & \sum_{i=1}^n 
\pi_i g_i \\
\st & \pi \in \cPH_n
\EA
\end{align}
where $g_i$ is the $i$-th entry of the gradient of the loss function. This linear form is minimized on a vertex of $\cPH$, \ie~on a permutation $\pi^*$. Let $\sigma \in \cP_n$ be a permutation that sorts the entries of $g$ by decreasing order, such that $g_{\sigma_1} \geq \ldots \geq g_{\sigma_n}$, then $\pi^*$ is defined by $\pi^{*}_{\sigma_1} = 1, \ldots, \pi^{*}_{\sigma_n} = n$.

The method (FWTB) adds a tie-breaking constraint  (\eg, $\pi_1 + 1 \leq \pi_n$) in order to break the symmetry and exclude the center $c_n$ from the feasible set, as suggested by \citet{lim2014beyond}.
Yet, while \citet{Fogel, lim2014beyond} proposed convex optimization methods that could incorporate any such linear constraint into the problem seamlessly,
if one wants to use Frank-Wolfe in the restriction of $\cPH$ where the tie-break is satisfied, the \ref{eqn:LMO-PH} has to be modified.
The new LMO must solve,
\begin{align}
\tag{LMO-tb} \label{eqn:LMO-tiebreak}
\BA{ll}
\mbox{minimize}  & \sum_{i=1}^n 
\pi_i g_i \\
\st  & \left\{ 
\BA{l} \pi \in \cPH_n , \\
\pi_i + 1 \leq \pi_j. 
\EA
\right.
\EA
\end{align}
where we let $1 \leq i \neq j \leq n$ be the tie-break indexes  (in \citet{Fogel, lim2014beyond}, $i=1$ and $j=n$).
\citet{lim2014beyond} propose an algorithm for solving \ref{eqn:LMO-PH} that preserves the $O(n \log n)$ complexity of the \ref{eqn:LMO-PH}.
We describe in Algorithm~\ref{alg:LMO-tb} a slightly simplified version of theirs, for any tie-break indexes $1 \leq i \neq j \leq n$. We use the matlab-like notation $x(i)$ to denote $x_i$ for ease of reading. 

\begin{algorithm}[ht]
	\caption{Minimizing $g^T \pi$ over $\cPH_n$ with tie-break $\pi(i) + 1 \leq \pi(j)$.}
	\label{alg:LMO-tb}
	\begin{algorithmic} [1]
		\STATE $g^\prime, \sigma \gets $ sort $g$ in decreasing order ( \ie, $g(\sigma_1) \geq \ldots \geq g(\sigma_n)$ )
		\FOR{$k \gets 1$ to $n-1$}
		\IF{$g^\prime(k) < \frac{g(i) + g(j)}{2}$}
		\STATE \textbf{break}
		\ENDIF
		\STATE $\sigma^{-1} \gets \argsort \sigma$
		\STATE Set $\tilde{z} = (1, \ldots, k-1, k+2, \ldots, n)^T \in \reals^{n-2}$
		\STATE $\pi(l) \gets \tilde{z}(\sigma^{-1}(l))$ for $l \in \{1,\ldots,n\} \smallsetminus \{i,j\}$
		\STATE $\pi(i) \gets k$
		\STATE $\pi(j) \gets k+1$.
		\ENDFOR
		\ENSURE A permutation $\pi^{(T)}$.
	\end{algorithmic}
\end{algorithm}

\begin{proposition}
	Algorithm~\ref{alg:LMO-tb} minimizes $g^T \pi$ over $\cPH_n$ with tie-break $\pi(i) + 1 \leq \pi(j)$.
\end{proposition}
\begin{proof}
	Without loss of generality, let us assume for simplicity that $g$ is already sorted by decreasing value. Let $\pi^*$ be the solution of \ref{eqn:LMO-PH}. If $\pi^*_i + 1 \leq \pi^*_j$, then $\pi^*$ is also solution of \ref{eqn:LMO-tiebreak}. Otherwise, the solution of \ref{eqn:LMO-tiebreak} will be a permutation $\pi$ where the constraint is active : $\pi_i + 1 = \pi_j$ \citep{lim2014beyond}.
	Let $k=\pi_i$. There are $n-1$ possible values for $k$ : $\{1,\ldots, n-1\}$. For a given $k$,  the vector $\tilde{\pi}_k$, the restriction of $\pi$ to the $n-2$ indexes other than $i$ and $j$ is given by Smith's rule : it is the concatenation of the remaining values $ \tilde{\pi}_k = (1, \ldots, k-1, k+2, \ldots, n)$ (given that $g$ is sorted). Therefore, the permutation $\pi$ optimal for \ref{eqn:LMO-tiebreak} is determined by $k$. Let us note $\tilde{g} \in \reals^{n-2}$ the vector $g$ without the two entries corresponding to indexes $i$ and $j$, that is to say, if $i < j$, $\tilde{g} = (g_1, \ldots, g_{i-1}, g_{i+1}, \ldots, g_{j-1}, g_{j+1}, \ldots, n)$.
	To know the optimal value of $k$, let us observe the difference between the objective of \ref{eqn:LMO-tiebreak} for $k=K$ and $k=K+1$, with $ 1 \leq K \leq n-2$.
	For a given $k$, he objective in \ref{eqn:LMO-tiebreak} can be written as the sum $\tilde{g}^T \tilde{\pi}_{k} + k g_i + (k+1) g_j$. Let us write the tilde scalar product part first.
	\begin{equation*}
	\BA{rlllllllllll}
	\tilde{g}^T \tilde{\pi}_{K}  & = 1 \tilde{g}_{1}   +  2 \tilde{g}_{2} & + \ldots + & (K-1) \tilde{g}_{K-1} & + & (K+2) \tilde{g}_{K} & + & (K+3) \tilde{g}_{K+1} & +  \ldots  +  n \tilde{g}_{n-2}\\ 
	\tilde{g}^T \tilde{\pi}_{K+1} & = 1 \tilde{g}_{1}   +  2 \tilde{g}_{2} & + \ldots + & (K-1) \tilde{g}_{K-1} & + & K \tilde{g}_{K} & + & (K+3) \tilde{g}_{K+1} & +  \ldots  +  n \tilde{g}_{n-2}\\ 
	\EA
	\end{equation*}
	The difference between the objective values for $k=K$ and $k=K+1$ is therefore $\Delta_K = 2 \tilde{g}_{K} - (g_i + g_j)$.
	Since we assumed $g$ sorted by decreasing order, $\tilde{g}$ also is, and consequently, $\Delta_K$ decreases with $K$. The optimal $K^*$ is therefore the smallest (first) index $k$ for which $\tilde{g}_{k} < \frac{(g_i + g_j)}{2}$, and if $\tilde{g}_{k} \geq \frac{(g_i + g_j)}{2}$ for all $k \in \{1, \ldots, n-2\}$, then $K^* = n-1$.
\end{proof}

\subsection{Frank-Wolfe with Tie-Break (FWTB) is biased}\label{ssec:FWTB-biased}
Using a tie-break constraint to break the symmetry and solve a problem such as \ref{eqn:2sum} or \ref{eqn:huberSUM} can perform surprisingly poorly because the tie-break constraint actually introduces a bias in the problem. To illustrate this, let us focus on the \ref{eqn:2sum} problem. The loss function is homogeneous,
\begin{equation*}
f_{{\text 2SUM}} (t x) = \sum_{i,j} A_{ij} (t x_i - t x_j)^2 = t^2  \sum_{i,j} A_{ij} (x_i - x_j)^2 = t^2 f_{{\text 2SUM}} (x).
\end{equation*}
Similarly, $ f_{{\text 1SUM}} (t x) = t f_{{\text 1SUM}} (x) $ for $t>0$.
Hence, scaling down a given vector $x$, \eg, letting $x \gets \frac{1}{2} x$, reduces the objective function but does not add information about the optimal permutation (the projection on the set of permutations is the same for both vectors).
What we are interested in is to find a direction $x^*$ which is optimal compared to other vectors $x$ of same norm. In the original problem over permutations, all permutation vectors have the same norm. In the spectral relaxations, we optimize over a sphere. However, when we relax to $\cPH$, the most prominent descent direction of the function is towards the center. The tie-breaking constraint prevents iterates reaching the center, but it adds a bias in a given direction because not all points saturating the tie-breaking contraint have the same norm nor the same distance to the center. On the set of points in $\cPH_n$ where the tie-break is active, \eg, $\{ x \in \cPH_n \:|\: x_1 + 1 \leq x_n \}$, the point $\tilde{c} = c +  e_n - \frac{1}{n} \ones $ has a squared distance to $c$ and $\ell_2$ norm : $\| \tilde{c} - c \|_{2}^2 \simeq 1$,  $\| \tilde{c}\|_{2}^2 \simeq \frac{n^3}{4}$, whereas a permutation $\pi$ that satisfies the constraints has a distance to $c$ that scales in $n^3$ and a larger norm : $\| \pi - c \|_{2}^2 \simeq \frac{n^3}{12}$,  $\| \pi \|_{2}^2 \simeq \frac{n^3}{3}$.
Therefore, although the direction $\tilde{c}$ may not be optimal for \ref{eqn:2sum} (compared to other vectors of same norm), the minimizer of \ref{eqn:2sum} with tie-break may be closer to the direction of $\tilde{c}$ than to the optimal one. 
This is what we observe in Figure~\ref{fig:goodVsBadTieBreakPH3} for the bad (orange, top-right triangle) tie-break.

When $n$ becomes large, this may actually lead to numerical precision issues. Indeed, the $n-1$ first entries of $\tilde{c}$ are equal. When the optimum $x^*$ in the tie-break-constrained $\cPH$ gets close to $\tilde{c}$, the variations among the $n-1$ first entries of $x^*$ also shrink and the precision required to sort them (in order to project back onto the set of permutations) may become too high.

In Figure~\ref{fig:goodVsBadTieBreakPH3}, we also display a good (green, top-left triangle) tie-break.
In practice, although there are $\binom{n}{2}$ non-redundant choices for the indexes $i$ and $j$ constituting a tie-breaking constraint $\pi_i + 1 \leq \pi_j$,
we can use the solution $\pi^{\text{spectr.}}$ of the cheap, spectral ordering (Algorithm~\ref{alg:spectral}) to find a good candidate tie-break. Specifically, chose $i \in \argmin \pi^{\text{spectr.}}$ and $j \in \argmax \pi^{\text{spectr.}}$.

The performances of FWTB with the naive ($i=1$, $j=n$) and spectral-initialized tie-breaking strategies are compared to that of the basic spectral Algorithm~\ref{alg:spectral} in Table~\ref{tb:KT-vs-s-FW-tb} (a -I is appended to the algorithm name for the spectral-initialized tie-break results), with the same experimental setup as in Section~\ref{sec:RobustSeriation} with matrices in $\cM_n(\delta, s)$.

\begin{table}[t]
	\caption{Kendall-$\tau$ score for different values of $s/s_{\text{lim}}$, for the spectral method and Frank-Wolfe with default and initialized tie-breaks (-I variants), with $n=200$, $\delta=20$.}
	\label{tb:KT-vs-s-FW-tb}
	\vskip .15in
	\begin{center}
		\begin{small}
			\begin{sc}
				\begin{tabular}{lcccccr}

					\toprule
					& $s/s_{\text{lim}}=0.5$ & $s/s_{\text{lim}}=1$ & $s/s_{\text{lim}}=2.5$ & $s/s_{\text{lim}}=5$ & $s/s_{\text{lim}}=7.5$ & $s/s_{\text{lim}}=10$ \\ 
					\midrule
					spectral & 0.96 {$\scriptstyle \pm 0.01 $} & 0.95 {$\scriptstyle \pm 0.01 $} & 0.91 {$\scriptstyle \pm 0.03 $} & 0.86 {$\scriptstyle \pm 0.06 $}  &  0.84 {$\scriptstyle \pm 0.06 $} & 0.80 {$\scriptstyle \pm 0.09 $} \\
					\midrule
					FWTB & 0.40 {$\scriptstyle \pm 0.27 $} &  0.33 {$\scriptstyle \pm 0.27 $} & 0.32 {$\scriptstyle \pm 0.26 $} & 0.34 {$\scriptstyle \pm 0.22 $} & 0.28 {$\scriptstyle \pm 0.21 $} & 0.24 {$\scriptstyle \pm 0.20 $} \\
					H-FWTB & 0.50 {$\scriptstyle \pm 0.31 $} & 0.36 {$\scriptstyle \pm 0.27 $} & 0.32 {$\scriptstyle \pm 0.25 $} & 0.32 {$\scriptstyle \pm 0.22 $} & 0.25 {$\scriptstyle \pm 0.21 $} & 0.22 {$\scriptstyle \pm 0.18 $} \\
					\midrule
					FWTB-i & 0.91 {$\scriptstyle \pm 0.20 $} & 0.92 {$\scriptstyle \pm 0.13 $} & 0.84 {$\scriptstyle \pm 0.19 $} & 0.73 {$\scriptstyle \pm 0.20 $} & 0.71 {$\scriptstyle \pm 0.13 $} & 0.64 {$\scriptstyle \pm 0.15 $} \\
					H-FWTB-i & 0.98 {$\scriptstyle \pm 0.01 $} & 0.94 {$\scriptstyle \pm 0.13 $}  &  0.86 {$\scriptstyle \pm 0.15 $}  & 0.70 {$\scriptstyle \pm 0.18 $} & 0.63 {$\scriptstyle \pm 0.15 $} & 0.57 {$\scriptstyle \pm 0.17 $}\\
					\bottomrule
					
				\end{tabular}
			\end{sc}
		\end{small}
	\end{center}
	\vskip -.1in
\end{table}
We can see that using a default tie-breaking constraint performs very poorly on average. Using the solution of the spectral algorithm to define the tie-breaking constraint significantly improves the performance compared to using a default tie-break. Still, it does not outperform the spectral algorithm except in a very low noise setting.
 
\subsection{Seriation with Duplication Algorithms}\label{ssec:SerDuplAlgo}

We now detail algorithmic solutions to several subproblems required by seriation with duplications.

\subsubsection{Projection on $\cR$ (step 3 of Algorithm~\ref{alg:AltProjBase})}\label{sssec:dupli-detail-proj-on-strongR}
In step 3 of Algorithm~\ref{alg:AltProjBase}, we wish to compute $(\Pi_*, S_*)$, solution of \eqref{eqn:robustseriation1} for $S^{(t)}$.
To do so, we can use one of the algorithms presented in Section~\ref{sec:RobustSeriation}.
However, these algorithms do not address the problem of \ref{eqn:robustseriation1} directly. Rather, they seek to find a permutation that is optimal for a objective function which coincides with \ref{eqn:robustseriation1} for the specific class of $\cM_n(\delta,s)$ matrices.
Two problems arise then. First, in our Seriation with Duplication setting \eqref{eqn:seriationdupli}, the matrices may not fit the class $\cM_n(\delta,s)$, especially when the matrix $S$ to be recovered is dense (and not a band matrix). Second, the output of the algorithm is a permutation $\Pi_*$, but what we are really interested in step 3 of Algorithm~\ref{alg:AltProjBase} is the matrix $S_* \in \cR^N$ that is the closest to $S^{(t)}$.
To approximate $S_* \in \cR^N$, we first use one of the methods introduced in Section~\ref{sec:RobustSeriation} to find a permutation $\Pi_*$ that makes $\Pi_* S^{(t)} \Pi^{T}_{*}$
as close to $\cR^N$ as possible.  Still, in general the permuted matrix $\Pi_* S^{(t)} \Pi^{T}_*$ will not be in $\cR^N$. We then project $\Pi_* S^{(t)} \Pi^{T}_*$ onto $\cR^N$, which is solved with linear programming.
Indeed, the projection, in $\ell_1$ norm for example of a matrix $S$, reads
\begin{align}\tag{R-proj} \label{eqn:R-proj}
\BA{ll}
\mbox{minimize} & \sum_{i,j=1}^{N} | R_{ij} - S_{ij} | \\
\st & R \in \cR^N.
\EA
\end{align}
We can also use a Froebenius norm and consider the sum of squares instead of the absolute differences. We would then use quadratic programming, as we have then a quadratic objective with linear constraints.
The constraint $R \in \cR^N$ can indeed be written as linear constaints on $R$. Specifically, we consider the vectorized forms of $S$ and $R$, $s,r \in \reals^{N^2}$, which are the concatenation of the columns of $S$ and $R$, respectively. Imposing $R \in \cR$ is equivalent to saying that $r_u \leq r_v$ for all pairs of indexes $(u,v)$ such that the corresponding subscripts for $u$ are on a diagonal higher than those for $v$. There is one linear constraint per pair $(u,v)$ (and there are $\frac{N(N-1)}{2}$ pairs), but we can reduce the number of constraints by adding slack variables $\{\lambda_k \}_{1 \leq k \leq N}$ and impose that for each element $r_u$ on a given diagonal $k$, $1 \leq k \leq N-1$, $r_u \leq \lambda_{k+1}$ and $r_u \geq \lambda_k$.
Finally, we can use {\it a priori} knowledge on how the values are supposed to decrease when moving away from the diagonal (\eg, a power law  $S_{ij} = |i-j|^{-\gamma}$ as in our experiments, which is consistent with the intra-chromosomal frequency observed in \citet{Lieberman2009Comprehensive}), to upper bound the values $\lambda_k$.
We end up with the following optimization problem over the variable $(r, \lambda)^T$,
\begin{align}\tag{R-proj} \label{eqn:R-proj-simple}
\BA{ll}
\mbox{minimize} & \| r - s \| \\
\st & C \begin{pmatrix}
r \\\lambda
\end{pmatrix} \leq 0, \\
& 0 \leq \lambda \leq b
\EA
\end{align}
where the matrix $C$ contains the strong-R constraints expressed between $r$ and $\lambda$, and the vector $b \in \reals^N$ contains upper bounds on the values of $\lambda_k$, \eg, $b_k = k^{-\gamma}$. 

\subsubsection{Projection on duplication constraints (step 4 of Algorithm~\ref{alg:AltProjBase})}\label{sssec:dupli-detail-proj-on-affine}
In step 4 of Algorithm~\ref{alg:AltProjBase}, we wish to compute the projection of $S$ on the set of matrices $X$ that satisfy $Z X  Z^{T} = A$, that is to say, solve the following optimization problem on variable $X$,
\begin{align}\tag{dupli-proj} \label{eqn:dupl-proj}
\BA{ll}
\mbox{minimize} & \sum_{k,l=1}^{N} | S_{kl} - X_{kl} | \\
\st & Z  X  Z^{T} = A.
\EA
\end{align}
The constraints impose that for each pair $(i,j) \in [1,n]\times [1,n]$, $A_{ij} = \sum_{k \in L_i, l \in L_j } X_{kl}$, where $L_i \subset [1,N]$ is the set of indexes assigned to $i$ through the assignment matrix $Z$.
The objective is also separable, since
\begin{equation*}
\sum_{k,l=1}^{N} | S_{kl} - X_{kl} | = \sum_{i,j=1}^{n} \sum_{k \in L_i, l \in L_j }  |S_{kl} - X_{kl}|
\end{equation*}
We can then solve separately, for each pair $(i,j)$, the subproblem,
\begin{align}\tag{dupli-proj(i,j)} \label{eqn:dupl-proj-sub}
\BA{ll}
\mbox{minimize} & \sum_{k \in L_i, l \in L_j } | S_{kl} - X_{kl} | \\
\st & A_{ij} = \sum_{k \in L_i, l \in L_j } X_{kl}.
\EA
\end{align}
For a given pair $(i,j)$, $L_i$ and $L_j$ are known (through $Z$), and if we consider the vectorization (stacking of the columns into a single vector) of the submatrices $X_{L_i, L_j}$ and $S_{L_i, L_j}$, denoted $x$ and $s$ respectively, and denote $a=A_{ij}$, the subproblem on the variable $x$ reads
\begin{align}\tag{dupli-proj(i,j)} \label{eqn:dupl-proj-sub-simple}
\BA{ll}
\mbox{minimize} & \| s - x \| \\
\st & x^T \ones = a,\\
& x \geq 0.
\EA
\end{align}
We impose non-negativity of the coefficients of $X$ since this is part of the definition of similarity matrices.
The above general problem of approximating a vector with a non-negative vector of fixed norm can be solved exactly when the norm is the $\ell_2$ norm (this solution is optimal for the $\ell_1$ norm too) with Algorithm~\ref{alg:dupli-proj}.
\begin{algorithm}[t]
	\caption{Minimizing $\| s - x \|$ with non-negativity ($x \geq 0$) and sum ($x^T \ones = a$) constraints.}
	\label{alg:dupli-proj}
	\begin{algorithmic} [1]
		\REQUIRE A target vector $s \in \reals^{p}_{+}$, a value $a \geq 0$.
		\STATE $s^\prime, \sigma \gets $ sort $s$ in decreasing order ( \ie, $s(\sigma_1) \geq \ldots \geq s(\sigma_n)$ )
		\FOR{$k \gets 1$ to $n$}
		\STATE $\tilde{x}^{\prime}(k) \gets s^{\prime}(k) + \frac{1}{k} ( a -  \sum_{i=1}^{k} s^{\prime}(i) )$
		\IF{$\tilde{x}^{\prime}(k) < 0$}
		\STATE $k \gets k-1$
		\STATE \textbf{break}
		\ENDIF
		\ENDFOR
		
		\STATE $x^{\prime}(j) = s^{\prime}(j) + \frac{1}{k} ( a -  \sum_{i=1}^{k} s^{\prime}(i) )$ for $j=1,\ldots,k$
		\STATE $x^{\prime}(j) = 0$ for $j > k$
		\STATE $x(\sigma_j) = x^{\prime}(j)$ for $j=1,\ldots,p$.
		
		\ENSURE A vector $x \in \reals^{p}_{+} $.
	\end{algorithmic}
\end{algorithm}

\begin{figure}[p]
	\begin{center}
		
		\begin{subfigure}[htb]{0.3\textwidth}
			\includegraphics[width=\textwidth]{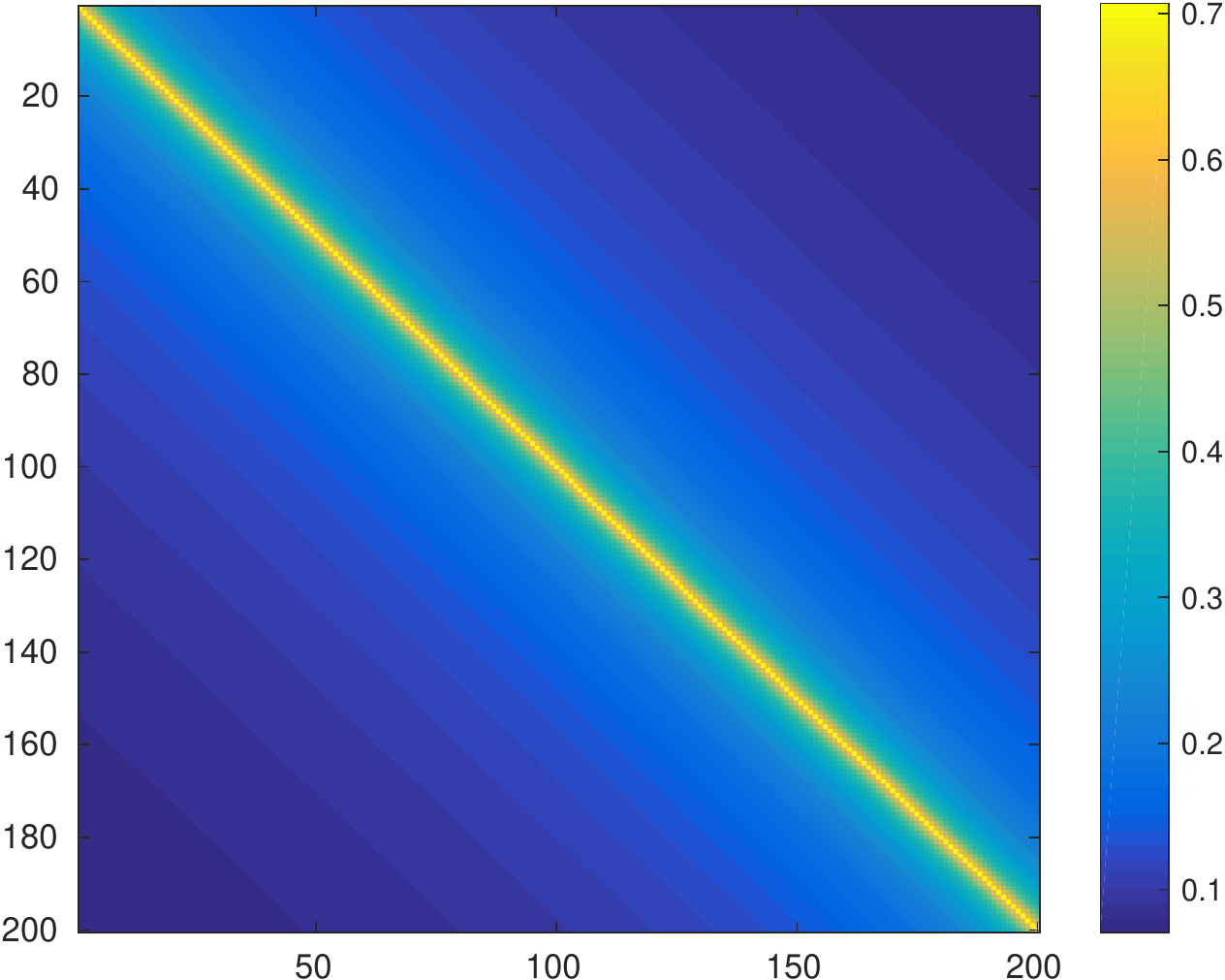}
			\caption{Ground truth $S$}\label{subfig:SerDuplDenseGroundTruth}
		\end{subfigure}
		\begin{subfigure}[htb]{0.3\textwidth}
			\includegraphics[width=\textwidth]{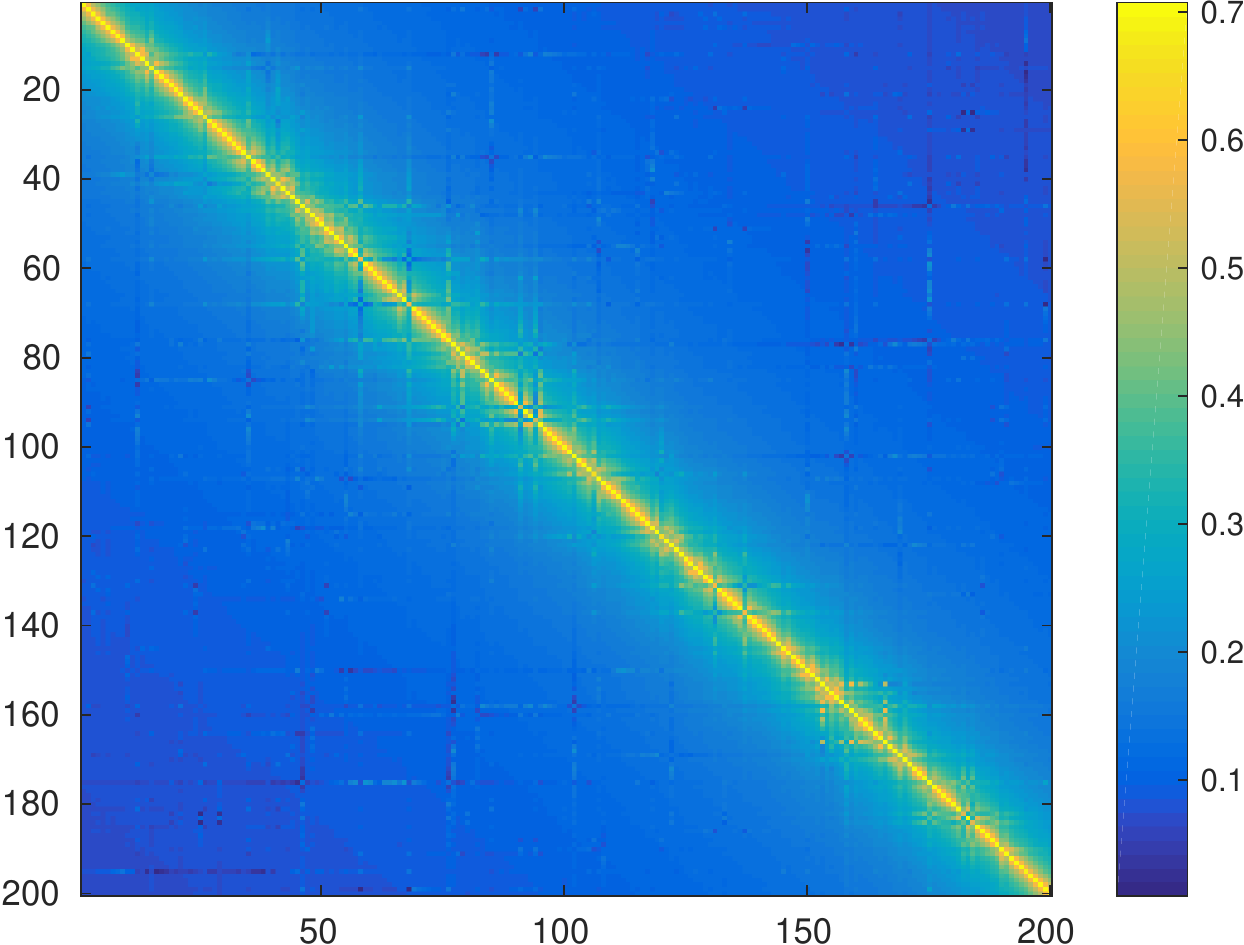}
			\caption{$S^{\text{out}}$ for $N/n=1.33$}\label{subfig:SerDuplDense_sz1}
		\end{subfigure}
		\begin{subfigure}[htb]{0.3\textwidth}
			\includegraphics[width=\textwidth]{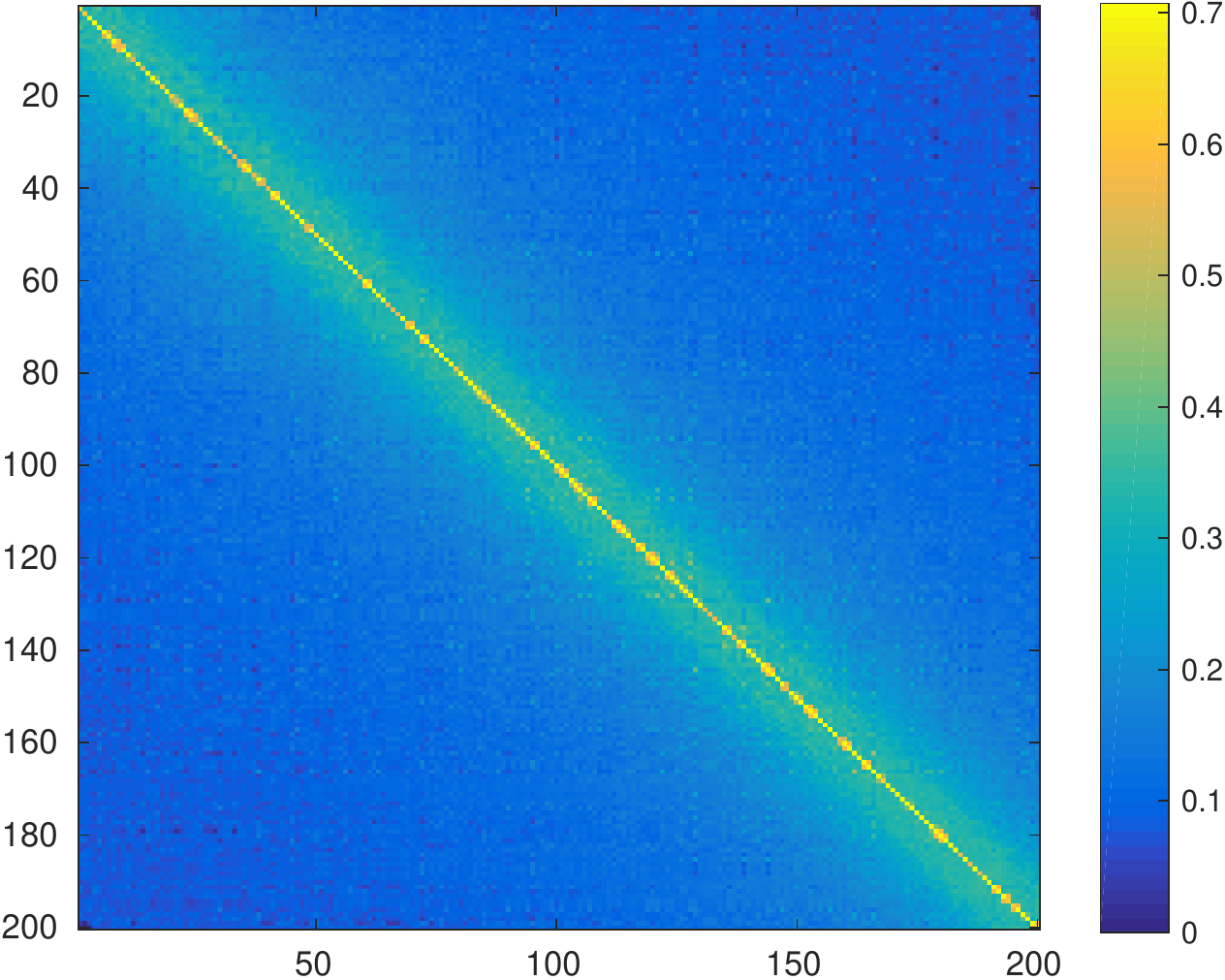}
			\caption{$S^{\text{out}}$ for $N/n=4$}\label{subfig:SerDuplDense_sz4}
		\end{subfigure}
		
		\caption{Original matrix $S$ (with parameter $\gamma=0.5$) from which the data $(A,c)$ is generated (\textsc{A}), output $S^{\text{out}}$ recovered from $(A,c)$ by Algorithm~\ref{alg:AltProjBase} (used with $\eta$-spectral) with $N/n=1.33$ (\textsc{B}) and with $N/n=4$ (\textsc{C}).
			The meanDist metric is 0.98 for $N/n=1.33$ (\textsc{B}) and 10.40 for $N/n=4$ (\textsc{C})}
		\label{fig:SerDuplDenseQualitResults}
	\end{center}
\end{figure}

\begin{figure}[p]
	\begin{center}
		
		\begin{subfigure}[htb]{0.3\textwidth}
			\includegraphics[width=\textwidth]{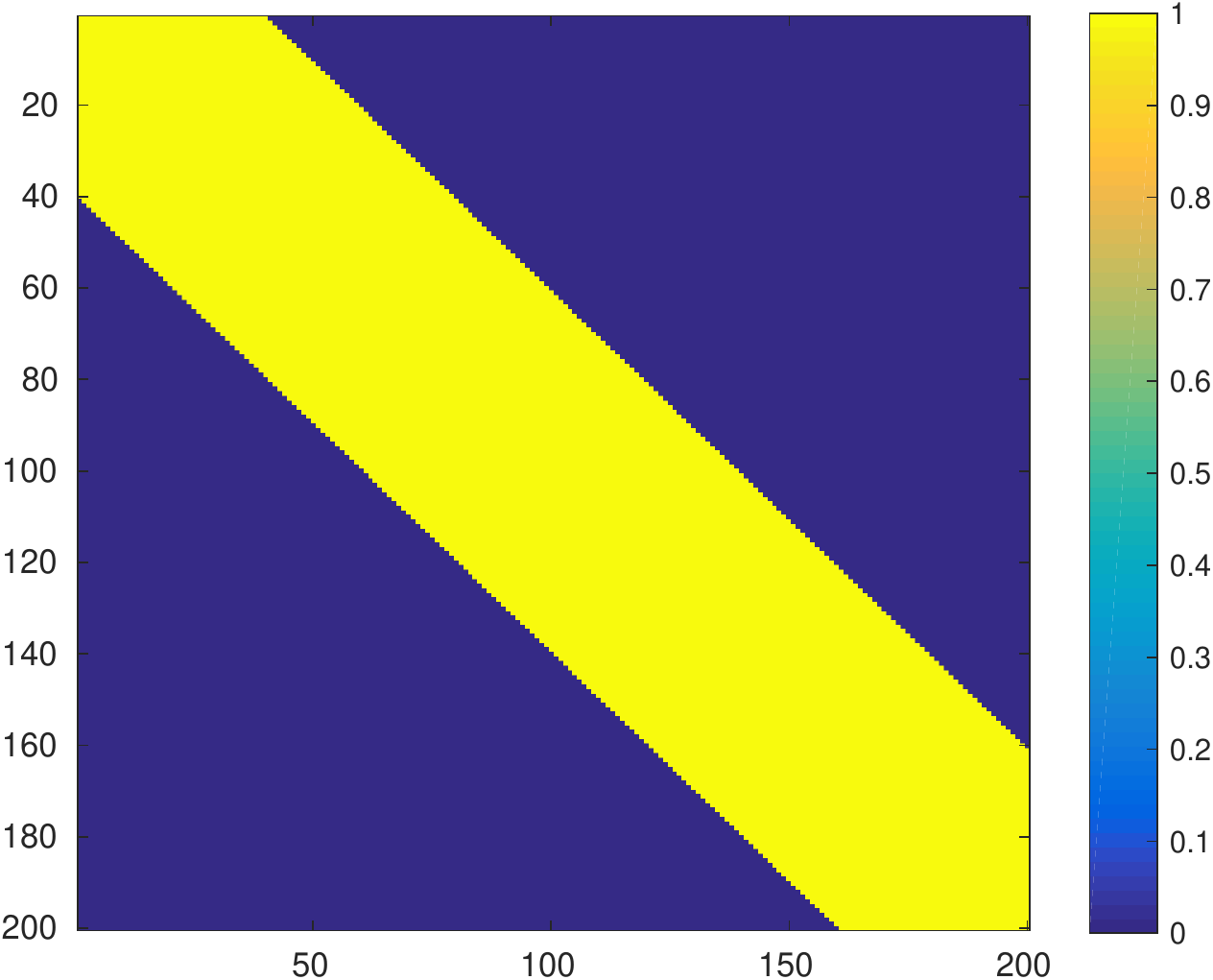}
			\caption{Ground truth $S$}\label{subfig:SerDuplSparseGroundTruth}
		\end{subfigure}
		\begin{subfigure}[htb]{0.3\textwidth}
			\includegraphics[width=\textwidth]{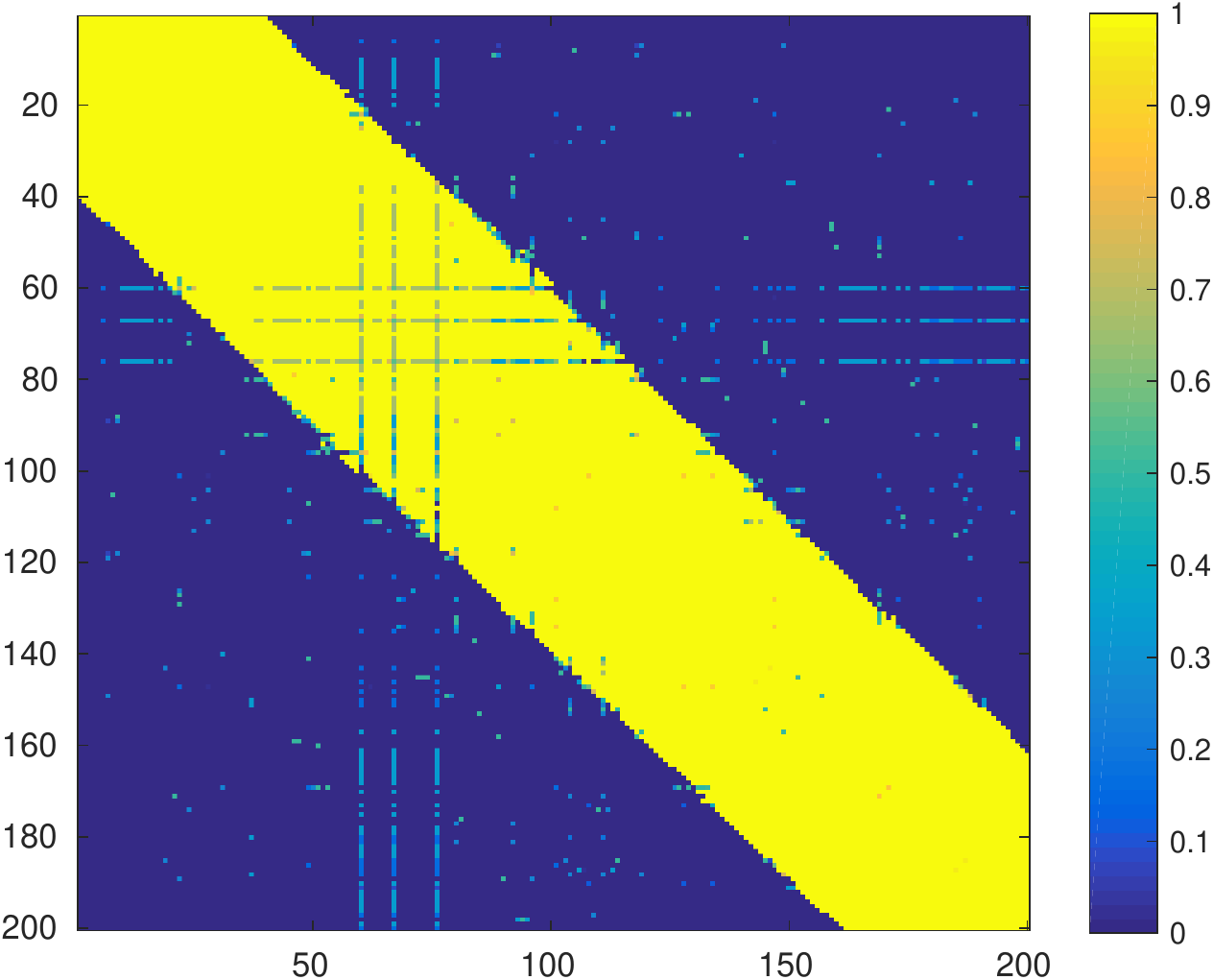}
			\caption{$S^{\text{out}}$ for $N/n=1.33$}\label{subfig:SerDuplSparse_sz1}
		\end{subfigure}
		\begin{subfigure}[htb]{0.3\textwidth}
			\includegraphics[width=\textwidth]{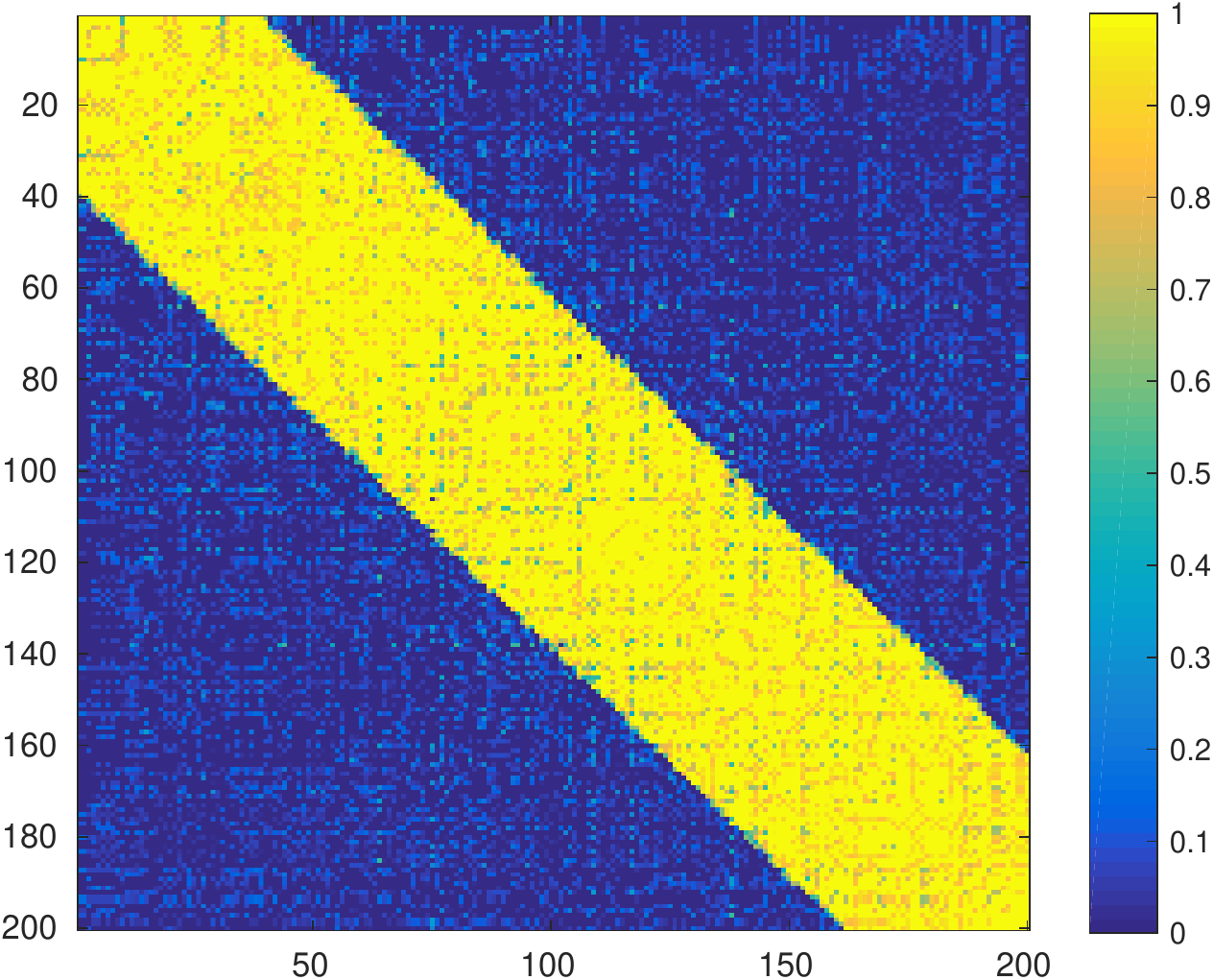}
			\caption{$S^{\text{out}}$ for $N/n=4$}\label{subfig:SerDuplSparse_sz4}
		\end{subfigure}
		
		\caption{Original matrix $S$ (with parameters $\delta=n/5$, $s=0$) from which the data $(A,c)$ is generated (\textsc{A}), output $S^{\text{out}}$ recovered from $(A,c)$ by Algorithm~\ref{alg:AltProjBase} (used with $\eta$-spectral) with $N/n=1.33$ (\textsc{B}) and with $N/n=4$ (\textsc{C}).
			The meanDist metric is 1.03 for $N/n=1.33$ (\textsc{B}) and 7.26 for $N/n=4$ (\textsc{C})}
		\label{fig:SerDuplSparseQualitResults}
	\end{center}
\end{figure}

\begin{figure}[p]
	\begin{center}
		
		\begin{subfigure}[htb]{0.3\textwidth}
			\includegraphics[width=\textwidth]{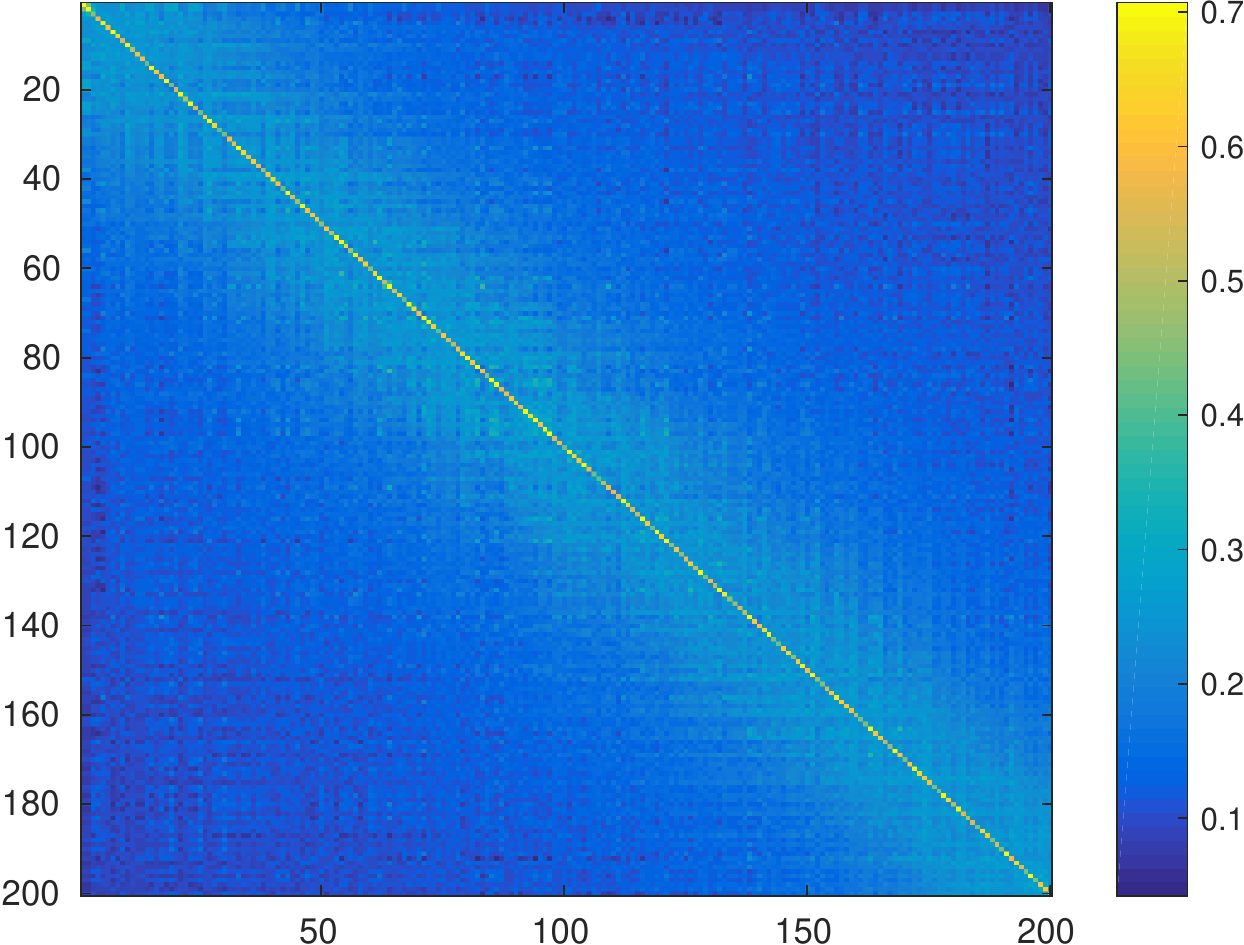}
			\caption{$\Pi_* S^{(t)} \Pi_*^T$ at step 3 of Alg~\ref{alg:AltProjBase}}\label{subfig:SerDenseBeforeRconsProj}
		\end{subfigure}
		\begin{subfigure}[htb]{0.3\textwidth}
			\includegraphics[width=\textwidth]{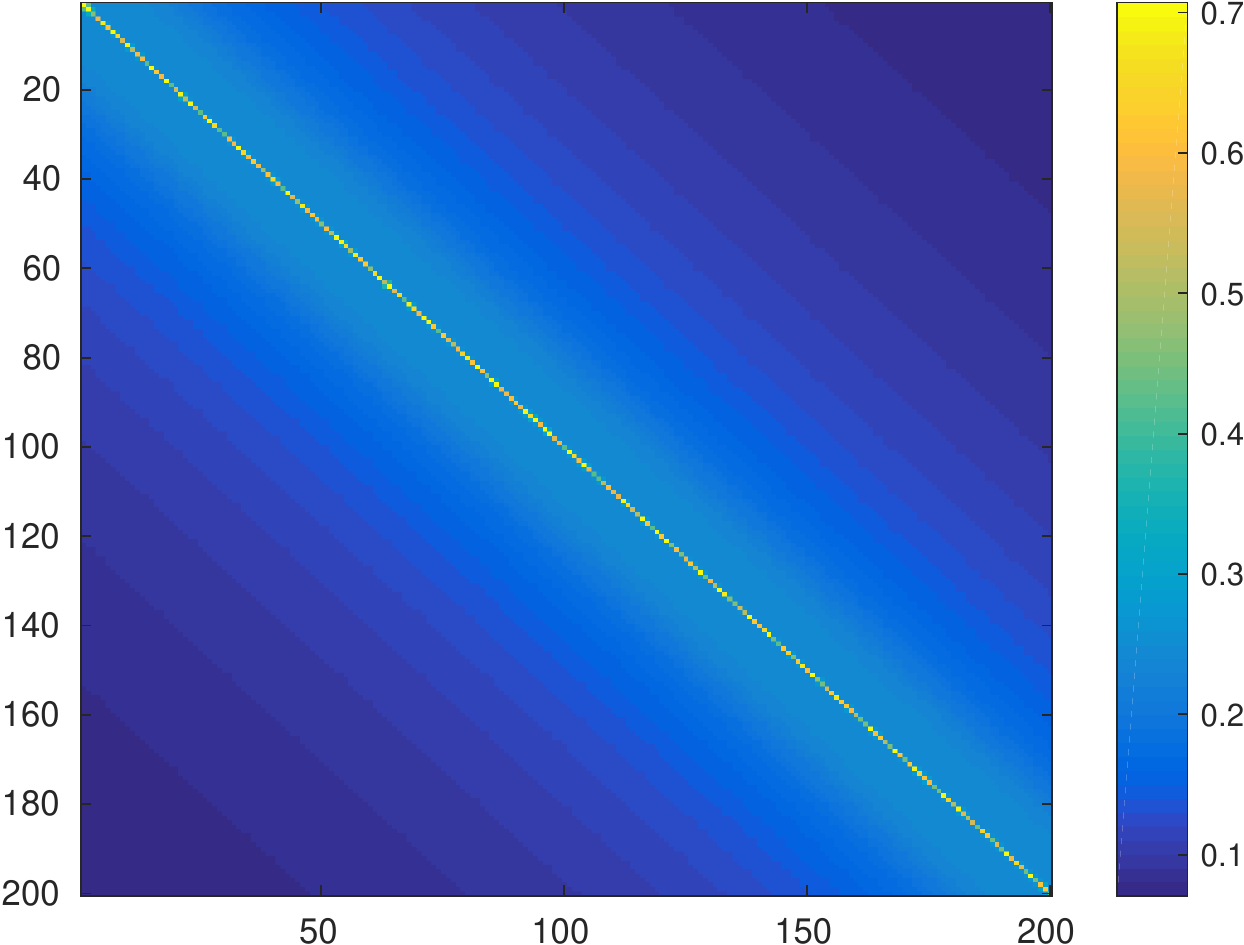}
			\caption{$S^{(t+ \frac{1}{2})}$ at step 3 of Alg~\ref{alg:AltProjBase}}\label{subfig:SerDenseAfterRconsProj}
		\end{subfigure}
		\begin{subfigure}[htb]{0.3\textwidth}
			\includegraphics[width=\textwidth]{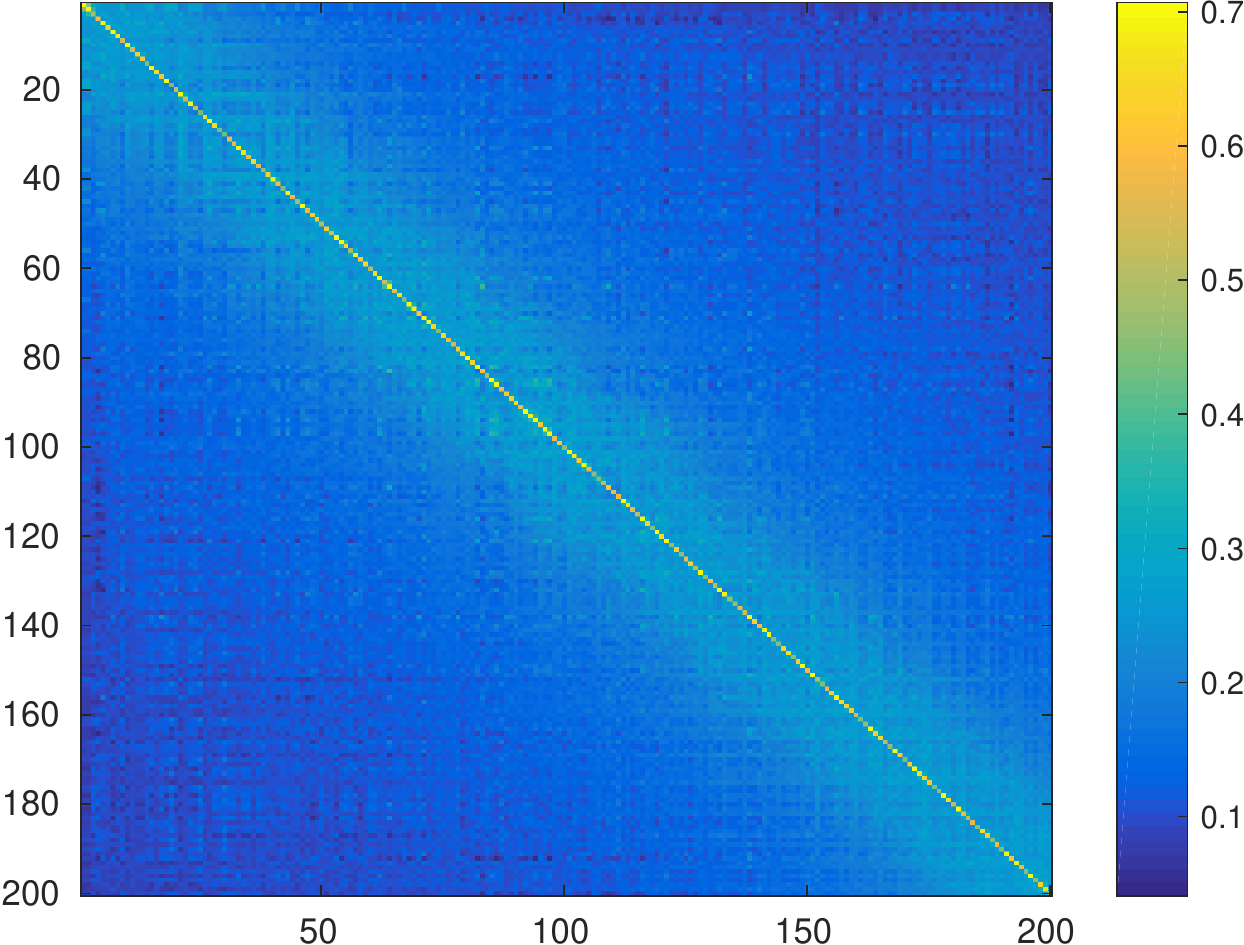}
			\caption{$S_A$ at step 4 of Alg~\ref{alg:AltProjBase}}\label{subfig:SerDenseAfterDuplconsProj}
		\end{subfigure}
		
		\caption{Three steps of Algorithm~\ref{alg:AltProjBase} for a dense matrix $S$ (with parameter $\gamma=0.5$, $n=200$, $N/n=4$).}
		\label{fig:SerDuplDenseThreeStepsOfAlgorithm}
	\end{center}
\end{figure}

\begin{figure}[p]
	\begin{center}
		
		\begin{subfigure}[htb]{0.3\textwidth}
			\includegraphics[width=\textwidth]{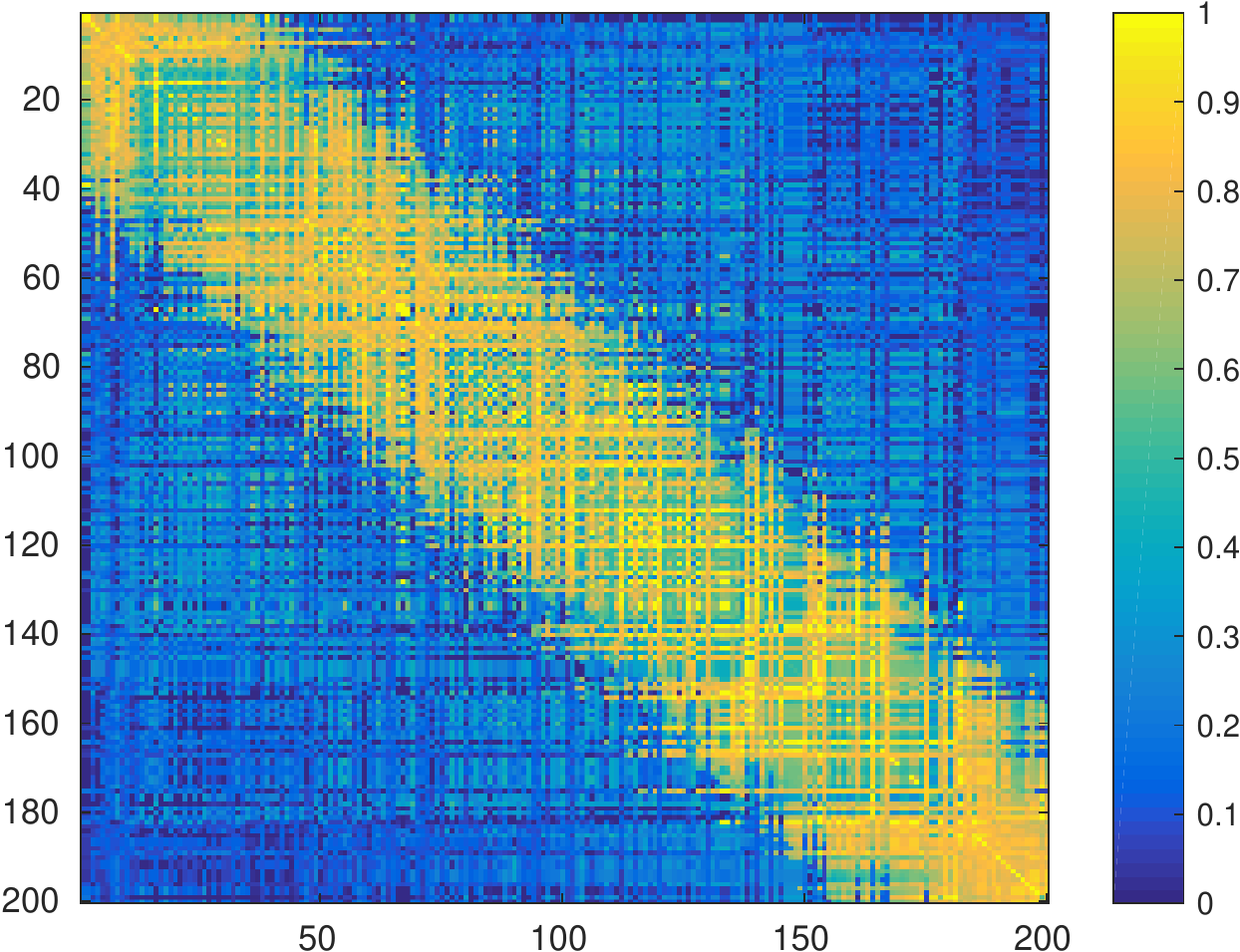}
			\caption{$\Pi_* S^{(t)} \Pi_*^T$ at step 3 of Alg~\ref{alg:AltProjBase}}\label{subfig:SerSparseBeforeRconsProj}
		\end{subfigure}
		\begin{subfigure}[htb]{0.3\textwidth}
			\includegraphics[width=\textwidth]{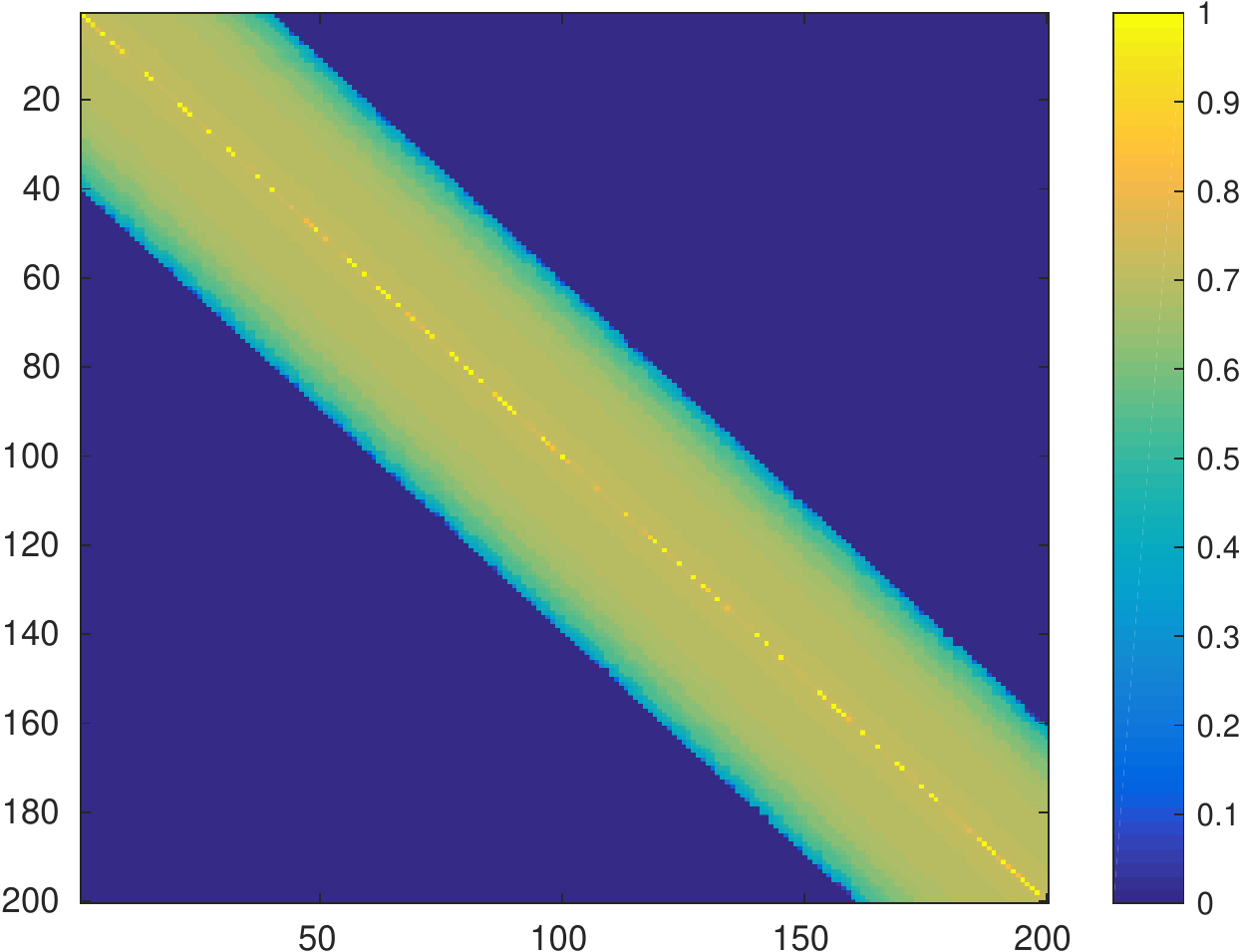}
			\caption{$S^{(t+ \frac{1}{2})}$ at step 3 of Alg~\ref{alg:AltProjBase}}\label{subfig:SerSparseAfterRconsProj}
		\end{subfigure}
		\begin{subfigure}[htb]{0.3\textwidth}
			\includegraphics[width=\textwidth]{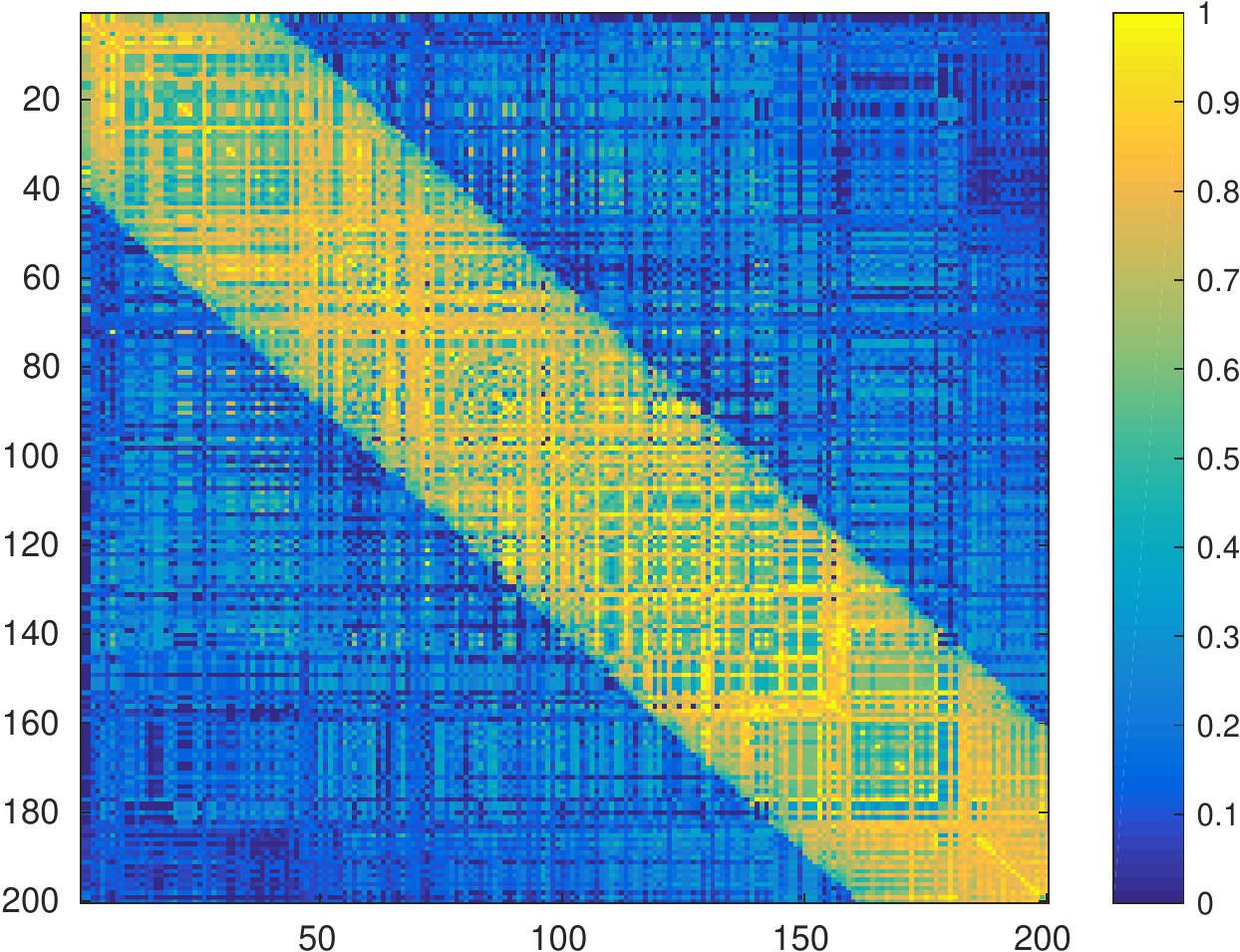}
			\caption{$S_A$ at step 4 of Alg~\ref{alg:AltProjBase}}\label{subfig:SerSparseAfterDuplconsProj}
		\end{subfigure}
		
		\caption{Three steps of Algorithm~\ref{alg:AltProjBase} for a sparse matrix $S$ with parameters $n=200$, $\delta=40$, $s=0$, $N/n=4$.}
		\label{fig:SerDuplDenseThreeStepsOfAlgorithm}
	\end{center}
\end{figure}

\begin{figure}[h]
	\begin{center}
		\centerline{\includegraphics[width=0.75\columnwidth]{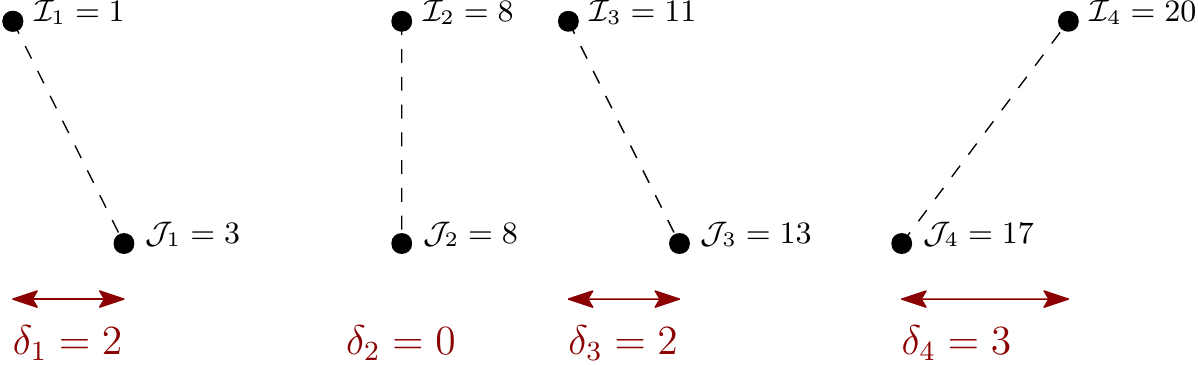}}
		\caption{Mean distance computation between two assignments $\mathcal{I}=\{1,8,11,20\}$ and $\mathcal{J}=\{3,8,13,17\}$ corresponding to the non-zeros in a given row $i$ of two assignment matrices $Z_1$ and $Z_2$.
		Before computing the $\delta_i$, a matching between $\mathcal{I}$ and $\mathcal{J}$ is performed.}
		\label{fig:matchingHungarian}
	\end{center}
	\vskip -0.2in
\end{figure}

\begin{figure}[h]
	\begin{center}
\begin{tabular}{cc}
		\includegraphics[width=0.4\textwidth]{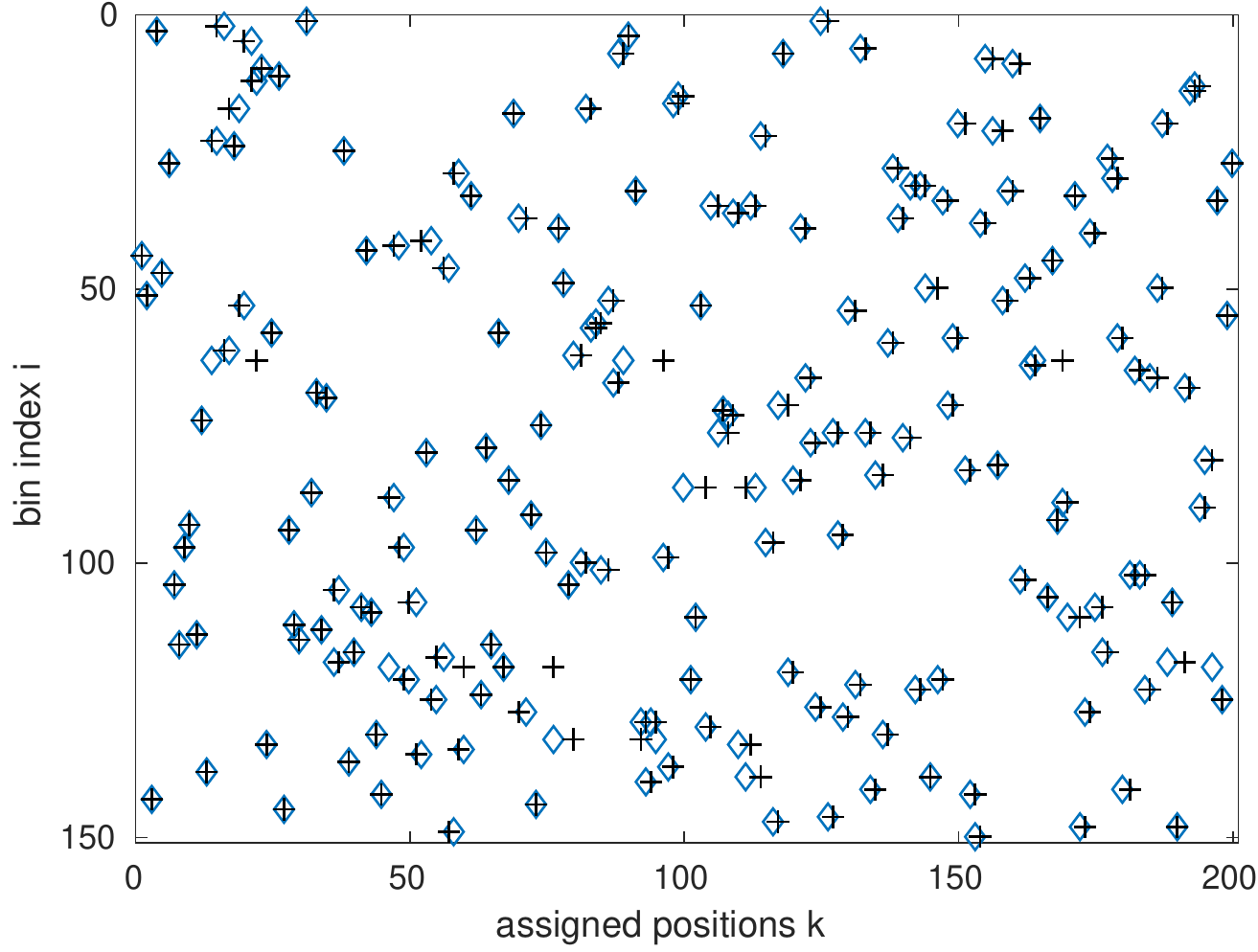}
		&
		\includegraphics[width=0.4\textwidth]{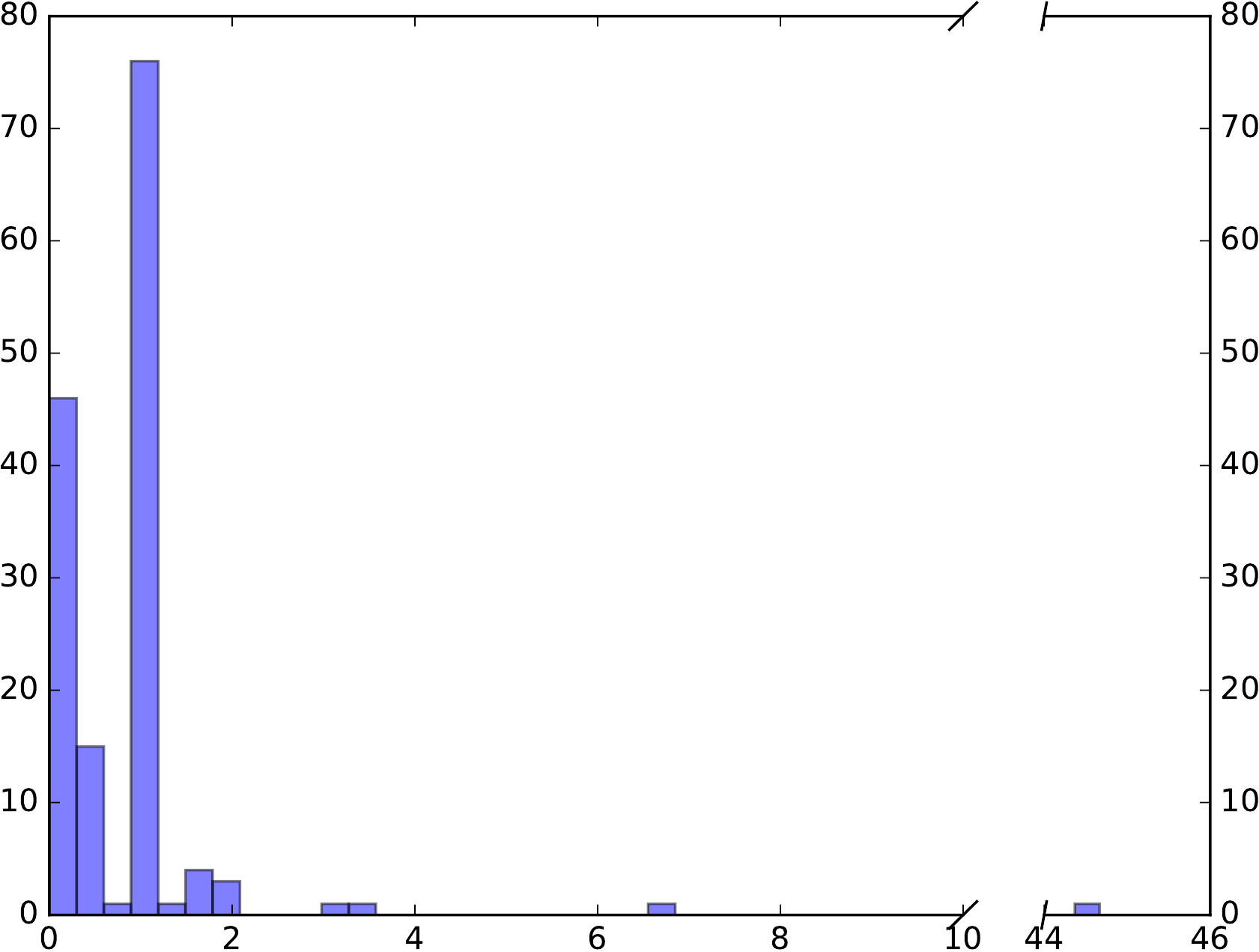}
\end{tabular}
		\caption{Plot of the true assignment matrix Z (blue diamonds) vs the one obtained with Algorithm~\ref{alg:AltProjBase} (black crosses) for an experiment with a sparse matrix $S$ with $n=200$, $\delta=n/5$. For each row, we compute the mean distance between the non-zero represented by the blue diamonds and the black crosses, as illustrated in Figure~\ref{fig:matchingHungarian}. The average over all rows of this mean distance is of 1.03 here.	{\em Left:} assignment matrices. {\em Right:} Histogram of the mean distance between the matched non-zero locations (distance between black crosses and associated blue diamond), among the rows $i$ of the two assignment matrices plotted on the left.\label{fig:ZvsZtrue}}
	\end{center}
	\vskip -0.2in
\end{figure}

\subsection{{\bf E. coli} experiment}
We expand the discussion about the results of $\eta$-Spectral on data from {\it E. coli} DNA reads.
Although we observe a few isolated outiers in Figure~\ref{fig:EcoliOrdVsTrue}, we believe their presence would almost not change the consensus sequence derived from the layout in a full OLC pipeline. Moreover, remark that the ground truth position obtained by mapping the reads to a reference genome is not error-proof (especially in repeated regions). 

While Subfigure~\ref{subfig:EcoliReorder1} strictly shows the ordering  found by the $\eta$-Spectral algorithm versus the one obtained by mapping the reads to a reference genome with BWA, Figure~\ref{fig:EcoliOrdVsTrue} includes two operations on the ordering to make the plot easier to read.
First, note that this  bacterial genome is not linear but circular, thus the reference ordering is defined up to a shift. In Figure~\ref{fig:EcoliOrdVsTrue}, we chose the shift so as to match the ordering found by $\eta$-Spectral as much as possible, in order to visualize more easily whether it resembles a straight line. Specifically, we replaced the  vector $(\pi_1, \pi_2, \ldots, \pi_{n-1}, \pi_n)$ with $(\pi_k, \pi_{k+1}, \ldots, \pi_{n-1}, \pi_n, \pi_1, \pi_2, \ldots, \pi_{k-1})$, with $k$ chosen as the breakpoint appearing on Figure~\ref{subfig:EcoliReorder1}.
Then, we flipped the permutation in order to observe a line of slope +1, \ie~we applied a flip $\pi_i \gets n+1 - \pi_i$, which is an operation that leaves the objective invariant as noted in \S~\ref{sssec:sym-issue}.

Finally, subfigure~\ref{subfig:EcoliWrongContig} zooms in a part of Subfigure~\ref{fig:EcoliOrdVsTrue} where the ordering is wrong. However, the ordering is not random in this zone, it is rather ``upside-down''. Therefore, in the ``upside-down'' part, the local ordering is correct (recall that the ordering is defined up to a shift).
It is likely that the elements in the ``upside-down'' part are weakly connected to the rest of the reads (meaning that they only overlap with some other reads at their very edges), resulting in few inconsistencies in the global ordering when this subset is flipped.

\begin{figure}[H]
	\begin{center}
		
		\begin{subfigure}[htb]{0.4\textwidth}
			\includegraphics[width=\textwidth]{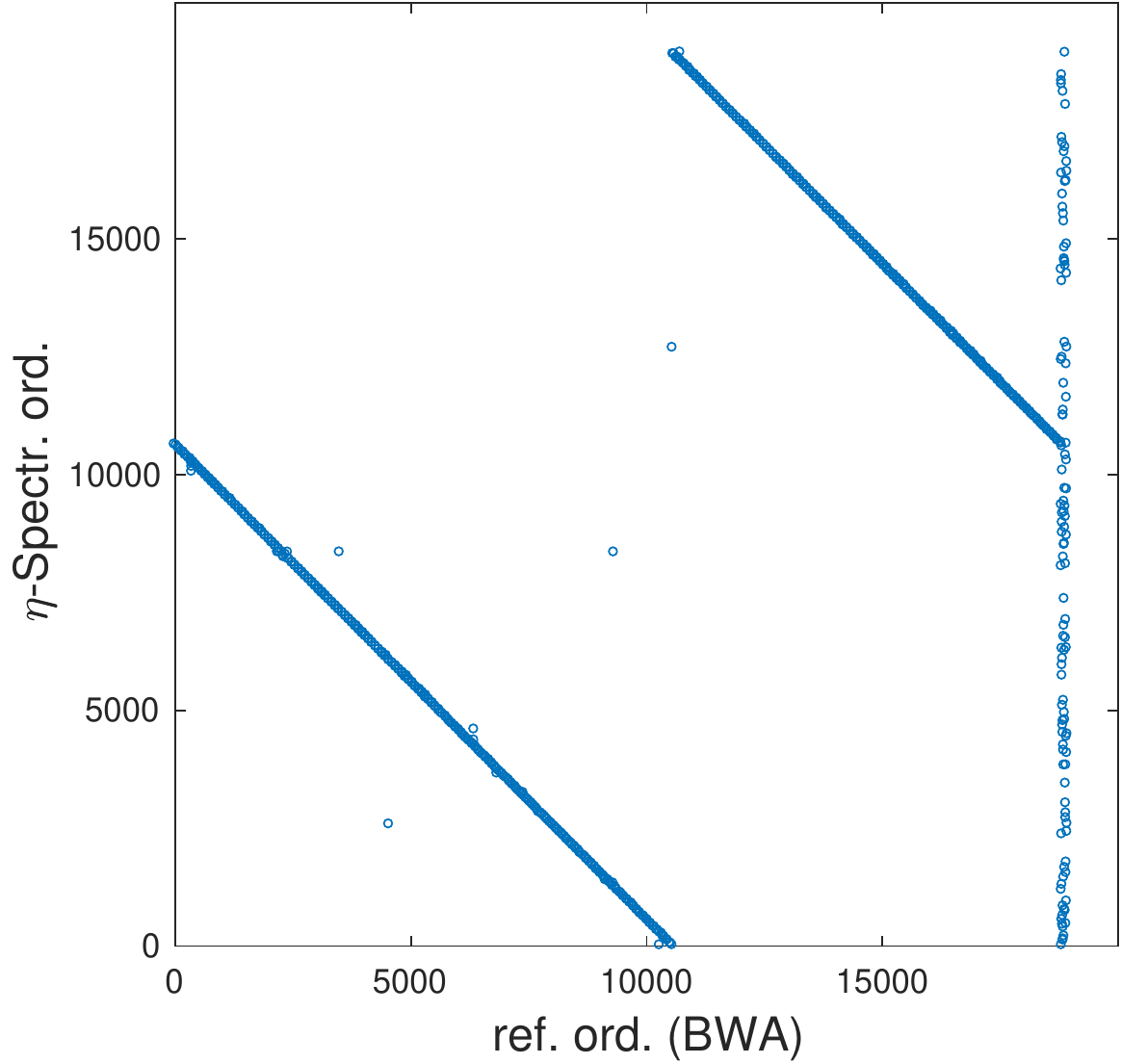}
			\caption{Ordering found by $\eta$-Spectral vs ref.}\label{subfig:EcoliReorder1}
		\end{subfigure}
\qquad
		\begin{subfigure}[htb]{0.35\textwidth}
			\includegraphics[width=\textwidth]{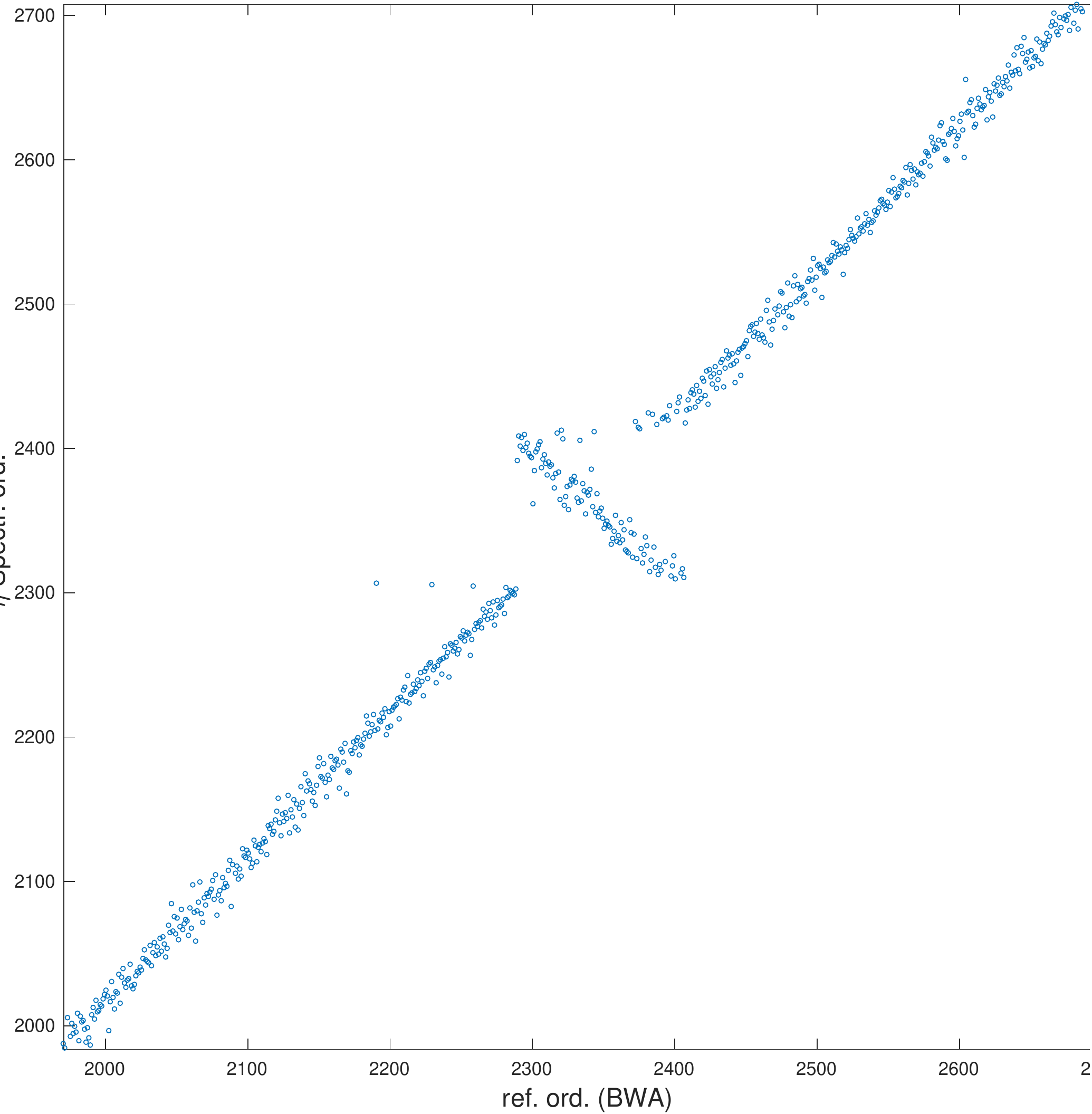}
			\caption{Zoom in wrongly reordered part}\label{subfig:EcoliWrongContig}
		\end{subfigure}

		\caption{Ordering found by the $\eta$-Spectral algorithm versus the one obtained by mapping the reads to a linearized reference genome (left), and a zoom in a wrongly reordered part (right).
		In \ref{subfig:EcoliReorder1}, we do not observe a straight line because the genome is circular, therefore a permutation is ``correct'' up to a shift. The points scattered in the top are reads that could not be mapped to the reference genome.}
	\end{center}
	\vskip -0.2in
\end{figure}

\subsection{Projecting on binary strong-R matrices.}
\begin{lemma}\label{lemm:proj-strongR-binary-is-binary}
	Given a binary symmetric matrix $S \in \{0,1\}^{n \times n}$, it has a binary projection in $\ell_1$ norm onto the set of strong-R-matrices, that is to say, there exists a solution $R \in \{0,1\}^{n \times n}$ to the following problem,
	\begin{align}\tag{R-proj}
		\BA{ll}
		\mbox{minimize} & \sum_{i,j=1}^{n} | R_{ij} - S_{ij} | \\
		\st & R \in \cR.
		\EA
	\end{align}
\end{lemma}

\begin{proof}
	Consider a given diagonal $0 \leq k \leq n-1$ in the lower triangle.
	The strong-R constraints are lower and upper bounds on the values of $R_{ij}$ on the $k$-th diagonal. Let $m_k \triangleq \min_{i,j \: : \: |i-j|=k} R_{ij}$, and $M_k \triangleq \max_{i,j \: : \: |i-j|=k} R_{ij}$.
	Recall that $S$ has only ones and zeros on the $k$-th diagonal.
	From \ref{eqn:R-proj}, $R_{ij}$ has values in $[0,1]$.
	Clearly, a solution of \ref{eqn:R-proj} satisfies,
	\begin{align*}
		R_{ij}=
		\begin{cases}
		M_{|i-j|}, & \text{if}\ S_{ij}=1 \\
		m_{|i-j|}, & \text{if}\ S_{ij}=0.
		\end{cases}
	\end{align*}
	Let $0 \leq p_k \leq n-k$ denote the number of ones on the $k$-th diagonal of $S$, and $0 \leq z_k = n-k-p_k$ the number of zeros on the $k$-th diagonal of $S$.
	Summing over all the diagonals of the matrix, the objective in \ref{eqn:R-proj} can be written as,
	\begin{align}\label{eqn:Rproj-binary-obj}
		\BA{ll}
			\|S-R\|_1 = &
			p_0 \left(1 - M_0\right) + z_0 \left( m_0 - 0 \right) + 2 \sum_{k=1}^{n-1} p_k \left(1 - M_k\right) + z_k \left( m_k - 0 \right)
		\EA
	\end{align}
	where we have separated the main diagonal from the others that are coupled with their symmetric.
	Now, we have that $0 \leq m_k \leq M_k \leq 1$ for all $0\leq k \leq n-1$.
	The strong-R constraints also require that $M_{k} \leq m_{k-1}$ for $1\leq k \leq n-1$.
	The minimizer of \ref{eqn:R-proj} saturates these constraints ($M_{k} = m_{k-1}$), and equation~\eqref{eqn:Rproj-binary-obj} can finally be written as,
	\begin{align*}
		\BA{ll}
			\|S-R\|_1 & =
			p_0 \left( 1 - M_0 \right) + z_0 m_0  + 2 \sum_{k=1}^{n-1} p_k \left( 1 - m_{k-1} \right) + z_k  m_k\\
			& = p_0 \left( 1 - M_0 \right) +  \left(z_0 - 2 p_1 \right) m_0 + 2 \sum_{k=1}^{n-1} \left(z_k - p_{k+1}\right) m_k + \sum_{k=1}^{n-1} p_k.
		\EA
	\end{align*}
	where by convention $p_n \triangleq 0$.
	\ref{eqn:R-proj} seeks to minimize this objective on the variables $\left(M_0,m_0,m_1,\ldots,m_{n-1}\right)$, under the constraints $1\geq M_0\geq m_0$, $m_{k-1} \geq m_k$ for $1 \leq k  \leq n-1$, and $m_k \geq 0$ for $0 \leq k  \leq n-1$.
	All in all, we can rewrite \ref{eqn:R-proj} as a linear program over the variable $x=\left(M_0,m_0,m_1,\ldots,m_{n-1}\right) \in \reals^{n+1}$,
	\begin{align*}
	\BA{ll}
	\mbox{minimize} & c^T x \\
	\st & A x \leq b\:,\:\: x \geq 0.
	\EA
	\end{align*}
	where 	
	\begin{align*}
		A  = \left(
		\begin{array}{ccccc}
		-1  & & & & \\
		1    &  -1   &    &   &   \\
		&  1   &  -1   &    &     \\
		&     &  \ddots   &  \ddots  &     \\
		&  &  &  1  & -1 
		\end{array}\right),
		\: \: \:
		b = \left(
		\begin{array}{c}
		 -1\\
		 0\\
		 \vdots\\
		 0
	 	\end{array}\right),
	 	\: \: \:
	 	c  =  \left(
	 	\begin{array}{c}
	 	-p0\\
	 	(z_0-2p_1)\\
	 	2(z_1-p_2)\\
	 	\vdots\\
	 	2(z_{n-1}-p_{n})
	 	\end{array}\right).
	\end{align*}
	Now, observe that $b \in \reals^{n+1}$ has integer entries, and that $A \in \reals^{(n+1)\times (n+1)}$ is totally unimodular. It follows that it has an integral solution $x^*$ \citep{papadimitriou1998combinatorial}[Th.\,13.3]. From the previous considerations, the corresponding matrix $R \in \cR$ has entries in $\{0,1\}$.
	\end{proof}

\subsection{Supplementary Tables}
\subsubsection{Robust Seriation}
Tables~\ref{tb:KT-vs-s-delta01n} and \ref{tb:KT-vs-s-delta005n} display the Kendall-$\tau$ correlation between the ordering found and the ground truth for different values of $s/s_{\text{lim}}$ and of $n$, with $\delta=n/10$ and $\delta=n/20$ respectively. For given values of $\delta /n$ and $s/s_{\text{lim}}$, the problem is easier (\ie, the methods perform better) when $n$ increases.
\begin{table}[t]
	\caption{Kendall-$\tau$ score for different values of $s/s_{\text{lim}}$, for the same methods as in Table~\ref{tb:rob-ser-synth}, for different values of $n$ (namely, $100 \:,\: 200 \:,\: 500$), and $\delta=n/10$ (namely, $10 \:,\: 20 \:,\: 50$).}
	\label{tb:KT-vs-s-delta01n}
	\begin{center}
		\begin{small}
			\begin{sc}
				\begin{tabular}{lccccccr}
					\toprule
					& & $s/s_{\text{lim}}=0.5$ & $s/s_{\text{lim}}=1$ & $s/s_{\text{lim}}=2.5$ & $s/s_{\text{lim}}=5$ & $s/s_{\text{lim}}=7.5$ & $s/s_{\text{lim}}=10$ \\
					\midrule
					\multirow{6}*{\parbox{1.3cm}{$n=100$}} & 
					spectral  & 0.91 {$\scriptstyle \pm 0.08 $} & 0.83 {$\scriptstyle \pm 0.13 $} & 0.72 {$\scriptstyle \pm 0.19 $} & 0.62 {$\scriptstyle \pm 0.21 $} & 0.55 {$\scriptstyle \pm 0.20 $} & 0.48 {$\scriptstyle \pm 0.21 $}\\
					& GnCR  & 0.92 {$\scriptstyle \pm 0.13 $} & 0.82 {$\scriptstyle \pm 0.23 $} & 0.70 {$\scriptstyle \pm 0.26 $} & 0.62 {$\scriptstyle \pm 0.26 $} & 0.55 {$\scriptstyle \pm 0.25 $} & 0.48 {$\scriptstyle \pm 0.24 $}\\
					& FAQ  & 0.93 {$\scriptstyle \pm 0.09 $} & 0.85 {$\scriptstyle \pm 0.17 $} & 0.72 {$\scriptstyle \pm 0.24 $} & 0.61 {$\scriptstyle \pm 0.25 $} & 0.55 {$\scriptstyle \pm 0.25 $} & 0.48 {$\scriptstyle \pm 0.23 $}\\
					& LWCD  & 0.93 {$\scriptstyle \pm 0.10 $} & 0.85 {$\scriptstyle \pm 0.17 $} & 0.72 {$\scriptstyle \pm 0.24 $} & 0.61 {$\scriptstyle \pm 0.25 $} & 0.55 {$\scriptstyle \pm 0.25 $} & 0.48 {$\scriptstyle \pm 0.23 $}\\
					& UBI  & 0.92 {$\scriptstyle \pm 0.09 $} & 0.85 {$\scriptstyle \pm 0.16 $} & 0.73 {$\scriptstyle \pm 0.24 $} & 0.62 {$\scriptstyle \pm 0.24 $} & 0.56 {$\scriptstyle \pm 0.24 $} & 0.49 {$\scriptstyle \pm 0.23 $}\\
					& Manopt  & 0.92 {$\scriptstyle \pm 0.08 $} & 0.84 {$\scriptstyle \pm 0.13 $} & 0.72 {$\scriptstyle \pm 0.19 $} & 0.62 {$\scriptstyle \pm 0.21 $} & 0.55 {$\scriptstyle \pm 0.20 $} & 0.48 {$\scriptstyle \pm 0.21 $}\\
					\midrule
					\multirow{6}*{\parbox{1.3cm}{$n=100$}} & 
					$\eta$-Spectr.  & 0.99 {$\scriptstyle \pm 0.00 $} & 0.98 {$\scriptstyle \pm 0.00 $} & 0.89 {$\scriptstyle \pm 0.17 $} & 0.74 {$\scriptstyle \pm 0.25 $} & 0.65 {$\scriptstyle \pm 0.26 $} & 0.56 {$\scriptstyle \pm 0.26 $}\\
					& HGnCR  & 0.98 {$\scriptstyle \pm 0.06 $} & 0.96 {$\scriptstyle \pm 0.14 $} & 0.80 {$\scriptstyle \pm 0.25 $} & 0.65 {$\scriptstyle \pm 0.30 $} & 0.54 {$\scriptstyle \pm 0.29 $} & 0.49 {$\scriptstyle \pm 0.29 $}\\
					& H-FAQ  & 0.97 {$\scriptstyle \pm 0.09 $} & 0.90 {$\scriptstyle \pm 0.16 $} & 0.80 {$\scriptstyle \pm 0.25 $} & 0.70 {$\scriptstyle \pm 0.29 $} & 0.64 {$\scriptstyle \pm 0.28 $} & 0.55 {$\scriptstyle \pm 0.26 $}\\
					& H-LWCD  & 0.97 {$\scriptstyle \pm 0.09 $} & 0.90 {$\scriptstyle \pm 0.16 $} & 0.80 {$\scriptstyle \pm 0.25 $} & 0.70 {$\scriptstyle \pm 0.29 $} & 0.65 {$\scriptstyle \pm 0.28 $} & 0.55 {$\scriptstyle \pm 0.28 $}\\
					& H-UBI  & 0.99 {$\scriptstyle \pm 0.00 $} & 0.98 {$\scriptstyle \pm 0.04 $} & 0.88 {$\scriptstyle \pm 0.20 $} & 0.75 {$\scriptstyle \pm 0.25 $} & 0.62 {$\scriptstyle \pm 0.26 $} & 0.54 {$\scriptstyle \pm 0.25 $}\\
					& H-Manopt  & 0.98 {$\scriptstyle \pm 0.05 $} & 0.91 {$\scriptstyle \pm 0.14 $} & 0.78 {$\scriptstyle \pm 0.23 $} & 0.65 {$\scriptstyle \pm 0.24 $} & 0.56 {$\scriptstyle \pm 0.21 $} & 0.48 {$\scriptstyle \pm 0.21 $}\\
					\midrule
					\multirow{2}*{\parbox{1.3cm}{$n=100$}} & 
					R-FAQ  & 0.96 {$\scriptstyle \pm 0.09 $} & 0.91 {$\scriptstyle \pm 0.16 $} & 0.80 {$\scriptstyle \pm 0.25 $} & 0.70 {$\scriptstyle \pm 0.28 $} & 0.65 {$\scriptstyle \pm 0.27 $} & 0.54 {$\scriptstyle \pm 0.28 $}\\
					& R-LWCD  & 0.95 {$\scriptstyle \pm 0.09 $} & 0.89 {$\scriptstyle \pm 0.17 $} & 0.78 {$\scriptstyle \pm 0.24 $} & 0.69 {$\scriptstyle \pm 0.28 $} & 0.62 {$\scriptstyle \pm 0.28 $} & 0.53 {$\scriptstyle \pm 0.28 $}\\
					\bottomrule
					\toprule
					\multirow{6}*{\parbox{1.3cm}{$n=200$}} & 
					spectral  & 0.96 {$\scriptstyle \pm 0.01 $} & 0.95 {$\scriptstyle \pm 0.01 $} & 0.91 {$\scriptstyle \pm 0.03 $} & 0.86 {$\scriptstyle \pm 0.06 $} & 0.84 {$\scriptstyle \pm 0.06 $} & 0.80 {$\scriptstyle \pm 0.09 $}\\
					& GnCR  & 0.98 {$\scriptstyle \pm 0.00 $} & 0.96 {$\scriptstyle \pm 0.04 $} & 0.93 {$\scriptstyle \pm 0.07 $} & 0.87 {$\scriptstyle \pm 0.15 $} & 0.81 {$\scriptstyle \pm 0.20 $} & 0.80 {$\scriptstyle \pm 0.18 $}\\
					& FAQ  & 0.98 {$\scriptstyle \pm 0.00 $} & 0.97 {$\scriptstyle \pm 0.00 $} & 0.94 {$\scriptstyle \pm 0.02 $} & 0.89 {$\scriptstyle \pm 0.08 $} & 0.87 {$\scriptstyle \pm 0.08 $} & 0.82 {$\scriptstyle \pm 0.13 $}\\
					& LWCD  & 0.98 {$\scriptstyle \pm 0.00 $} & 0.97 {$\scriptstyle \pm 0.00 $} & 0.94 {$\scriptstyle \pm 0.02 $} & 0.89 {$\scriptstyle \pm 0.08 $} & 0.87 {$\scriptstyle \pm 0.08 $} & 0.82 {$\scriptstyle \pm 0.13 $}\\
					& UBI  & 0.97 {$\scriptstyle \pm 0.00 $} & 0.96 {$\scriptstyle \pm 0.01 $} & 0.92 {$\scriptstyle \pm 0.03 $} & 0.89 {$\scriptstyle \pm 0.06 $} & 0.86 {$\scriptstyle \pm 0.07 $} & 0.82 {$\scriptstyle \pm 0.12 $}\\
					& Manopt  & 0.97 {$\scriptstyle \pm 0.00 $} & 0.95 {$\scriptstyle \pm 0.01 $} & 0.91 {$\scriptstyle \pm 0.03 $} & 0.86 {$\scriptstyle \pm 0.06 $} & 0.84 {$\scriptstyle \pm 0.06 $} & 0.80 {$\scriptstyle \pm 0.09 $}\\
					\midrule
					\multirow{6}*{\parbox{1.3cm}{$n=200$}} & 
					$\eta$-Spectr.  & 0.99 {$\scriptstyle \pm 0.00 $} & 0.99 {$\scriptstyle \pm 0.00 $} & 0.98 {$\scriptstyle \pm 0.00 $} & 0.97 {$\scriptstyle \pm 0.00 $} & 0.96 {$\scriptstyle \pm 0.00 $} & 0.94 {$\scriptstyle \pm 0.06 $}\\
					& HGnCR  & 1.00 {$\scriptstyle \pm 0.00 $} & 0.99 {$\scriptstyle \pm 0.00 $} & 0.99 {$\scriptstyle \pm 0.00 $} & 0.89 {$\scriptstyle \pm 0.22 $} & 0.85 {$\scriptstyle \pm 0.23 $} & 0.83 {$\scriptstyle \pm 0.25 $}\\
					& H-FAQ  & 1.00 {$\scriptstyle \pm 0.00 $} & 1.00 {$\scriptstyle \pm 0.00 $} & 0.99 {$\scriptstyle \pm 0.01 $} & 0.95 {$\scriptstyle \pm 0.08 $} & 0.94 {$\scriptstyle \pm 0.09 $} & 0.91 {$\scriptstyle \pm 0.13 $}\\
					& H-LWCD  & 1.00 {$\scriptstyle \pm 0.00 $} & 1.00 {$\scriptstyle \pm 0.00 $} & 0.99 {$\scriptstyle \pm 0.02 $} & 0.94 {$\scriptstyle \pm 0.09 $} & 0.94 {$\scriptstyle \pm 0.09 $} & 0.90 {$\scriptstyle \pm 0.14 $}\\
					& H-UBI  & 0.99 {$\scriptstyle \pm 0.00 $} & 0.99 {$\scriptstyle \pm 0.00 $} & 0.98 {$\scriptstyle \pm 0.00 $} & 0.97 {$\scriptstyle \pm 0.00 $} & 0.96 {$\scriptstyle \pm 0.01 $} & 0.94 {$\scriptstyle \pm 0.03 $}\\
					& H-Manopt  & 1.00 {$\scriptstyle \pm 0.00 $} & 0.99 {$\scriptstyle \pm 0.00 $} & 0.97 {$\scriptstyle \pm 0.02 $} & 0.92 {$\scriptstyle \pm 0.06 $} & 0.89 {$\scriptstyle \pm 0.07 $} & 0.84 {$\scriptstyle \pm 0.10 $}\\
					\midrule
					\multirow{2}*{\parbox{1.3cm}{$n=200$}} & 
					R-FAQ  & 1.00 {$\scriptstyle \pm 0.00 $} & 1.00 {$\scriptstyle \pm 0.00 $} & 0.99 {$\scriptstyle \pm 0.04 $} & 0.95 {$\scriptstyle \pm 0.10 $} & 0.94 {$\scriptstyle \pm 0.10 $} & 0.90 {$\scriptstyle \pm 0.15 $}\\
					& R-LWCD  & 0.99 {$\scriptstyle \pm 0.00 $} & 1.00 {$\scriptstyle \pm 0.00 $} & 0.99 {$\scriptstyle \pm 0.04 $} & 0.94 {$\scriptstyle \pm 0.09 $} & 0.94 {$\scriptstyle \pm 0.10 $} & 0.90 {$\scriptstyle \pm 0.16 $}\\
					\bottomrule
					\toprule
					\multirow{6}*{\parbox{1.3cm}{$n=500$}} & 
					spectral  & 0.98 {$\scriptstyle \pm 0.00 $} & 0.98 {$\scriptstyle \pm 0.00 $} & 0.96 {$\scriptstyle \pm 0.00 $} & 0.95 {$\scriptstyle \pm 0.01 $} & 0.94 {$\scriptstyle \pm 0.01 $} & 0.93 {$\scriptstyle \pm 0.01 $}\\
					& GnCR  & 0.99 {$\scriptstyle \pm 0.00 $} & 0.99 {$\scriptstyle \pm 0.00 $} & 0.98 {$\scriptstyle \pm 0.00 $} & 0.97 {$\scriptstyle \pm 0.00 $} & 0.96 {$\scriptstyle \pm 0.00 $} & 0.95 {$\scriptstyle \pm 0.05 $}\\
					& FAQ  & 0.99 {$\scriptstyle \pm 0.00 $} & 0.99 {$\scriptstyle \pm 0.00 $} & 0.98 {$\scriptstyle \pm 0.00 $} & 0.97 {$\scriptstyle \pm 0.00 $} & 0.96 {$\scriptstyle \pm 0.00 $} & 0.95 {$\scriptstyle \pm 0.00 $}\\
					& LWCD  & 0.99 {$\scriptstyle \pm 0.00 $} & 0.99 {$\scriptstyle \pm 0.00 $} & 0.98 {$\scriptstyle \pm 0.00 $} & 0.97 {$\scriptstyle \pm 0.00 $} & 0.96 {$\scriptstyle \pm 0.00 $} & 0.95 {$\scriptstyle \pm 0.00 $}\\
					& UBI  & 0.99 {$\scriptstyle \pm 0.00 $} & 0.98 {$\scriptstyle \pm 0.00 $} & 0.97 {$\scriptstyle \pm 0.00 $} & 0.96 {$\scriptstyle \pm 0.01 $} & 0.95 {$\scriptstyle \pm 0.00 $} & 0.94 {$\scriptstyle \pm 0.00 $}\\
					& Manopt  & 0.99 {$\scriptstyle \pm 0.00 $} & 0.98 {$\scriptstyle \pm 0.00 $} & 0.97 {$\scriptstyle \pm 0.00 $} & 0.95 {$\scriptstyle \pm 0.00 $} & 0.94 {$\scriptstyle \pm 0.01 $} & 0.93 {$\scriptstyle \pm 0.01 $}\\
					\midrule
					\multirow{6}*{\parbox{1.3cm}{$n=500$}} & 
					$\eta$-Spectr.  & 1.00 {$\scriptstyle \pm 0.00 $} & 1.00 {$\scriptstyle \pm 0.00 $} & 0.99 {$\scriptstyle \pm 0.00 $} & 0.99 {$\scriptstyle \pm 0.00 $} & 0.99 {$\scriptstyle \pm 0.00 $} & 0.98 {$\scriptstyle \pm 0.00 $}\\
					& HGnCR  & 1.00 {$\scriptstyle \pm 0.00 $} & 1.00 {$\scriptstyle \pm 0.00 $} & 0.99 {$\scriptstyle \pm 0.00 $} & 0.99 {$\scriptstyle \pm 0.00 $} & 0.99 {$\scriptstyle \pm 0.00 $} & 0.99 {$\scriptstyle \pm 0.00 $}\\
					& H-FAQ  & 1.00 {$\scriptstyle \pm 0.00 $} & 1.00 {$\scriptstyle \pm 0.00 $} & 1.00 {$\scriptstyle \pm 0.00 $} & 0.99 {$\scriptstyle \pm 0.00 $} & 0.99 {$\scriptstyle \pm 0.00 $} & 0.99 {$\scriptstyle \pm 0.00 $}\\
					& H-LWCD  & 1.00 {$\scriptstyle \pm 0.00 $} & 1.00 {$\scriptstyle \pm 0.00 $} & 1.00 {$\scriptstyle \pm 0.00 $} & 0.99 {$\scriptstyle \pm 0.00 $} & 0.99 {$\scriptstyle \pm 0.00 $} & 0.99 {$\scriptstyle \pm 0.00 $}\\
					& H-UBI  & 1.00 {$\scriptstyle \pm 0.00 $} & 1.00 {$\scriptstyle \pm 0.00 $} & 0.99 {$\scriptstyle \pm 0.00 $} & 0.99 {$\scriptstyle \pm 0.00 $} & 0.99 {$\scriptstyle \pm 0.00 $} & 0.98 {$\scriptstyle \pm 0.00 $}\\
					& H-Manopt  & 1.00 {$\scriptstyle \pm 0.00 $} & 1.00 {$\scriptstyle \pm 0.00 $} & 1.00 {$\scriptstyle \pm 0.00 $} & 0.99 {$\scriptstyle \pm 0.00 $} & 0.98 {$\scriptstyle \pm 0.01 $} & 0.97 {$\scriptstyle \pm 0.01 $}\\
					\midrule
					\multirow{2}*{\parbox{1.3cm}{$n=500$}} & 
					R-FAQ  & 1.00 {$\scriptstyle \pm 0.00 $} & 1.00 {$\scriptstyle \pm 0.00 $} & 1.00 {$\scriptstyle \pm 0.00 $} & 1.00 {$\scriptstyle \pm 0.00 $} & 1.00 {$\scriptstyle \pm 0.00 $} & 1.00 {$\scriptstyle \pm 0.00 $}\\
					& R-LWCD  & 1.00 {$\scriptstyle \pm 0.00 $} & 1.00 {$\scriptstyle \pm 0.00 $} & 1.00 {$\scriptstyle \pm 0.00 $} & 1.00 {$\scriptstyle \pm 0.00 $} & 1.00 {$\scriptstyle \pm 0.00 $} & 1.00 {$\scriptstyle \pm 0.01 $}\\
					\bottomrule
				\end{tabular}
			\end{sc}
		\end{small}
	\end{center}
	\vskip -.1in
\end{table}

\begin{table}[t]
	\caption{Kendall-$\tau$ score for different values of $s/s_{\text{lim}}$, for the same methods as in Table~\ref{tb:rob-ser-synth}, for different values of $n$ (namely, $100 \:,\: 200 \:,\: 500$), and $\delta=n/20$ (namely, $5 \:,\: 10 \:,\: 25$).}
	\label{tb:KT-vs-s-delta005n}
	\begin{center}
		\begin{small}
			\begin{sc}
				\begin{tabular}{lccccccr}
					\toprule
					& & $s/s_{\text{lim}}=0.5$ & $s/s_{\text{lim}}=1$ & $s/s_{\text{lim}}=2.5$ & $s/s_{\text{lim}}=5$ & $s/s_{\text{lim}}=7.5$ & $s/s_{\text{lim}}=10$ \\
					\midrule
					\multirow{6}*{\parbox{1.3cm}{$n=100$}} & 
					spectral  & 0.46 {$\scriptstyle \pm0.24 $} & 0.39 {$\scriptstyle \pm0.21 $} & 0.31 {$\scriptstyle \pm0.20 $} & 0.25 {$\scriptstyle \pm0.16 $} & 0.22 {$\scriptstyle \pm0.15 $} & 0.20 {$\scriptstyle \pm0.14 $}\\
					& GnCR  & 0.43 {$\scriptstyle \pm0.28 $} & 0.37 {$\scriptstyle \pm0.21 $} & 0.32 {$\scriptstyle \pm0.21 $} & 0.25 {$\scriptstyle \pm0.16 $} & 0.25 {$\scriptstyle \pm0.14 $} & 0.20 {$\scriptstyle \pm0.13 $}\\
					& FAQ  & 0.45 {$\scriptstyle \pm0.25 $} & 0.39 {$\scriptstyle \pm0.22 $} & 0.31 {$\scriptstyle \pm0.21 $} & 0.25 {$\scriptstyle \pm0.17 $} & 0.23 {$\scriptstyle \pm0.15 $} & 0.22 {$\scriptstyle \pm0.14 $}\\
					& LWCD  & 0.45 {$\scriptstyle \pm0.26 $} & 0.39 {$\scriptstyle \pm0.22 $} & 0.31 {$\scriptstyle \pm0.21 $} & 0.25 {$\scriptstyle \pm0.17 $} & 0.23 {$\scriptstyle \pm0.15 $} & 0.22 {$\scriptstyle \pm0.14 $}\\
					& UBI  & 0.45 {$\scriptstyle \pm0.26 $} & 0.40 {$\scriptstyle \pm0.22 $} & 0.32 {$\scriptstyle \pm0.21 $} & 0.26 {$\scriptstyle \pm0.17 $} & 0.23 {$\scriptstyle \pm0.15 $} & 0.23 {$\scriptstyle \pm0.14 $}\\
					& Manopt  & 0.46 {$\scriptstyle \pm0.25 $} & 0.40 {$\scriptstyle \pm0.21 $} & 0.31 {$\scriptstyle \pm0.20 $} & 0.25 {$\scriptstyle \pm0.16 $} & 0.22 {$\scriptstyle \pm0.15 $} & 0.21 {$\scriptstyle \pm0.14 $}\\
					\midrule
					\multirow{6}*{\parbox{1.3cm}{$n=100$}} & 
					$\eta$-Spectr.  & 0.65 {$\scriptstyle \pm0.33 $} & 0.50 {$\scriptstyle \pm0.28 $} & 0.37 {$\scriptstyle \pm0.24 $} & 0.28 {$\scriptstyle \pm0.19 $} & 0.25 {$\scriptstyle \pm0.16 $} & 0.23 {$\scriptstyle \pm0.16 $}\\
					& HGnCR  & 0.53 {$\scriptstyle \pm0.31 $} & 0.43 {$\scriptstyle \pm0.26 $} & 0.36 {$\scriptstyle \pm0.22 $} & 0.25 {$\scriptstyle \pm0.17 $} & 0.22 {$\scriptstyle \pm0.15 $} & 0.18 {$\scriptstyle \pm0.14 $}\\
					& H-FAQ  & 0.48 {$\scriptstyle \pm0.26 $} & 0.41 {$\scriptstyle \pm0.23 $} & 0.33 {$\scriptstyle \pm0.23 $} & 0.28 {$\scriptstyle \pm0.17 $} & 0.24 {$\scriptstyle \pm0.16 $} & 0.23 {$\scriptstyle \pm0.15 $}\\
					& H-LWCD  & 0.49 {$\scriptstyle \pm0.27 $} & 0.42 {$\scriptstyle \pm0.23 $} & 0.34 {$\scriptstyle \pm0.23 $} & 0.28 {$\scriptstyle \pm0.18 $} & 0.24 {$\scriptstyle \pm0.16 $} & 0.23 {$\scriptstyle \pm0.16 $}\\
					& H-UBI  & 0.60 {$\scriptstyle \pm0.35 $} & 0.52 {$\scriptstyle \pm0.29 $} & 0.40 {$\scriptstyle \pm0.26 $} & 0.28 {$\scriptstyle \pm0.19 $} & 0.25 {$\scriptstyle \pm0.16 $} & 0.23 {$\scriptstyle \pm0.15 $}\\
					& H-Manopt  & 0.54 {$\scriptstyle \pm0.30 $} & 0.44 {$\scriptstyle \pm0.25 $} & 0.33 {$\scriptstyle \pm0.22 $} & 0.25 {$\scriptstyle \pm0.16 $} & 0.22 {$\scriptstyle \pm0.15 $} & 0.21 {$\scriptstyle \pm0.14 $}\\
					\midrule
					\multirow{2}*{\parbox{1.3cm}{$n=100$}} & 
					R-FAQ  & 0.48 {$\scriptstyle \pm0.25 $} & 0.41 {$\scriptstyle \pm0.22 $} & 0.33 {$\scriptstyle \pm0.21 $} & 0.26 {$\scriptstyle \pm0.18 $} & 0.23 {$\scriptstyle \pm0.15 $} & 0.23 {$\scriptstyle \pm0.16 $}\\
					& R-LWCD  & 0.47 {$\scriptstyle \pm0.24 $} & 0.41 {$\scriptstyle \pm0.22 $} & 0.32 {$\scriptstyle \pm0.21 $} & 0.25 {$\scriptstyle \pm0.16 $} & 0.22 {$\scriptstyle \pm0.15 $} & 0.22 {$\scriptstyle \pm0.15 $}\\
					\bottomrule
					\toprule
					\multirow{6}*{\parbox{1.3cm}{$n=200$}} & 
					spectral  & 0.72 {$\scriptstyle \pm0.21 $} & 0.59 {$\scriptstyle \pm0.24 $} & 0.49 {$\scriptstyle \pm0.26 $} & 0.42 {$\scriptstyle \pm0.23 $} & 0.35 {$\scriptstyle \pm0.20 $} & 0.31 {$\scriptstyle \pm0.18 $}\\
					& GnCR  & 0.69 {$\scriptstyle \pm0.29 $} & 0.56 {$\scriptstyle \pm0.31 $} & 0.45 {$\scriptstyle \pm0.26 $} & 0.37 {$\scriptstyle \pm0.27 $} & 0.34 {$\scriptstyle \pm0.22 $} & 0.32 {$\scriptstyle \pm0.23 $}\\
					& FAQ  & 0.72 {$\scriptstyle \pm0.24 $} & 0.60 {$\scriptstyle \pm0.26 $} & 0.49 {$\scriptstyle \pm0.26 $} & 0.41 {$\scriptstyle \pm0.24 $} & 0.35 {$\scriptstyle \pm0.21 $} & 0.33 {$\scriptstyle \pm0.20 $}\\
					& LWCD  & 0.72 {$\scriptstyle \pm0.24 $} & 0.60 {$\scriptstyle \pm0.26 $} & 0.49 {$\scriptstyle \pm0.27 $} & 0.42 {$\scriptstyle \pm0.25 $} & 0.36 {$\scriptstyle \pm0.21 $} & 0.33 {$\scriptstyle \pm0.20 $}\\
					& UBI  & 0.73 {$\scriptstyle \pm0.26 $} & 0.59 {$\scriptstyle \pm0.28 $} & 0.50 {$\scriptstyle \pm0.28 $} & 0.42 {$\scriptstyle \pm0.25 $} & 0.35 {$\scriptstyle \pm0.21 $} & 0.33 {$\scriptstyle \pm0.21 $}\\
					& Manopt  & 0.72 {$\scriptstyle \pm0.22 $} & 0.59 {$\scriptstyle \pm0.24 $} & 0.49 {$\scriptstyle \pm0.26 $} & 0.42 {$\scriptstyle \pm0.24 $} & 0.35 {$\scriptstyle \pm0.20 $} & 0.31 {$\scriptstyle \pm0.18 $}\\
					\midrule
					\multirow{6}*{\parbox{1.3cm}{$n=200$}} & 
					$\eta$-Spectr.  & 0.99 {$\scriptstyle \pm0.00 $} & 0.91 {$\scriptstyle \pm0.21 $} & 0.65 {$\scriptstyle \pm0.33 $} & 0.52 {$\scriptstyle \pm0.30 $} & 0.41 {$\scriptstyle \pm0.25 $} & 0.37 {$\scriptstyle \pm0.23 $}\\
					& HGnCR  & 0.73 {$\scriptstyle \pm0.33 $} & 0.61 {$\scriptstyle \pm0.32 $} & 0.50 {$\scriptstyle \pm0.31 $} & 0.44 {$\scriptstyle \pm0.29 $} & 0.39 {$\scriptstyle \pm0.25 $} & 0.35 {$\scriptstyle \pm0.22 $}\\
					& H-FAQ  & 0.75 {$\scriptstyle \pm0.24 $} & 0.63 {$\scriptstyle \pm0.27 $} & 0.53 {$\scriptstyle \pm0.29 $} & 0.46 {$\scriptstyle \pm0.27 $} & 0.38 {$\scriptstyle \pm0.23 $} & 0.35 {$\scriptstyle \pm0.23 $}\\
					& H-LWCD  & 0.75 {$\scriptstyle \pm0.23 $} & 0.62 {$\scriptstyle \pm0.27 $} & 0.53 {$\scriptstyle \pm0.29 $} & 0.46 {$\scriptstyle \pm0.27 $} & 0.38 {$\scriptstyle \pm0.23 $} & 0.35 {$\scriptstyle \pm0.22 $}\\
					& H-UBI  & 0.94 {$\scriptstyle \pm0.19 $} & 0.82 {$\scriptstyle \pm0.30 $} & 0.69 {$\scriptstyle \pm0.34 $} & 0.57 {$\scriptstyle \pm0.32 $} & 0.46 {$\scriptstyle \pm0.28 $} & 0.40 {$\scriptstyle \pm0.23 $}\\
					& H-Manopt  & 0.84 {$\scriptstyle \pm0.23 $} & 0.67 {$\scriptstyle \pm0.29 $} & 0.54 {$\scriptstyle \pm0.29 $} & 0.45 {$\scriptstyle \pm0.26 $} & 0.36 {$\scriptstyle \pm0.21 $} & 0.31 {$\scriptstyle \pm0.19 $}\\
					\midrule
					\multirow{2}*{\parbox{1.3cm}{$n=200$}} & 
					R-FAQ  & 0.75 {$\scriptstyle \pm0.23 $} & 0.62 {$\scriptstyle \pm0.26 $} & 0.53 {$\scriptstyle \pm0.28 $} & 0.45 {$\scriptstyle \pm0.27 $} & 0.38 {$\scriptstyle \pm0.23 $} & 0.33 {$\scriptstyle \pm0.23 $}\\
					& R-LWCD  & 0.74 {$\scriptstyle \pm0.22 $} & 0.62 {$\scriptstyle \pm0.25 $} & 0.51 {$\scriptstyle \pm0.27 $} & 0.44 {$\scriptstyle \pm0.25 $} & 0.37 {$\scriptstyle \pm0.23 $} & 0.33 {$\scriptstyle \pm0.21 $}\\
					\bottomrule
					\toprule
					\multirow{6}*{\parbox{1.3cm}{$n=500$}} & 
					spectral  & 0.96 {$\scriptstyle \pm0.03 $} & 0.93 {$\scriptstyle \pm0.05 $} & 0.86 {$\scriptstyle \pm0.11 $} & 0.76 {$\scriptstyle \pm0.18 $} & 0.71 {$\scriptstyle \pm0.19 $} & 0.67 {$\scriptstyle \pm0.21 $}\\
					& GnCR  & 0.90 {$\scriptstyle \pm0.21 $} & 0.80 {$\scriptstyle \pm0.28 $} & 0.71 {$\scriptstyle \pm0.31 $} & 0.60 {$\scriptstyle \pm0.31 $} & 0.62 {$\scriptstyle \pm0.29 $} & 0.55 {$\scriptstyle \pm0.31 $}\\
					& FAQ  & 0.98 {$\scriptstyle \pm0.03 $} & 0.95 {$\scriptstyle \pm0.06 $} & 0.87 {$\scriptstyle \pm0.13 $} & 0.76 {$\scriptstyle \pm0.21 $} & 0.72 {$\scriptstyle \pm0.22 $} & 0.67 {$\scriptstyle \pm0.24 $}\\
					& LWCD  & 0.98 {$\scriptstyle \pm0.03 $} & 0.95 {$\scriptstyle \pm0.06 $} & 0.87 {$\scriptstyle \pm0.13 $} & 0.76 {$\scriptstyle \pm0.21 $} & 0.72 {$\scriptstyle \pm0.22 $} & 0.67 {$\scriptstyle \pm0.24 $}\\
					& UBI  & 0.97 {$\scriptstyle \pm0.02 $} & 0.95 {$\scriptstyle \pm0.04 $} & 0.88 {$\scriptstyle \pm0.14 $} & 0.76 {$\scriptstyle \pm0.24 $} & 0.71 {$\scriptstyle \pm0.25 $} & 0.67 {$\scriptstyle \pm0.25 $}\\
					& Manopt  & 0.97 {$\scriptstyle \pm0.03 $} & 0.94 {$\scriptstyle \pm0.06 $} & 0.86 {$\scriptstyle \pm0.12 $} & 0.76 {$\scriptstyle \pm0.18 $} & 0.72 {$\scriptstyle \pm0.19 $} & 0.67 {$\scriptstyle \pm0.22 $}\\
					\midrule
					\multirow{6}*{\parbox{1.3cm}{$n=500$}} & 
					$\eta$-Spectr.  & 1.00 {$\scriptstyle \pm0.00 $} & 1.00 {$\scriptstyle \pm0.00 $} & 0.99 {$\scriptstyle \pm0.00 $} & 0.96 {$\scriptstyle \pm0.12 $} & 0.88 {$\scriptstyle \pm0.18 $} & 0.81 {$\scriptstyle \pm0.24 $}\\
					& HGnCR  & 1.00 {$\scriptstyle \pm0.00 $} & 0.96 {$\scriptstyle \pm0.18 $} & 0.87 {$\scriptstyle \pm0.28 $} & 0.80 {$\scriptstyle \pm0.32 $} & 0.70 {$\scriptstyle \pm0.36 $} & 0.75 {$\scriptstyle \pm0.33 $}\\
					& H-FAQ  & 0.99 {$\scriptstyle \pm0.02 $} & 0.98 {$\scriptstyle \pm0.06 $} & 0.91 {$\scriptstyle \pm0.13 $} & 0.82 {$\scriptstyle \pm0.21 $} & 0.78 {$\scriptstyle \pm0.23 $} & 0.74 {$\scriptstyle \pm0.26 $}\\
					& H-LWCD  & 0.99 {$\scriptstyle \pm0.03 $} & 0.97 {$\scriptstyle \pm0.07 $} & 0.90 {$\scriptstyle \pm0.13 $} & 0.80 {$\scriptstyle \pm0.21 $} & 0.77 {$\scriptstyle \pm0.23 $} & 0.72 {$\scriptstyle \pm0.25 $}\\
					& H-UBI  & 1.00 {$\scriptstyle \pm0.00 $} & 1.00 {$\scriptstyle \pm0.00 $} & 0.99 {$\scriptstyle \pm0.00 $} & 0.98 {$\scriptstyle \pm0.07 $} & 0.95 {$\scriptstyle \pm0.13 $} & 0.92 {$\scriptstyle \pm0.17 $}\\
					& H-Manopt  & 1.00 {$\scriptstyle \pm0.00 $} & 0.99 {$\scriptstyle \pm0.01 $} & 0.93 {$\scriptstyle \pm0.12 $} & 0.81 {$\scriptstyle \pm0.21 $} & 0.76 {$\scriptstyle \pm0.22 $} & 0.72 {$\scriptstyle \pm0.25 $}\\
					\midrule
					\multirow{2}*{\parbox{1.3cm}{$n=500$}} & 
					R-FAQ  & 0.99 {$\scriptstyle \pm0.03 $} & 0.97 {$\scriptstyle \pm0.07 $} & 0.90 {$\scriptstyle \pm0.13 $} & 0.80 {$\scriptstyle \pm0.21 $} & 0.76 {$\scriptstyle \pm0.23 $} & 0.72 {$\scriptstyle \pm0.25 $}\\
					& R-LWCD  & 0.98 {$\scriptstyle \pm0.03 $} & 0.96 {$\scriptstyle \pm0.06 $} & 0.89 {$\scriptstyle \pm0.13 $} & 0.80 {$\scriptstyle \pm0.21 $} & 0.76 {$\scriptstyle \pm0.23 $} & 0.71 {$\scriptstyle \pm0.25 $}\\
					\bottomrule

				\end{tabular}
			\end{sc}
		\end{small}
	\end{center}
	\vskip -.1in
\end{table}

\subsubsection{Seriation with Duplications}
Tables~\ref{tb:SerDupliDense-full}, \ref{tb:SerDupliDenseNoise5} and \ref{tb:SerDupliDenseNoise10} display additional results of Seriation with Duplication (with the same scores as in Table~\ref{tb:SerDupliDense}) on dense matrices expanding thoses from \S\ref{ssec:exps-ser-dupli}.
Tables~\ref{tb:SerDupliSparseBD40-full}, \ref{tb:SerDupliSparseBD20-full} expand the results from \S\ref{ssec:exps-ser-dupli} on matrices in  $\cM_N(\delta,s)$.

\begin{table}[t]
	\caption{
		Results for Seriation with Duplications on dense, strong-R matrices (with several values of the parameter $\gamma$ and $N/n$), and no noise added.
	}
	\label{tb:SerDupliDense-full}
	\begin{center}
		\begin{small}
			\begin{sc}
				\begin{tabular}{llcccccc}
					\toprule
					$\gamma$ & $N/n$ & method & d2S & Huber (x1e-7) & meanDist & stdDist & Time (x1e-2s) \\ 
					\midrule
					\multirow{9}*{\parbox{1.cm}{$0.1$}}
					&
					\multirow{3}*{\parbox{1.3cm}{$1.33$}}
					& spectral & 0.03 {$\scriptstyle \pm0.00$} & 8.33 {$\scriptstyle \pm0.01$} & 3.0 {$\scriptstyle \pm0.7$} & 5.5 {$\scriptstyle \pm1.0$} & 5.14 {$\scriptstyle \pm1.36$} \\ 
					& & $\eta$-Spectr. & 0.03 {$\scriptstyle \pm0.00$} & 8.33 {$\scriptstyle \pm0.01$} & 3.0 {$\scriptstyle \pm0.7$} & 5.5 {$\scriptstyle \pm1.0$} & 5.75 {$\scriptstyle \pm1.41$} \\ 
					& & H-UBI & 0.02 {$\scriptstyle \pm0.00$} & 8.33 {$\scriptstyle \pm0.01$} & 2.8 {$\scriptstyle \pm0.7$} & 5.2 {$\scriptstyle \pm1.0$} & 6.37 {$\scriptstyle \pm1.60$} \\
					
					\cmidrule{2-8}
					&  \multirow{3}*{\parbox{1.3cm}{$2$}} 
					& spectral & 0.03 {$\scriptstyle \pm0.00$} & 8.37 {$\scriptstyle \pm0.01$} & 7.1 {$\scriptstyle \pm1.0$} & 7.5 {$\scriptstyle \pm1.0$} & 5.05 {$\scriptstyle \pm1.05$} \\ 
					& & $\eta$-Spectr. & 0.03 {$\scriptstyle \pm0.00$} & 8.37 {$\scriptstyle \pm0.01$} & 7.1 {$\scriptstyle \pm1.0$} & 7.6 {$\scriptstyle \pm1.0$} & 5.41 {$\scriptstyle \pm1.07$} \\ 
					& & H-UBI & 0.03 {$\scriptstyle \pm0.00$} & 8.37 {$\scriptstyle \pm0.01$} & 7.0 {$\scriptstyle \pm1.0$} & 7.5 {$\scriptstyle \pm0.9$} & 6.14 {$\scriptstyle \pm1.17$} \\ 
					
					\cmidrule{2-8}
					& \multirow{3}*{\parbox{1.3cm}{$4$}}
					& spectral & 0.02 {$\scriptstyle \pm0.00$} & 8.35 {$\scriptstyle \pm0.01$} & 12.8 {$\scriptstyle \pm2.2$} & 7.8 {$\scriptstyle \pm1.8$} & 5.41 {$\scriptstyle \pm2.03$} \\ 
					& & $\eta$-Spectr. & 0.02 {$\scriptstyle \pm0.00$} & 8.35 {$\scriptstyle \pm0.01$} & 12.9 {$\scriptstyle \pm2.3$} & 7.9 {$\scriptstyle \pm2.1$} & 5.86 {$\scriptstyle \pm2.16$} \\ 
					& & H-UBI & 0.03 {$\scriptstyle \pm0.00$} & 8.35 {$\scriptstyle \pm0.01$} & 13.1 {$\scriptstyle \pm1.4$} & 7.9 {$\scriptstyle \pm1.4$} & 7.05 {$\scriptstyle \pm2.36$} \\

					\midrule

					\multirow{9}*{\parbox{1.cm}{$0.5$}}
					&
					\multirow{3}*{\parbox{1.3cm}{$1.33$}}
					& spectral & 0.25 {$\scriptstyle \pm0.04$} & 1.36 {$\scriptstyle \pm0.03$} & 6.1 {$\scriptstyle \pm1.8$} & 7.9 {$\scriptstyle \pm1.6$} & 8.74 {$\scriptstyle \pm4.85$} \\ 
					& & $\eta$-Spectr. & 0.15 {$\scriptstyle \pm0.02$} & 1.30 {$\scriptstyle \pm0.01$} & 2.2 {$\scriptstyle \pm0.7$} & 3.7 {$\scriptstyle \pm1.1$} & 6.12 {$\scriptstyle \pm4.84$} \\ 
					& & H-UBI & 0.24 {$\scriptstyle \pm0.04$} & 1.35 {$\scriptstyle \pm0.03$} & 5.5 {$\scriptstyle \pm1.6$} & 7.3 {$\scriptstyle \pm1.4$} & 11.06 {$\scriptstyle \pm7.56$} \\ 
					
					\cmidrule{2-8}
					&  \multirow{3}*{\parbox{1.3cm}{$2$}} 
					& spectral & 0.27 {$\scriptstyle \pm0.02$} & 1.41 {$\scriptstyle \pm0.02$} & 9.5 {$\scriptstyle \pm1.6$} & 8.4 {$\scriptstyle \pm1.3$} & 7.47 {$\scriptstyle \pm3.20$} \\ 
					& & $\eta$-Spectr. & 0.22 {$\scriptstyle \pm0.02$} & 1.37 {$\scriptstyle \pm0.02$} & 6.6 {$\scriptstyle \pm1.5$} & 6.7 {$\scriptstyle \pm1.9$} & 7.89 {$\scriptstyle \pm3.89$} \\ 
					& & H-UBI & 0.26 {$\scriptstyle \pm0.02$} & 1.40 {$\scriptstyle \pm0.02$} & 9.0 {$\scriptstyle \pm1.5$} & 8.1 {$\scriptstyle \pm1.2$} & 10.09 {$\scriptstyle \pm4.90$} \\ 
					
					\cmidrule{2-8}
					& \multirow{3}*{\parbox{1.3cm}{$4$}}
					& spectral & 0.18 {$\scriptstyle \pm0.01$} & 1.35 {$\scriptstyle \pm0.01$} & 14.4 {$\scriptstyle \pm2.8$} & 8.7 {$\scriptstyle \pm2.7$} & 6.53 {$\scriptstyle \pm1.90$} \\ 
					& & $\eta$-Spectr. & 0.18 {$\scriptstyle \pm0.01$} & 1.35 {$\scriptstyle \pm0.01$} & 14.3 {$\scriptstyle \pm2.9$} & 8.9 {$\scriptstyle \pm2.9$} & 7.59 {$\scriptstyle \pm2.28$} \\ 
					& & H-UBI & 0.19 {$\scriptstyle \pm0.01$} & 1.35 {$\scriptstyle \pm0.01$} & 14.8 {$\scriptstyle \pm2.5$} & 8.8 {$\scriptstyle \pm2.1$} & 8.62 {$\scriptstyle \pm2.46$} \\ 
					
					\midrule
					
					\multirow{9}*{\parbox{1.cm}{$1$}}
					&
					\multirow{3}*{\parbox{1.3cm}{$1.33$}}
					& spectral & 0.61 {$\scriptstyle \pm0.02$} & 2.10 {$\scriptstyle \pm0.13$} & 15.2 {$\scriptstyle \pm2.4$} & 15.2 {$\scriptstyle \pm1.4$} & 9.04 {$\scriptstyle \pm8.61$} \\ 
					& & $\eta$-Spectr. & 0.30 {$\scriptstyle \pm0.06$} & 1.48 {$\scriptstyle \pm0.08$} & 2.2 {$\scriptstyle \pm1.4$} & 3.1 {$\scriptstyle \pm1.5$} & 15.35 {$\scriptstyle \pm7.54$} \\ 
					& & H-UBI & 0.30 {$\scriptstyle \pm0.12$} & 1.50 {$\scriptstyle \pm0.15$} & 2.4 {$\scriptstyle \pm2.1$} & 3.1 {$\scriptstyle \pm2.1$} & 26.12 {$\scriptstyle \pm2.96$} \\

					\cmidrule{2-8}
					&  \multirow{3}*{\parbox{1.3cm}{$2$}} 
					& spectral & 0.60 {$\scriptstyle \pm0.03$} & 2.46 {$\scriptstyle \pm0.11$} & 19.3 {$\scriptstyle \pm6.6$} & 12.6 {$\scriptstyle \pm4.9$} & 1.78 {$\scriptstyle \pm0.31$} \\ 
					& & $\eta$-Spectr. & 0.42 {$\scriptstyle \pm0.04$} & 1.91 {$\scriptstyle \pm0.13$} & 10.3 {$\scriptstyle \pm8.6$} & 9.8 {$\scriptstyle \pm6.4$} & 1.20 {$\scriptstyle \pm0.51$} \\ 
					& & H-UBI & 0.49 {$\scriptstyle \pm0.05$} & 2.06 {$\scriptstyle \pm0.14$} & 10.4 {$\scriptstyle \pm7.9$} & 8.5 {$\scriptstyle \pm6.0$} & 2.57 {$\scriptstyle \pm0.21$} \\

					\cmidrule{2-8}
					& \multirow{3}*{\parbox{1.3cm}{$4$}}
					& spectral & 0.37 {$\scriptstyle \pm0.02$} & 1.81 {$\scriptstyle \pm0.05$} & 19.3 {$\scriptstyle \pm4.7$} & 11.6 {$\scriptstyle \pm4.4$} & 1.96 {$\scriptstyle \pm0.50$} \\ 
					& & $\eta$-Spectr. & 0.34 {$\scriptstyle \pm0.01$} & 1.78 {$\scriptstyle \pm0.04$} & 20.0 {$\scriptstyle \pm6.9$} & 13.2 {$\scriptstyle \pm6.1$} & 1.00 {$\scriptstyle \pm0.34$} \\ 
					& & H-UBI & 0.36 {$\scriptstyle \pm0.01$} & 1.80 {$\scriptstyle \pm0.04$} & 18.9 {$\scriptstyle \pm5.2$} & 11.5 {$\scriptstyle \pm4.8$} & 2.25 {$\scriptstyle \pm0.65$} \\ 
					
					\bottomrule
				\end{tabular}
			\end{sc}
		\end{small}
	\end{center}
	\vskip -.1in
\end{table}

\begin{table}[t]
	\caption{
		Results for Seriation with Duplications on dense, strong-R matrices (with several values of the parameter $\gamma$ and $N/n$), and noiseProp=$5\%$.
	}
	\label{tb:SerDupliDenseNoise5}
	\begin{center}
		\begin{small}
			\begin{sc}
				\begin{tabular}{llcccccc}
					\toprule
					$\gamma$ & $N/n$ & method & d2S & Huber (x1e-7) & meanDist & stdDist & Time (x1e-2s) \\ 
					\midrule
					\multirow{9}*{\parbox{1.cm}{$0.1$}}
					&
					\multirow{3}*{\parbox{1.3cm}{$1.33$}}
					& spectral & 0.07 {$\scriptstyle \pm0.00$} & 8.36 {$\scriptstyle \pm0.01$} & 5.7 {$\scriptstyle \pm0.9$} & 7.2 {$\scriptstyle \pm1.3$} & 1.27 {$\scriptstyle \pm0.78$} \\ 
					& & $\eta$-Spectr. & 0.07 {$\scriptstyle \pm0.00$} & 8.36 {$\scriptstyle \pm0.01$} & 5.7 {$\scriptstyle \pm0.9$} & 7.2 {$\scriptstyle \pm1.2$} & 1.39 {$\scriptstyle \pm0.80$} \\ 
					& & H-UBI & 0.07 {$\scriptstyle \pm0.00$} & 8.35 {$\scriptstyle \pm0.02$} & 5.2 {$\scriptstyle \pm0.9$} & 6.4 {$\scriptstyle \pm1.4$} & 1.48 {$\scriptstyle \pm0.93$} \\

					\cmidrule{2-8}
					&  \multirow{3}*{\parbox{1.3cm}{$2$}} 
					& spectral & 0.07 {$\scriptstyle \pm0.00$} & 8.38 {$\scriptstyle \pm0.01$} & 8.5 {$\scriptstyle \pm0.8$} & 7.7 {$\scriptstyle \pm0.8$} & 6.62 {$\scriptstyle \pm4.69$} \\ 
					& & $\eta$-Spectr. & 0.07 {$\scriptstyle \pm0.00$} & 8.38 {$\scriptstyle \pm0.01$} & 8.5 {$\scriptstyle \pm0.8$} & 7.7 {$\scriptstyle \pm0.9$} & 7.62 {$\scriptstyle \pm5.05$} \\ 
					& & H-UBI & 0.07 {$\scriptstyle \pm0.00$} & 8.37 {$\scriptstyle \pm0.01$} & 8.4 {$\scriptstyle \pm0.8$} & 7.5 {$\scriptstyle \pm0.9$} & 8.75 {$\scriptstyle \pm6.02$} \\

					\cmidrule{2-8}
					& \multirow{3}*{\parbox{1.3cm}{$4$}}
					& spectral & 0.06 {$\scriptstyle \pm0.00$} & 8.35 {$\scriptstyle \pm0.01$} & 13.7 {$\scriptstyle \pm2.4$} & 7.9 {$\scriptstyle \pm2.7$} & 5.15 {$\scriptstyle \pm1.49$} \\ 
					& & $\eta$-Spectr. & 0.06 {$\scriptstyle \pm0.00$} & 8.35 {$\scriptstyle \pm0.01$} & 13.8 {$\scriptstyle \pm2.3$} & 8.0 {$\scriptstyle \pm2.7$} & 5.47 {$\scriptstyle \pm1.58$} \\ 
					& & H-UBI & 0.06 {$\scriptstyle \pm0.00$} & 8.35 {$\scriptstyle \pm0.01$} & 13.8 {$\scriptstyle \pm2.2$} & 7.9 {$\scriptstyle \pm2.7$} & 6.17 {$\scriptstyle \pm1.58$} \\

					\midrule

					\multirow{9}*{\parbox{1.cm}{$0.5$}}
					&
					\multirow{3}*{\parbox{1.3cm}{$1.33$}}
					& spectral & 0.27 {$\scriptstyle \pm0.04$} & 1.37 {$\scriptstyle \pm0.03$} & 6.7 {$\scriptstyle \pm1.8$} & 8.4 {$\scriptstyle \pm1.6$} & 1.60 {$\scriptstyle \pm0.58$} \\ 
					& & $\eta$-Spectr. & 0.17 {$\scriptstyle \pm0.02$} & 1.31 {$\scriptstyle \pm0.01$} & 2.6 {$\scriptstyle \pm0.7$} & 4.1 {$\scriptstyle \pm1.0$} & 1.61 {$\scriptstyle \pm0.78$} \\ 
					& & H-UBI & 0.25 {$\scriptstyle \pm0.03$} & 1.36 {$\scriptstyle \pm0.02$} & 5.6 {$\scriptstyle \pm1.5$} & 7.3 {$\scriptstyle \pm1.4$} & 2.01 {$\scriptstyle \pm0.74$} \\

					\cmidrule{2-8}
					&  \multirow{3}*{\parbox{1.3cm}{$2$}} 
					& spectral & 0.28 {$\scriptstyle \pm0.02$} & 1.41 {$\scriptstyle \pm0.02$} & 9.7 {$\scriptstyle \pm1.5$} & 8.5 {$\scriptstyle \pm1.2$} & 1.07 {$\scriptstyle \pm0.58$} \\ 
					& & $\eta$-Spectr. & 0.23 {$\scriptstyle \pm0.02$} & 1.37 {$\scriptstyle \pm0.02$} & 6.7 {$\scriptstyle \pm1.4$} & 6.6 {$\scriptstyle \pm1.9$} & 1.08 {$\scriptstyle \pm0.64$} \\ 
					& & H-UBI & 0.26 {$\scriptstyle \pm0.02$} & 1.40 {$\scriptstyle \pm0.02$} & 9.0 {$\scriptstyle \pm1.5$} & 8.1 {$\scriptstyle \pm1.3$} & 1.46 {$\scriptstyle \pm0.91$} \\

					\cmidrule{2-8}
					& \multirow{3}*{\parbox{1.3cm}{$4$}}
					& spectral & 0.19 {$\scriptstyle \pm0.01$} & 1.35 {$\scriptstyle \pm0.01$} & 14.4 {$\scriptstyle \pm2.4$} & 8.4 {$\scriptstyle \pm2.1$} & 6.21 {$\scriptstyle \pm1.85$} \\ 
					& & $\eta$-Spectr. & 0.19 {$\scriptstyle \pm0.01$} & 1.35 {$\scriptstyle \pm0.01$} & 14.2 {$\scriptstyle \pm2.8$} & 8.7 {$\scriptstyle \pm2.6$} & 6.72 {$\scriptstyle \pm1.72$} \\ 
					& & H-UBI & 0.19 {$\scriptstyle \pm0.01$} & 1.35 {$\scriptstyle \pm0.01$} & 14.8 {$\scriptstyle \pm2.6$} & 8.8 {$\scriptstyle \pm2.2$} & 7.86 {$\scriptstyle \pm1.89$} \\

					\midrule
					
					\multirow{9}*{\parbox{1.cm}{$1$}}
					&
					\multirow{3}*{\parbox{1.3cm}{$1.33$}}
					& spectral & 0.62 {$\scriptstyle \pm0.02$} & 2.10 {$\scriptstyle \pm0.13$} & 15.3 {$\scriptstyle \pm2.4$} & 15.3 {$\scriptstyle \pm1.3$} & 9.20 {$\scriptstyle \pm8.34$} \\ 
					& & $\eta$-Spectr. & 0.32 {$\scriptstyle \pm0.06$} & 1.49 {$\scriptstyle \pm0.08$} & 2.3 {$\scriptstyle \pm1.1$} & 3.2 {$\scriptstyle \pm1.3$} & 18.45 {$\scriptstyle \pm6.40$} \\ 
					& & H-UBI & 0.32 {$\scriptstyle \pm0.11$} & 1.50 {$\scriptstyle \pm0.16$} & 2.6 {$\scriptstyle \pm2.3$} & 3.2 {$\scriptstyle \pm2.3$} & 26.30 {$\scriptstyle \pm3.52$} \\

					\cmidrule{2-8}
					&  \multirow{3}*{\parbox{1.3cm}{$2$}} 
					& spectral & 0.61 {$\scriptstyle \pm0.03$} & 2.46 {$\scriptstyle \pm0.12$} & 19.4 {$\scriptstyle \pm6.7$} & 12.6 {$\scriptstyle \pm4.9$} & 2.13 {$\scriptstyle \pm0.66$} \\ 
					& & $\eta$-Spectr. & 0.42 {$\scriptstyle \pm0.04$} & 1.92 {$\scriptstyle \pm0.13$} & 10.6 {$\scriptstyle \pm8.8$} & 10.1 {$\scriptstyle \pm6.6$} & 1.52 {$\scriptstyle \pm0.75$} \\ 
					& & H-UBI & 0.49 {$\scriptstyle \pm0.05$} & 2.06 {$\scriptstyle \pm0.15$} & 10.3 {$\scriptstyle \pm7.9$} & 8.4 {$\scriptstyle \pm6.1$} & 3.43 {$\scriptstyle \pm0.80$} \\

					\cmidrule{2-8}
					& \multirow{3}*{\parbox{1.3cm}{$4$}}
					& spectral & 0.37 {$\scriptstyle \pm0.02$} & 1.80 {$\scriptstyle \pm0.05$} & 19.0 {$\scriptstyle \pm4.9$} & 11.2 {$\scriptstyle \pm4.6$} & 1.44 {$\scriptstyle \pm0.30$} \\ 
					& & $\eta$-Spectr. & 0.35 {$\scriptstyle \pm0.01$} & 1.79 {$\scriptstyle \pm0.04$} & 20.0 {$\scriptstyle \pm6.8$} & 13.2 {$\scriptstyle \pm6.0$} & 0.77 {$\scriptstyle \pm0.17$} \\ 
					& & H-UBI & 0.37 {$\scriptstyle \pm0.02$} & 1.80 {$\scriptstyle \pm0.05$} & 18.9 {$\scriptstyle \pm4.9$} & 11.3 {$\scriptstyle \pm4.4$} & 1.83 {$\scriptstyle \pm0.41$} \\

					\bottomrule
				\end{tabular}
			\end{sc}
		\end{small}
	\end{center}
\end{table}

\begin{table}[t]
	\caption{
		Results for Seriation with Duplications on dense, strong-R matrices (with several values of the parameter $\gamma$ and $N/n$), and noiseProp=$10\%$.
	}
	\label{tb:SerDupliDenseNoise10}
	\begin{center}
		\begin{small}
			\begin{sc}
				\begin{tabular}{llcccccc}
					\toprule
					$\gamma$ & $N/n$ & method & d2S & Huber (x1e-7) & meanDist & stdDist & Time (x1e-2s) \\ 
					\midrule
					\multirow{9}*{\parbox{1.cm}{$0.1$}}
					&
					\multirow{3}*{\parbox{1.3cm}{$1.33$}}
					& spectral & 0.13 {$\scriptstyle \pm 0.00 $} & 8.36 {$\scriptstyle \pm 0.01 $} & 8.0 {$\scriptstyle \pm 0.7 $} & 8.0 {$\scriptstyle \pm 1.0 $} & 1.26 {$\scriptstyle \pm 0.74 $} \\ 
					& & $\eta$-Spectr. & 0.13 {$\scriptstyle \pm 0.00 $} & 8.36 {$\scriptstyle \pm 0.01 $} & 7.9 {$\scriptstyle \pm 0.7 $} & 7.9 {$\scriptstyle \pm 1.1 $} & 1.23 {$\scriptstyle \pm 0.83 $} \\ 
					& & H-UBI & 0.13 {$\scriptstyle \pm 0.00 $} & 8.35 {$\scriptstyle \pm 0.02 $} & 7.1 {$\scriptstyle \pm 0.7 $} & 6.5 {$\scriptstyle \pm 1.0 $} & 1.43 {$\scriptstyle \pm 0.87 $} \\

					\cmidrule{2-8}
					&  \multirow{3}*{\parbox{1.3cm}{$2$}} 
					& spectral & 0.12 {$\scriptstyle \pm 0.00 $} & 8.38 {$\scriptstyle \pm 0.01 $} & 11.1 {$\scriptstyle \pm 0.9 $} & 8.4 {$\scriptstyle \pm 0.9 $} & 6.44 {$\scriptstyle \pm 4.31 $} \\ 
					& & $\eta$-Spectr. & 0.12 {$\scriptstyle \pm 0.00 $} & 8.38 {$\scriptstyle \pm 0.01 $} & 11.0 {$\scriptstyle \pm 0.8 $} & 8.4 {$\scriptstyle \pm 0.9 $} & 7.08 {$\scriptstyle \pm 4.86 $} \\ 
					& & H-UBI & 0.12 {$\scriptstyle \pm 0.00 $} & 8.38 {$\scriptstyle \pm 0.01 $} & 10.8 {$\scriptstyle \pm 0.9 $} & 8.2 {$\scriptstyle \pm 1.0 $} & 8.49 {$\scriptstyle \pm 5.50 $} \\

					\cmidrule{2-8}
					& \multirow{3}*{\parbox{1.3cm}{$4$}}
					& spectral & 0.11 {$\scriptstyle \pm 0.00 $} & 8.35 {$\scriptstyle \pm 0.01 $} & 15.6 {$\scriptstyle \pm 2.8 $} & 8.2 {$\scriptstyle \pm 2.9 $} & 5.54 {$\scriptstyle \pm 1.55 $} \\ 
					& & $\eta$-Spectr. & 0.11 {$\scriptstyle \pm 0.00 $} & 8.35 {$\scriptstyle \pm 0.01 $} & 15.5 {$\scriptstyle \pm 2.5 $} & 8.3 {$\scriptstyle \pm 3.0 $} & 6.19 {$\scriptstyle \pm 2.23 $} \\ 
					& & H-UBI & 0.11 {$\scriptstyle \pm 0.00 $} & 8.35 {$\scriptstyle \pm 0.01 $} & 15.5 {$\scriptstyle \pm 2.0 $} & 8.2 {$\scriptstyle \pm 1.8 $} & 6.98 {$\scriptstyle \pm 2.38 $} \\

					\midrule

					\multirow{9}*{\parbox{1.cm}{$0.5$}}
					&
					\multirow{3}*{\parbox{1.3cm}{$1.33$}}
					& spectral & 0.31 {$\scriptstyle \pm 0.03 $} & 1.38 {$\scriptstyle \pm 0.02 $} & 7.5 {$\scriptstyle \pm 1.6 $} & 9.0 {$\scriptstyle \pm 1.4 $} & 1.73 {$\scriptstyle \pm 0.50 $} \\ 
					& & $\eta$-Spectr. & 0.21 {$\scriptstyle \pm 0.02 $} & 1.31 {$\scriptstyle \pm 0.01 $} & 3.0 {$\scriptstyle \pm 0.6 $} & 4.3 {$\scriptstyle \pm 1.1 $} & 1.67 {$\scriptstyle \pm 0.79 $} \\ 
					& & H-UBI & 0.29 {$\scriptstyle \pm 0.03 $} & 1.36 {$\scriptstyle \pm 0.02 $} & 6.1 {$\scriptstyle \pm 1.6 $} & 7.5 {$\scriptstyle \pm 1.5 $} & 2.06 {$\scriptstyle \pm 0.81 $} \\

					\cmidrule{2-8}
					&  \multirow{3}*{\parbox{1.3cm}{$2$}} 
					& spectral & 0.29 {$\scriptstyle \pm 0.02 $} & 1.41 {$\scriptstyle \pm 0.02 $} & 9.8 {$\scriptstyle \pm 1.4 $} & 8.5 {$\scriptstyle \pm 1.2 $} & 1.08 {$\scriptstyle \pm 0.59 $} \\ 
					& & $\eta$-Spectr. & 0.25 {$\scriptstyle \pm 0.02 $} & 1.38 {$\scriptstyle \pm 0.02 $} & 7.0 {$\scriptstyle \pm 1.3 $} & 6.9 {$\scriptstyle \pm 1.9 $} & 0.90 {$\scriptstyle \pm 0.57 $} \\ 
					& & H-UBI & 0.28 {$\scriptstyle \pm 0.02 $} & 1.40 {$\scriptstyle \pm 0.02 $} & 9.3 {$\scriptstyle \pm 1.5 $} & 8.1 {$\scriptstyle \pm 1.3 $} & 1.25 {$\scriptstyle \pm 0.69 $} \\

					\cmidrule{2-8}
					& \multirow{3}*{\parbox{1.3cm}{$4$}}
					& spectral & 0.21 {$\scriptstyle \pm 0.01 $} & 1.35 {$\scriptstyle \pm 0.01 $} & 14.6 {$\scriptstyle \pm 2.6 $} & 8.5 {$\scriptstyle \pm 2.2 $} & 6.65 {$\scriptstyle \pm 2.23 $} \\ 
					& & $\eta$-Spectr. & 0.21 {$\scriptstyle \pm 0.01 $} & 1.35 {$\scriptstyle \pm 0.01 $} & 14.4 {$\scriptstyle \pm 3.3 $} & 8.8 {$\scriptstyle \pm 3.1 $} & 7.54 {$\scriptstyle \pm 2.72 $} \\ 
					& & H-UBI & 0.21 {$\scriptstyle \pm 0.01 $} & 1.35 {$\scriptstyle \pm 0.01 $} & 15.1 {$\scriptstyle \pm 2.5 $} & 8.8 {$\scriptstyle \pm 2.2 $} & 8.95 {$\scriptstyle \pm 3.53 $} \\

					\midrule
					
					\multirow{9}*{\parbox{1.cm}{$1$}}
					&
					\multirow{3}*{\parbox{1.3cm}{$1.33$}}
					& spectral & 0.64 {$\scriptstyle \pm 0.02 $} & 2.10 {$\scriptstyle \pm 0.13 $} & 15.4 {$\scriptstyle \pm 2.3 $} & 15.4 {$\scriptstyle \pm 1.3 $} & 8.93 {$\scriptstyle \pm 8.70 $} \\ 
					& & $\eta$-Spectr. & 0.35 {$\scriptstyle \pm 0.05 $} & 1.52 {$\scriptstyle \pm 0.07 $} & 2.5 {$\scriptstyle \pm 1.1 $} & 3.4 {$\scriptstyle \pm 1.3 $} & 20.46 {$\scriptstyle \pm 7.24 $} \\ 
					& & H-UBI & 0.36 {$\scriptstyle \pm 0.10 $} & 1.54 {$\scriptstyle \pm 0.16 $} & 2.9 {$\scriptstyle \pm 2.4 $} & 3.4 {$\scriptstyle \pm 2.4 $} & 29.20 {$\scriptstyle \pm 3.68 $} \\

					\cmidrule{2-8}
					&  \multirow{3}*{\parbox{1.3cm}{$2$}} 
					& spectral & 0.61 {$\scriptstyle \pm 0.03 $} & 2.46 {$\scriptstyle \pm 0.11 $} & 19.6 {$\scriptstyle \pm 6.6 $} & 12.9 {$\scriptstyle \pm 4.8 $} & 1.70 {$\scriptstyle \pm 0.36 $} \\ 
					& & $\eta$-Spectr. & 0.43 {$\scriptstyle \pm 0.04 $} & 1.92 {$\scriptstyle \pm 0.13 $} & 10.4 {$\scriptstyle \pm 8.6 $} & 9.9 {$\scriptstyle \pm 6.4 $} & 1.18 {$\scriptstyle \pm 0.54 $} \\ 
					& & H-UBI & 0.50 {$\scriptstyle \pm 0.04 $} & 2.07 {$\scriptstyle \pm 0.14 $} & 10.6 {$\scriptstyle \pm 7.8 $} & 8.7 {$\scriptstyle \pm 6.2 $} & 2.49 {$\scriptstyle \pm 0.24 $} \\

					\cmidrule{2-8}
					& \multirow{3}*{\parbox{1.3cm}{$4$}}
					& spectral & 0.38 {$\scriptstyle \pm 0.02 $} & 1.81 {$\scriptstyle \pm 0.05 $} & 19.7 {$\scriptstyle \pm 5.2 $} & 11.7 {$\scriptstyle \pm 5.0 $} & 1.59 {$\scriptstyle \pm 0.42 $} \\ 
					& & $\eta$-Spectr. & 0.36 {$\scriptstyle \pm 0.01 $} & 1.79 {$\scriptstyle \pm 0.04 $} & 20.0 {$\scriptstyle \pm 6.9 $} & 13.1 {$\scriptstyle \pm 6.0 $} & 0.87 {$\scriptstyle \pm 0.25 $} \\ 
					& & H-UBI & 0.38 {$\scriptstyle \pm 0.02 $} & 1.80 {$\scriptstyle \pm 0.05 $} & 19.5 {$\scriptstyle \pm 5.8 $} & 11.9 {$\scriptstyle \pm 5.4 $} & 1.85 {$\scriptstyle \pm 0.43 $} \\

					\bottomrule
				\end{tabular}
			\end{sc}
		\end{small}
	\end{center}
\end{table}

\pushpagenumber
\begin{table}[t]
	\caption{
		Results for Seriation with Duplications on sparse, strong-R matrices (with several values of the parameter $s/s_{\text{lim}}$ and $N/n$), and $\delta=n/5$.
	}
	\label{tb:SerDupliSparseBD40-full}
	\begin{center}
		\begin{small}
			\begin{sc}
				\begin{tabular}{llcccccc}
					\toprule
					$s/s_{\text{lim}}$ & $N/n$ & method & d2S & Huber (x1e-7) & meanDist & stdDist & Time (x1e-2s) \\ 
					\midrule
					\multirow{9}*{\parbox{1.cm}{$0$}}
					&
					\multirow{3}*{\parbox{1.3cm}{$1.33$}}
					& spectral & 0.53 {$\scriptstyle \pm 0.08 $} & 1.67 {$\scriptstyle \pm 0.33 $} & 11.8 {$\scriptstyle \pm 3.5 $} & 13.2 {$\scriptstyle \pm 1.7 $} & 7.45 {$\scriptstyle \pm 4.08 $} \\ 
					& & $\eta$-Spectr. & 0.12 {$\scriptstyle \pm 0.06 $} & 0.76 {$\scriptstyle \pm 0.06 $} & 0.8 {$\scriptstyle \pm 0.8 $} & 2.4 {$\scriptstyle \pm 2.2 $} & 2.85 {$\scriptstyle \pm 1.78 $} \\ 
					& & H-UBI & 0.09 {$\scriptstyle \pm 0.06 $} & 0.74 {$\scriptstyle \pm 0.05 $} & 0.6 {$\scriptstyle \pm 0.6 $} & 1.8 {$\scriptstyle \pm 1.9 $} & 3.99 {$\scriptstyle \pm 2.76 $} \\

					\cmidrule{2-8}
					&  \multirow{3}*{\parbox{1.3cm}{$2$}} 
					& spectral & 0.38 {$\scriptstyle \pm 0.05 $} & 1.48 {$\scriptstyle \pm 0.26 $} & 10.3 {$\scriptstyle \pm 4.2 $} & 10.5 {$\scriptstyle \pm 2.8 $} & 1.30 {$\scriptstyle \pm 0.25 $} \\ 
					& & $\eta$-Spectr. & 0.21 {$\scriptstyle \pm 0.04 $} & 0.99 {$\scriptstyle \pm 0.12 $} & 4.1 {$\scriptstyle \pm 4.1 $} & 6.9 {$\scriptstyle \pm 3.9 $} & 0.50 {$\scriptstyle \pm 0.19 $} \\ 
					& & H-UBI & 0.19 {$\scriptstyle \pm 0.05 $} & 0.96 {$\scriptstyle \pm 0.14 $} & 4.0 {$\scriptstyle \pm 5.8 $} & 6.2 {$\scriptstyle \pm 4.6 $} & 0.79 {$\scriptstyle \pm 0.31 $} \\

					\cmidrule{2-8}
					& \multirow{3}*{\parbox{1.3cm}{$4$}}
					& spectral & 0.29 {$\scriptstyle \pm 0.02 $} & 1.45 {$\scriptstyle \pm 0.09 $} & 18.4 {$\scriptstyle \pm 4.5 $} & 11.8 {$\scriptstyle \pm 3.1 $} & 1.34 {$\scriptstyle \pm 0.23 $} \\ 
					& & $\eta$-Spectr. & 0.22 {$\scriptstyle \pm 0.02 $} & 1.29 {$\scriptstyle \pm 0.06 $} & 16.3 {$\scriptstyle \pm 6.8 $} & 12.2 {$\scriptstyle \pm 5.1 $} & 0.61 {$\scriptstyle \pm 0.14 $} \\ 
					& & H-UBI & 0.22 {$\scriptstyle \pm 0.02 $} & 1.26 {$\scriptstyle \pm 0.06 $} & 15.9 {$\scriptstyle \pm 7.2 $} & 12.0 {$\scriptstyle \pm 5.6 $} & 0.91 {$\scriptstyle \pm 0.25 $} \\

					\midrule

					\multirow{9}*{\parbox{1.cm}{$0.5$}}
					&
					\multirow{3}*{\parbox{1.3cm}{$1.33$}}
					& spectral & 0.52 {$\scriptstyle \pm 0.08 $} & 1.68 {$\scriptstyle \pm 0.33 $} & 11.1 {$\scriptstyle \pm 3.5 $} & 12.9 {$\scriptstyle \pm 1.8 $} & 8.79 {$\scriptstyle \pm 3.83 $} \\ 
					& & $\eta$-Spectr. & 0.21 {$\scriptstyle \pm 0.03 $} & 0.87 {$\scriptstyle \pm 0.06 $} & 1.3 {$\scriptstyle \pm 0.7 $} & 2.6 {$\scriptstyle \pm 2.0 $} & 4.15 {$\scriptstyle \pm 3.10 $} \\ 
					& & H-UBI & 0.19 {$\scriptstyle \pm 0.02 $} & 0.85 {$\scriptstyle \pm 0.04 $} & 0.9 {$\scriptstyle \pm 0.5 $} & 1.8 {$\scriptstyle \pm 1.5 $} & 5.95 {$\scriptstyle \pm 4.06 $} \\

					\cmidrule{2-8}
					&  \multirow{3}*{\parbox{1.3cm}{$2$}} 
					& spectral & 0.40 {$\scriptstyle \pm 0.04 $} & 1.55 {$\scriptstyle \pm 0.23 $} & 10.3 {$\scriptstyle \pm 3.8 $} & 10.5 {$\scriptstyle \pm 2.6 $} & 1.33 {$\scriptstyle \pm 0.24 $} \\ 
					& & $\eta$-Spectr. & 0.24 {$\scriptstyle \pm 0.03 $} & 1.07 {$\scriptstyle \pm 0.11 $} & 4.3 {$\scriptstyle \pm 4.0 $} & 7.0 {$\scriptstyle \pm 3.9 $} & 0.55 {$\scriptstyle \pm 0.22 $} \\ 
					& & H-UBI & 0.23 {$\scriptstyle \pm 0.05 $} & 1.06 {$\scriptstyle \pm 0.16 $} & 4.6 {$\scriptstyle \pm 7.0 $} & 6.4 {$\scriptstyle \pm 5.1 $} & 0.76 {$\scriptstyle \pm 0.33 $} \\

					\cmidrule{2-8}
					& \multirow{3}*{\parbox{1.3cm}{$4$}}
					& spectral & 0.30 {$\scriptstyle \pm 0.03 $} & 1.50 {$\scriptstyle \pm 0.09 $} & 19.0 {$\scriptstyle \pm 5.1 $} & 12.1 {$\scriptstyle \pm 3.5 $} & 1.35 {$\scriptstyle \pm 0.19 $} \\ 
					& & $\eta$-Spectr. & 0.24 {$\scriptstyle \pm 0.02 $} & 1.34 {$\scriptstyle \pm 0.06 $} & 16.3 {$\scriptstyle \pm 7.1 $} & 12.0 {$\scriptstyle \pm 5.1 $} & 0.65 {$\scriptstyle \pm 0.17 $} \\ 
					& & H-UBI & 0.24 {$\scriptstyle \pm 0.02 $} & 1.31 {$\scriptstyle \pm 0.06 $} & 15.8 {$\scriptstyle \pm 7.1 $} & 11.8 {$\scriptstyle \pm 5.6 $} & 0.97 {$\scriptstyle \pm 0.26 $} \\

					\midrule
					
					\multirow{9}*{\parbox{1.cm}{$1$}}
					&
					\multirow{3}*{\parbox{1.3cm}{$1.33$}}
					& spectral & 0.51 {$\scriptstyle \pm 0.07 $} & 1.65 {$\scriptstyle \pm 0.30 $} & 9.9 {$\scriptstyle \pm 3.1 $} & 12.4 {$\scriptstyle \pm 1.8 $} & 1.03 {$\scriptstyle \pm 0.28 $} \\ 
					& & $\eta$-Spectr. & 0.26 {$\scriptstyle \pm 0.02 $} & 0.95 {$\scriptstyle \pm 0.05 $} & 1.5 {$\scriptstyle \pm 0.6 $} & 2.7 {$\scriptstyle \pm 1.9 $} & 0.41 {$\scriptstyle \pm 0.33 $} \\ 
					& & H-UBI & 0.25 {$\scriptstyle \pm 0.02 $} & 0.94 {$\scriptstyle \pm 0.04 $} & 1.2 {$\scriptstyle \pm 0.5 $} & 2.1 {$\scriptstyle \pm 1.7 $} & 0.59 {$\scriptstyle \pm 0.72 $} \\

					\cmidrule{2-8}
					&  \multirow{3}*{\parbox{1.3cm}{$2$}} 
					& spectral & 0.39 {$\scriptstyle \pm 0.04 $} & 1.51 {$\scriptstyle \pm 0.18 $} & 9.2 {$\scriptstyle \pm 3.9 $} & 10.0 {$\scriptstyle \pm 2.7 $} & 1.24 {$\scriptstyle \pm 0.25 $} \\ 
					& & $\eta$-Spectr. & 0.27 {$\scriptstyle \pm 0.04 $} & 1.13 {$\scriptstyle \pm 0.13 $} & 4.5 {$\scriptstyle \pm 5.2 $} & 6.9 {$\scriptstyle \pm 4.4 $} & 0.55 {$\scriptstyle \pm 0.23 $} \\ 
					& & H-UBI & 0.26 {$\scriptstyle \pm 0.04 $} & 1.11 {$\scriptstyle \pm 0.14 $} & 4.3 {$\scriptstyle \pm 6.3 $} & 6.3 {$\scriptstyle \pm 4.7 $} & 0.80 {$\scriptstyle \pm 0.36 $} \\

					\cmidrule{2-8}
					& \multirow{3}*{\parbox{1.3cm}{$4$}}
					& spectral & 0.30 {$\scriptstyle \pm 0.02 $} & 1.50 {$\scriptstyle \pm 0.09 $} & 18.7 {$\scriptstyle \pm 5.0 $} & 12.1 {$\scriptstyle \pm 3.3 $} & 1.29 {$\scriptstyle \pm 0.18 $} \\ 
					& & $\eta$-Spectr. & 0.25 {$\scriptstyle \pm 0.02 $} & 1.37 {$\scriptstyle \pm 0.06 $} & 16.5 {$\scriptstyle \pm 7.2 $} & 12.1 {$\scriptstyle \pm 5.2 $} & 0.64 {$\scriptstyle \pm 0.20 $} \\ 
					& & H-UBI & 0.25 {$\scriptstyle \pm 0.01 $} & 1.34 {$\scriptstyle \pm 0.06 $} & 15.4 {$\scriptstyle \pm 6.6 $} & 11.4 {$\scriptstyle \pm 4.9 $} & 0.91 {$\scriptstyle \pm 0.27 $} \\

					\midrule
					
					\multirow{9}*{\parbox{1.cm}{$2.5$}}
					&
					\multirow{3}*{\parbox{1.3cm}{$1.33$}}
					& spectral & 0.51 {$\scriptstyle \pm 0.05 $} & 1.71 {$\scriptstyle \pm 0.23 $} & 8.1 {$\scriptstyle \pm 2.4 $} & 11.1 {$\scriptstyle \pm 2.0 $} & 1.79 {$\scriptstyle \pm 1.50 $} \\ 
					& & $\eta$-Spectr. & 0.35 {$\scriptstyle \pm 0.01 $} & 1.25 {$\scriptstyle \pm 0.04 $} & 1.9 {$\scriptstyle \pm 0.4 $} & 2.7 {$\scriptstyle \pm 1.4 $} & 0.90 {$\scriptstyle \pm 1.39 $} \\ 
					& & H-UBI & 0.35 {$\scriptstyle \pm 0.01 $} & 1.24 {$\scriptstyle \pm 0.04 $} & 1.8 {$\scriptstyle \pm 0.4 $} & 2.4 {$\scriptstyle \pm 1.3 $} & 1.20 {$\scriptstyle \pm 1.48 $} \\

					\cmidrule{2-8}
					&  \multirow{3}*{\parbox{1.3cm}{$2$}} 
					& spectral & 0.43 {$\scriptstyle \pm 0.03 $} & 1.69 {$\scriptstyle \pm 0.13 $} & 9.3 {$\scriptstyle \pm 4.5 $} & 10.2 {$\scriptstyle \pm 3.4 $} & 1.24 {$\scriptstyle \pm 0.22 $} \\ 
					& & $\eta$-Spectr. & 0.34 {$\scriptstyle \pm 0.03 $} & 1.39 {$\scriptstyle \pm 0.13 $} & 5.1 {$\scriptstyle \pm 6.3 $} & 7.0 {$\scriptstyle \pm 4.8 $} & 0.49 {$\scriptstyle \pm 0.18 $} \\ 
					& & H-UBI & 0.34 {$\scriptstyle \pm 0.04 $} & 1.38 {$\scriptstyle \pm 0.15 $} & 5.1 {$\scriptstyle \pm 7.0 $} & 6.3 {$\scriptstyle \pm 5.1 $} & 0.75 {$\scriptstyle \pm 0.30 $} \\

					\cmidrule{2-8}
					& \multirow{3}*{\parbox{1.3cm}{$4$}}	
					& spectral & 0.36 {$\scriptstyle \pm 0.02 $} & 1.64 {$\scriptstyle \pm 0.07 $} & 19.1 {$\scriptstyle \pm 5.3 $} & 12.1 {$\scriptstyle \pm 3.7 $} & 1.30 {$\scriptstyle \pm 0.20 $} \\ 
					& & $\eta$-Spectr. & 0.32 {$\scriptstyle \pm 0.01 $} & 1.52 {$\scriptstyle \pm 0.06 $} & 16.6 {$\scriptstyle \pm 7.2 $} & 12.1 {$\scriptstyle \pm 5.3 $} & 0.64 {$\scriptstyle \pm 0.15 $} \\ 
					& & H-UBI & 0.32 {$\scriptstyle \pm 0.01 $} & 1.49 {$\scriptstyle \pm 0.05 $} & 15.6 {$\scriptstyle \pm 6.3 $} & 11.3 {$\scriptstyle \pm 4.6 $} & 0.97 {$\scriptstyle \pm 0.29 $} \\

					\midrule
					
					\multirow{9}*{\parbox{1.cm}{$5$}}
					&
					\multirow{3}*{\parbox{1.3cm}{$1.33$}}
					& spectral & 0.54 {$\scriptstyle \pm 0.02 $} & 2.01 {$\scriptstyle \pm 0.09 $} & 6.7 {$\scriptstyle \pm 1.0 $} & 9.0 {$\scriptstyle \pm 1.8 $} & 1.08 {$\scriptstyle \pm 0.14 $} \\ 
					& & $\eta$-Spectr. & 0.45 {$\scriptstyle \pm 0.01 $} & 1.77 {$\scriptstyle \pm 0.03 $} & 2.7 {$\scriptstyle \pm 0.3 $} & 3.0 {$\scriptstyle \pm 1.1 $} & 0.43 {$\scriptstyle \pm 0.34 $} \\ 
					& & H-UBI & 0.45 {$\scriptstyle \pm 0.01 $} & 1.77 {$\scriptstyle \pm 0.03 $} & 2.8 {$\scriptstyle \pm 0.3 $} & 3.1 {$\scriptstyle \pm 1.0 $} & 0.98 {$\scriptstyle \pm 0.54 $} \\

					\cmidrule{2-8}
					&  \multirow{3}*{\parbox{1.3cm}{$2$}} 
					& spectral & 0.49 {$\scriptstyle \pm 0.02 $} & 2.00 {$\scriptstyle \pm 0.10 $} & 9.1 {$\scriptstyle \pm 5.0 $} & 9.5 {$\scriptstyle \pm 3.5 $} & 1.21 {$\scriptstyle \pm 0.20 $} \\ 
					& & $\eta$-Spectr. & 0.43 {$\scriptstyle \pm 0.03 $} & 1.83 {$\scriptstyle \pm 0.11 $} & 5.5 {$\scriptstyle \pm 6.4 $} & 6.4 {$\scriptstyle \pm 4.7 $} & 0.45 {$\scriptstyle \pm 0.14 $} \\ 
					& & H-UBI & 0.43 {$\scriptstyle \pm 0.03 $} & 1.83 {$\scriptstyle \pm 0.11 $} & 5.5 {$\scriptstyle \pm 6.2 $} & 6.3 {$\scriptstyle \pm 4.5 $} & 0.89 {$\scriptstyle \pm 0.42 $} \\

					\cmidrule{2-8}
					& \multirow{3}*{\parbox{1.3cm}{$4$}}
					& spectral & 0.45 {$\scriptstyle \pm 0.01 $} & 1.83 {$\scriptstyle \pm 0.07 $} & 19.7 {$\scriptstyle \pm 5.3 $} & 12.3 {$\scriptstyle \pm 3.9 $} & 1.25 {$\scriptstyle \pm 0.22 $} \\ 
					& & $\eta$-Spectr. & 0.43 {$\scriptstyle \pm 0.01 $} & 1.76 {$\scriptstyle \pm 0.06 $} & 17.5 {$\scriptstyle \pm 7.1 $} & 11.8 {$\scriptstyle \pm 4.9 $} & 0.61 {$\scriptstyle \pm 0.16 $} \\ 
					& & H-UBI & 0.43 {$\scriptstyle \pm 0.01 $} & 1.74 {$\scriptstyle \pm 0.05 $} & 16.5 {$\scriptstyle \pm 5.9 $} & 11.2 {$\scriptstyle \pm 4.5 $} & 0.87 {$\scriptstyle \pm 0.29 $} \\

					\bottomrule
				\end{tabular}
			\end{sc}
		\end{small}
	\end{center}
	\vskip -.1in
\end{table}

\begin{table}[t]
	\caption{
		Results for Seriation with Duplications on sparse, strong-R matrices (with several values of the parameter $s/s_{\text{lim}}$ and $N/n$), and $\delta=n/10$.
	}
	\label{tb:SerDupliSparseBD20-full}
	\begin{center}
		\begin{small}
			\begin{sc}
				\begin{tabular}{llcccccc}
					\toprule
					$s/s_{\text{lim}}$ & $N/n$ & method & d2S & Huber (x1e-7) & meanDist & stdDist & Time (x1e-2s) \\ 
					\midrule
					\multirow{9}*{\parbox{1.cm}{$0$}}
					&
					\multirow{3}*{\parbox{1.3cm}{$1.33$}}
					& spectral & 0.85 {$\scriptstyle \pm 0.04 $} & 6.42 {$\scriptstyle \pm 0.63 $} & 29.1 {$\scriptstyle \pm 14.3 $} & 23.4 {$\scriptstyle \pm 8.1 $} & 2.13 {$\scriptstyle \pm 3.72 $} \\ 
					& & $\eta$-Spectr. & 0.28 {$\scriptstyle \pm 0.17 $} & 1.86 {$\scriptstyle \pm 1.13 $} & 5.4 {$\scriptstyle \pm 12.2 $} & 6.6 {$\scriptstyle \pm 8.5 $} & 5.67 {$\scriptstyle \pm 3.65 $} \\ 
					& & H-UBI & 0.29 {$\scriptstyle \pm 0.22 $} & 2.09 {$\scriptstyle \pm 1.66 $} & 8.5 {$\scriptstyle \pm 17.2 $} & 8.4 {$\scriptstyle \pm 11.7 $} & 10.01 {$\scriptstyle \pm 5.36 $} \\

					\cmidrule{2-8}
					&  \multirow{3}*{\parbox{1.3cm}{$2$}} 
					& spectral & 0.87 {$\scriptstyle \pm 0.02 $} & 1.01 {$\scriptstyle \pm 0.06 $} & 44.7 {$\scriptstyle \pm 11.9 $} & 26.7 {$\scriptstyle \pm 7.4 $} & 9.01 {$\scriptstyle \pm 13.31 $} \\ 
					& & $\eta$-Spectr. & 0.49 {$\scriptstyle \pm 0.10 $} & 0.43 {$\scriptstyle \pm 0.11 $} & 26.3 {$\scriptstyle \pm 17.2 $} & 21.1 {$\scriptstyle \pm 10.8 $} & 85.15 {$\scriptstyle \pm 27.35 $} \\ 
					& & H-UBI & 0.53 {$\scriptstyle \pm 0.14 $} & 0.52 {$\scriptstyle \pm 0.21 $} & 28.9 {$\scriptstyle \pm 17.9 $} & 22.3 {$\scriptstyle \pm 11.8 $} & 176.26 {$\scriptstyle \pm 39.43 $} \\

					\cmidrule{2-8}
					& \multirow{3}*{\parbox{1.3cm}{$4$}}
					& spectral & 0.78 {$\scriptstyle \pm 0.05 $} & 1.19 {$\scriptstyle \pm 0.10 $} & 47.5 {$\scriptstyle \pm 7.7 $} & 21.3 {$\scriptstyle \pm 5.1 $} & 1.04 {$\scriptstyle \pm 0.62 $} \\ 
					& & $\eta$-Spectr. & 0.39 {$\scriptstyle \pm 0.02 $} & 0.44 {$\scriptstyle \pm 0.02 $} & 29.6 {$\scriptstyle \pm 7.2 $} & 18.0 {$\scriptstyle \pm 5.4 $} & 0.60 {$\scriptstyle \pm 0.13 $} \\ 
					& & H-UBI & 0.50 {$\scriptstyle \pm 0.17 $} & 0.65 {$\scriptstyle \pm 0.33 $} & 33.1 {$\scriptstyle \pm 10.6 $} & 18.6 {$\scriptstyle \pm 6.0 $} & 1.76 {$\scriptstyle \pm 0.40 $} \\

					\midrule

					\multirow{9}*{\parbox{1.cm}{$0.5$}}
					&
					\multirow{3}*{\parbox{1.3cm}{$1.33$}}
					& spectral & 0.86 {$\scriptstyle \pm 0.04 $} & 6.90 {$\scriptstyle \pm 0.63 $} & 29.6 {$\scriptstyle \pm 14.6 $} & 23.7 {$\scriptstyle \pm 8.3 $} & 2.07 {$\scriptstyle \pm 3.74 $} \\ 
					& & $\eta$-Spectr. & 0.37 {$\scriptstyle \pm 0.13 $} & 2.38 {$\scriptstyle \pm 1.09 $} & 6.0 {$\scriptstyle \pm 12.8 $} & 7.2 {$\scriptstyle \pm 9.1 $} & 6.62 {$\scriptstyle \pm 3.44 $} \\ 
					& & H-UBI & 0.37 {$\scriptstyle \pm 0.17 $} & 2.46 {$\scriptstyle \pm 1.51 $} & 7.6 {$\scriptstyle \pm 16.0 $} & 7.4 {$\scriptstyle \pm 11.0 $} & 10.88 {$\scriptstyle \pm 4.48 $} \\

					\cmidrule{2-8}
					&  \multirow{3}*{\parbox{1.3cm}{$2$}} 
					& spectral & 0.87 {$\scriptstyle \pm 0.01 $} & 1.05 {$\scriptstyle \pm 0.06 $} & 45.1 {$\scriptstyle \pm 12.4 $} & 27.0 {$\scriptstyle \pm 7.4 $} & 8.42 {$\scriptstyle \pm 3.70 $} \\ 
					& & $\eta$-Spectr. & 0.51 {$\scriptstyle \pm 0.09 $} & 0.47 {$\scriptstyle \pm 0.10 $} & 27.1 {$\scriptstyle \pm 17.4 $} & 21.8 {$\scriptstyle \pm 11.4 $} & 89.50 {$\scriptstyle \pm 28.54 $} \\ 
					& & H-UBI & 0.56 {$\scriptstyle \pm 0.13 $} & 0.58 {$\scriptstyle \pm 0.21 $} & 29.5 {$\scriptstyle \pm 18.4 $} & 22.5 {$\scriptstyle \pm 12.0 $} & 175.28 {$\scriptstyle \pm 48.61 $} \\

					\cmidrule{2-8}
					& \multirow{3}*{\parbox{1.3cm}{$4$}}
					& spectral & 0.78 {$\scriptstyle \pm 0.05 $} & 1.23 {$\scriptstyle \pm 0.11 $} & 47.1 {$\scriptstyle \pm 7.8 $} & 20.7 {$\scriptstyle \pm 5.2 $} & 1.08 {$\scriptstyle \pm 0.60 $} \\ 
					& & $\eta$-Spectr. & 0.40 {$\scriptstyle \pm 0.02 $} & 0.46 {$\scriptstyle \pm 0.02 $} & 29.9 {$\scriptstyle \pm 7.1 $} & 18.6 {$\scriptstyle \pm 5.5 $} & 0.62 {$\scriptstyle \pm 0.15 $} \\ 
					& & H-UBI & 0.49 {$\scriptstyle \pm 0.16 $} & 0.64 {$\scriptstyle \pm 0.32 $} & 31.8 {$\scriptstyle \pm 9.8 $} & 18.2 {$\scriptstyle \pm 6.1 $} & 1.78 {$\scriptstyle \pm 0.42 $} \\

					\midrule
					
					\multirow{9}*{\parbox{1.cm}{$1$}}
					&
					\multirow{3}*{\parbox{1.3cm}{$1.33$}}
					& spectral & 0.88 {$\scriptstyle \pm 0.03 $} & 7.67 {$\scriptstyle \pm 0.69 $} & 29.4 {$\scriptstyle \pm 14.3 $} & 23.5 {$\scriptstyle \pm 8.2 $} & 1.57 {$\scriptstyle \pm 3.20 $} \\ 
					& & $\eta$-Spectr. & 0.42 {$\scriptstyle \pm 0.11 $} & 2.79 {$\scriptstyle \pm 1.26 $} & 5.7 {$\scriptstyle \pm 12.5 $} & 6.6 {$\scriptstyle \pm 8.9 $} & 6.34 {$\scriptstyle \pm 3.79 $} \\ 
					& & H-UBI & 0.41 {$\scriptstyle \pm 0.14 $} & 2.81 {$\scriptstyle \pm 1.45 $} & 6.4 {$\scriptstyle \pm 14.6 $} & 6.5 {$\scriptstyle \pm 10.0 $} & 10.22 {$\scriptstyle \pm 4.59 $} \\

					\cmidrule{2-8}
					&  \multirow{3}*{\parbox{1.3cm}{$2$}} 
					& spectral & 0.87 {$\scriptstyle \pm 0.01 $} & 1.14 {$\scriptstyle \pm 0.06 $} & 44.7 {$\scriptstyle \pm 12.2 $} & 26.7 {$\scriptstyle \pm 7.1 $} & 1.53 {$\scriptstyle \pm 2.76 $} \\ 
					& & $\eta$-Spectr. & 0.51 {$\scriptstyle \pm 0.08 $} & 0.51 {$\scriptstyle \pm 0.12 $} & 26.1 {$\scriptstyle \pm 17.7 $} & 21.1 {$\scriptstyle \pm 11.7 $} & 8.06 {$\scriptstyle \pm 2.78 $} \\ 
					& & H-UBI & 0.58 {$\scriptstyle \pm 0.13 $} & 0.64 {$\scriptstyle \pm 0.23 $} & 29.0 {$\scriptstyle \pm 18.4 $} & 21.6 {$\scriptstyle \pm 11.9 $} & 17.74 {$\scriptstyle \pm 4.40 $} \\

					\cmidrule{2-8}
					& \multirow{3}*{\parbox{1.3cm}{$4$}}
					& spectral & 0.75 {$\scriptstyle \pm 0.06 $} & 1.26 {$\scriptstyle \pm 0.14 $} & 44.6 {$\scriptstyle \pm 7.7 $} & 20.5 {$\scriptstyle \pm 5.2 $} & 1.21 {$\scriptstyle \pm 0.55 $} \\ 
					& & $\eta$-Spectr. & 0.40 {$\scriptstyle \pm 0.01 $} & 0.48 {$\scriptstyle \pm 0.02 $} & 29.4 {$\scriptstyle \pm 7.0 $} & 18.3 {$\scriptstyle \pm 6.2 $} & 0.63 {$\scriptstyle \pm 0.16 $} \\ 
					& & H-UBI & 0.42 {$\scriptstyle \pm 0.08 $} & 0.51 {$\scriptstyle \pm 0.18 $} & 28.8 {$\scriptstyle \pm 8.6 $} & 18.0 {$\scriptstyle \pm 6.4 $} & 1.59 {$\scriptstyle \pm 0.38 $} \\

					\midrule
					
					\multirow{9}*{\parbox{1.cm}{$2.5$}}
					&
					\multirow{3}*{\parbox{1.3cm}{$1.33$}}
					& spectral & 0.90 {$\scriptstyle \pm 0.03 $} & 9.46 {$\scriptstyle \pm 0.74 $} & 30.2 {$\scriptstyle \pm 14.5 $} & 23.8 {$\scriptstyle \pm 8.6 $} & 1.76 {$\scriptstyle \pm 3.33 $} \\ 
					& & $\eta$-Spectr. & 0.51 {$\scriptstyle \pm 0.05 $} & 4.19 {$\scriptstyle \pm 0.66 $} & 3.9 {$\scriptstyle \pm 8.3 $} & 5.1 {$\scriptstyle \pm 6.2 $} & 6.31 {$\scriptstyle \pm 3.76 $} \\ 
					& & H-UBI & 0.54 {$\scriptstyle \pm 0.11 $} & 4.58 {$\scriptstyle \pm 1.53 $} & 9.4 {$\scriptstyle \pm 17.5 $} & 8.5 {$\scriptstyle \pm 12.5 $} & 11.94 {$\scriptstyle \pm 4.72 $} \\

					\cmidrule{2-8}
					&  \multirow{3}*{\parbox{1.3cm}{$2$}} 
					& spectral & 0.88 {$\scriptstyle \pm 0.01 $} & 1.33 {$\scriptstyle \pm 0.06 $} & 44.8 {$\scriptstyle \pm 12.2 $} & 26.9 {$\scriptstyle \pm 7.1 $} & 2.28 {$\scriptstyle \pm 3.51 $} \\ 
					& & $\eta$-Spectr. & 0.55 {$\scriptstyle \pm 0.05 $} & 0.63 {$\scriptstyle \pm 0.10 $} & 26.0 {$\scriptstyle \pm 17.3 $} & 21.1 {$\scriptstyle \pm 11.5 $} & 7.16 {$\scriptstyle \pm 2.69 $} \\ 
					& & H-UBI & 0.61 {$\scriptstyle \pm 0.11 $} & 0.75 {$\scriptstyle \pm 0.25 $} & 28.0 {$\scriptstyle \pm 18.9 $} & 21.2 {$\scriptstyle \pm 12.5 $} & 18.28 {$\scriptstyle \pm 4.10 $} \\

					\cmidrule{2-8}
					& \multirow{3}*{\parbox{1.3cm}{$4$}}	
					& spectral & 0.72 {$\scriptstyle \pm 0.06 $} & 1.31 {$\scriptstyle \pm 0.19 $} & 41.8 {$\scriptstyle \pm 8.7 $} & 20.6 {$\scriptstyle \pm 5.4 $} & 1.35 {$\scriptstyle \pm 0.41 $} \\ 
					& & $\eta$-Spectr. & 0.45 {$\scriptstyle \pm 0.01 $} & 0.55 {$\scriptstyle \pm 0.03 $} & 31.5 {$\scriptstyle \pm 7.1 $} & 19.1 {$\scriptstyle \pm 5.3 $} & 0.62 {$\scriptstyle \pm 0.16 $} \\ 
					& & H-UBI & 0.44 {$\scriptstyle \pm 0.01 $} & 0.53 {$\scriptstyle \pm 0.03 $} & 28.2 {$\scriptstyle \pm 7.9 $} & 17.8 {$\scriptstyle \pm 6.2 $} & 1.49 {$\scriptstyle \pm 0.38 $} \\

					\midrule
					
					\multirow{9}*{\parbox{1.cm}{$5$}}
					&
					\multirow{3}*{\parbox{1.3cm}{$1.33$}}
					& spectral & 0.93 {$\scriptstyle \pm 0.02 $} & 1.33 {$\scriptstyle \pm 0.09 $} & 31.4 {$\scriptstyle \pm 13.8 $} & 24.9 {$\scriptstyle \pm 8.8 $} & 2.60 {$\scriptstyle \pm 4.09 $} \\ 
					& & $\eta$-Spectr. & 0.64 {$\scriptstyle \pm 0.04 $} & 0.75 {$\scriptstyle \pm 0.07 $} & 6.5 {$\scriptstyle \pm 11.3 $} & 7.0 {$\scriptstyle \pm 8.9 $} & 6.73 {$\scriptstyle \pm 4.22 $} \\ 
					& & H-UBI & 0.68 {$\scriptstyle \pm 0.10 $} & 0.83 {$\scriptstyle \pm 0.19 $} & 12.4 {$\scriptstyle \pm 18.4 $} & 10.6 {$\scriptstyle \pm 13.2 $} & 16.54 {$\scriptstyle \pm 3.74 $} \\

					\cmidrule{2-8}
					&  \multirow{3}*{\parbox{1.3cm}{$2$}} 
					& spectral & 0.88 {$\scriptstyle \pm 0.01 $} & 1.73 {$\scriptstyle \pm 0.08 $} & 44.4 {$\scriptstyle \pm 11.6 $} & 27.0 {$\scriptstyle \pm 6.9 $} & 4.87 {$\scriptstyle \pm 5.68 $} \\ 
					& & $\eta$-Spectr. & 0.63 {$\scriptstyle \pm 0.03 $} & 0.87 {$\scriptstyle \pm 0.08 $} & 26.5 {$\scriptstyle \pm 16.0 $} & 21.3 {$\scriptstyle \pm 10.4 $} & 7.03 {$\scriptstyle \pm 2.69 $} \\ 
					& & H-UBI & 0.65 {$\scriptstyle \pm 0.06 $} & 0.92 {$\scriptstyle \pm 0.16 $} & 26.0 {$\scriptstyle \pm 17.9 $} & 20.0 {$\scriptstyle \pm 12.1 $} & 19.58 {$\scriptstyle \pm 2.81 $} \\

					\cmidrule{2-8}
					& \multirow{3}*{\parbox{1.3cm}{$4$}}
					& spectral & 0.66 {$\scriptstyle \pm 0.04 $} & 1.23 {$\scriptstyle \pm 0.22 $} & 35.4 {$\scriptstyle \pm 7.3 $} & 18.5 {$\scriptstyle \pm 5.3 $} & 1.42 {$\scriptstyle \pm 0.19 $} \\ 
					& & $\eta$-Spectr. & 0.57 {$\scriptstyle \pm 0.01 $} & 0.68 {$\scriptstyle \pm 0.03 $} & 30.9 {$\scriptstyle \pm 5.6 $} & 19.0 {$\scriptstyle \pm 4.9 $} & 0.58 {$\scriptstyle \pm 0.14 $} \\ 
					& & H-UBI & 0.56 {$\scriptstyle \pm 0.01 $} & 0.66 {$\scriptstyle \pm 0.03 $} & 28.6 {$\scriptstyle \pm 7.1 $} & 17.2 {$\scriptstyle \pm 5.8 $} & 0.94 {$\scriptstyle \pm 0.29 $} \\

					\bottomrule
				\end{tabular}
			\end{sc}
		\end{small}
	\end{center}
	\vskip -.1in
\end{table}

\end{document}